\documentclass[11pt,leqno]{amsart}
\usepackage{hyperref}
\usepackage{amssymb, amsmath, amsthm, amsfonts}
\usepackage{verbatim}
\usepackage{bbm}
\usepackage{enumerate}
\usepackage{dsfont}
\usepackage[mathscr]{eucal}
\usepackage{tikz}
\usepackage[normalem]{ulem}

\DeclareFontFamily{U}{mathx}{\hyphenchar\font45}
\DeclareFontShape{U}{mathx}{m}{n}{
<5> <6> <7> <8> <9> <10>
<10.95> <12> <14.4> <17.28> <20.74> <24.88>
mathx10
}{}
\DeclareSymbolFont{mathx}{U}{mathx}{m}{n}
\DeclareFontSubstitution{U}{mathx}{m}{n}
\DeclareMathAccent{\widecheck}{0}{mathx}{"71}
\DeclareMathAccent{\wideparen}{0}{mathx}{"75}

\usetikzlibrary{patterns}

\usepackage{todonotes}
\usepackage{comment}
\usepackage{color}
\definecolor{gray}{gray}{0.5}

\usepackage{tikz}
\newcommand{\trapez}{
\begin{tikzpicture} [scale=.23]
\draw (0,0) -- (.5,1);
\draw (0,0) -- (2,0);
\draw (1.5,1) -- (2,0);
\draw (.5,1) -- (1.5,1);
\end{tikzpicture}
}

\newcommand{\vertiii}[1]{{\left\vert\kern-0.25ex\left\vert\kern-0.25ex\left\vert #1
\right\vert\kern-0.25ex\right\vert\kern-0.25ex\right\vert}}

\def\emph#1{{\it #1 }}

\newcommand{\dilq}[2]{{#1}{#2}}

\newcommand{\tr}[1]{\dilq{3}{#1}}

\def\inn#1#2{\langle#1,#2\rangle}
\def\biginn#1#2{\big\langle#1,#2\big\rangle}
\def\Biginn#1#2{\Big\langle#1,#2\Big\rangle}
\def\jp#1{{\langle#1\rangle}}

\def\emph#1{{\it #1}}
\def\textbf#1{{\bf #1}}

\theoremstyle{plain}
\newtheorem{thm}{Theorem}[section]
\newtheorem{prop}[thm]{Proposition}

\newtheorem{lem}[thm]{Lemma}
\newtheorem{lemma}[thm]{Lemma}
\newtheorem{cor}[thm]{Corollary}
\newtheorem{definition}[thm]{Definition}

\newtheorem*{thm*}{Theorem}
\newtheorem*{conj*}{Conjecture}
\newtheorem*{openproblem*}{Open Problem}

\newtheorem{observation}[thm]{Observation}
\theoremstyle{remark}

\newtheorem*{remarksa}{Remarks}

\numberwithin{equation}{section}

\usepackage[utf8]{inputenc}

\usepackage[explicit,indentafter]{titlesec}
\titleformat{\section}{\centering\normalfont\scshape}{\thesection.}{.5em}{#1}
\titleformat{\subsection}[runin]{\normalfont\itshape}{\textnormal{\thesubsection.}}{.5em}{#1.}
\titleformat{\subsubsection}[runin]{\normalfont\itshape}{\thesubsubsection.}{.5em}{#1.}
\titlespacing{\section}{0em}{1em}{0.5em}
\titlespacing{\subsection}{0em}{.5em}{0.5em}

\subjclass[2020]{42B15, 42B20, 42B25}
\keywords{Sparse domination, Fourier multipliers, Endpoint estimates}

\begin{document}

\title[Endpoint sparse domination]
{Endpoint sparse domination for classes of multiplier transformations}
\author[D. Beltran \ \ \ \ \ \ \ J. Roos \ \ \ \ \ \ \ A. Seeger ] {David Beltran \ \ \ \ Joris Roos \ \ \ \ Andreas Seeger }

\address{David Beltran: Departament d'Anàlisi Matemàtica, Universitat de València, Dr. Moliner 50, 46100 Burjassot, Spain}

\email{david.beltran@uv.es}

\address{Joris Roos: Department of Mathematics and Statistics, University of Massachusetts Lowell, Lowell, MA 01854, USA}
\email{joris\_roos@uml.edu}

\address{Andreas Seeger: Department of Mathematics, University of Wisconsin-Madison, 480 Lincoln Dr, Madison, WI-53706, USA}
\email{seeger@math.wisc.edu}

\begin{abstract}
We prove endpoint results for sparse domination of translation invariant multiscale operators. The results are formulated in terms of dilation invariant classes of Fourier multipliers based on natural
localized $M^{p\to q}$ norms which express appropriate endpoint regularity hypotheses. The
applications include new and optimal sparse bounds for classical oscillatory multipliers and multi-scale versions of radial bump multipliers.
\end{abstract}


\maketitle
\section{Introduction}
The purpose of this paper is to prove new endpoint bounds in multiscale sparse domination for certain scale invariant classes of translation invariant operators.
Interesting partial endpoint sparse bounds are known in some cases (see for example \cite{conde-alonso-etal,KeslerLacey}), but they seem to be generally missing in situations where the sharp $L^p$-bounds rely on Hardy-space or BMO techniques.
Model cases for these situations are given by oscillatory Fourier multipliers, for which we obtain optimal endpoint sparse bounds in Theorem \ref{thm:oscmult} below, and multi-scale extensions of radial $\delta$-bumps (see Theorem \ref{cor:BR}).

\subsection{Background and definitions} We begin by reviewing some definitions (see the introduction of \cite{BRS} for more details). Fix a lattice ${\mathfrak {Q}}$ of dyadic cubes in the sense of Lerner and Nazarov \cite[\S 2]{lerner-nazarov}; this implies, in particular, that the dyadic cubes at a fixed scale are half-open pairwise disjoint cubes, and that every compact set is contained in some $Q\in {\mathfrak {Q}}$. For $f\in L^1_{\mathrm{loc}}$, $Q\in {\mathfrak {Q}}$ and $1 \leq p < \infty$, we set
$\jp{f}_{Q,p}=(|Q|^{-1}\int_Q|f(y)|^p\, \mathrm{d} y)^{1/p}$.
Given $0<\gamma<1$ a collection ${\mathfrak {S}}\in {\mathfrak {Q}}$ is called $\gamma$-sparse if for every $Q\in {\mathfrak {S}}$ there is a measurable subset $E_Q\subset Q$ so that $|E_Q|\geq \gamma|Q|$ and the sets $E_Q$ with $Q\in {\mathfrak {S}}$ are pairwise disjoint. Given a $\gamma$-sparse family ${\mathfrak {S}}$ of dyadic cubes, and $1\le p_1,p_2<\infty$ the corresponding sparse form is defined by
\[\Lambda_{p_1,p_2}^{\mathfrak {S}} (f_1,f_2) = \sum_{Q\in {\mathfrak {S}}}|Q| \jp{f_1}_{Q,p_1}
\jp{f_2}_{Q,p_2};\]
this is interesting in the range $p_2<p_1'$.
The maximal $(p_1,p_2)$-form $\Lambda^*_{p_1,p_2} $ is given by
\begin{equation} \label{eq:maxLambda} \Lambda^*_{p_1,p_2} (f_1,f_2)=\sup_{{\mathfrak {S}}:\gamma\text{-sparse}}\Lambda^{\mathfrak {S}}_{p_1,p_2} (f_1,f_2),\end{equation}
where the supremum is taken over all $\gamma$-sparse families of dyadic cubes in ${\mathfrak {Q}}$. We say that a linear operator $T: C^\infty_c({\mathbb {R}}^d) \to {\mathcal {D}}'({\mathbb {R}}^d)$ belongs to the space $\text{Sp}_\gamma(p_1,p_2)$ (or satisfies a $(p_1,p_2)$ sparse bound) if for all $f_1,\,f_2\in C^\infty_c$ the inequality
\begin{equation}\label{eq:sparse-dom} |\inn{Tf_1}{f_2}|\le C \Lambda^*_{p_1,p_2}(f_1,f_2) \end{equation} holds with some constant $C$ independent of $f_1$ and $f_2$, and we denote by $\|T\|_{{\mathrm{Sp}}_\gamma(p_1,p_2)}$ the best constant in this inequality.
The norm $\|T\|_{{\mathrm{Sp}}_\gamma(p_1,p_2)}$ depends on $\gamma$, but the space ${\mathrm{Sp}}_\gamma(p_1,p_2)$ does not. As we keep $\gamma$ fixed throughout this paper we will drop the subscript $\gamma$ when using the ${\mathrm{Sp}}_\gamma(p_1,p_2)$ norm.
As mentioned above,
the relevant case for applications is $p_2<p_1'$
(and indeed if $T$ is a convolution operator with compactly supported kernel, a $(p_1,p_1')$ sparse bound follows immediately from the $L^{p_1}$ boundedness of $T$). We remark that when
$p_2<p_1'$ we can change the a priori assumption of $f_1, f_2\in C^\infty_c$ to $f_1,f_2\in {\mathbb {V}}$ where ${\mathbb {V}}$ is any subspace dense in $L^p$ for some $p\in (p_1,p_2')$; for example, it is natural to choose ${\mathbb {V}}=L^{p_1}\cap L^{p_2'}$
(see \cite[Lemma A.1]{BRS}).

The interest in (a pointwise/normed version of) sparse domination started because of its important consequences in weighted inequalities for Calder\'on--Zygmund operators \cite{LeCZ, LeA2, CAR2014,lerner-nazarov,Lac2015,LeNew,benea-bernicot2018,DiPlinioHytonenLi}.
For consequences of the bilinear sparse domination \eqref{eq:sparse-dom} in weighted theory we refer to the paper by Bernicot, Frey and Petermichl \cite{bernicot-frey-petermichl}.
A
detailed exposition of the importance of sparse domination in harmonic analysis can be found in the introduction of \cite{Pereyra2019}; for many further examples beyond Calderón--Zygmund theory see \cite{laceyJdA19, CDPV,BRS}
and references therein.

In this paper we shall consider operators that commute with translations. They are defined as a Fourier multiplier transformation $m(D)$ where $\widehat {m(D) f}(\xi) =m(\xi) \widehat f(\xi)$. Here we work with
$\widehat f(\xi)=\int f(y) e^{i\inn y\xi} \, \mathrm{d} y$ as the definition of the Fourier transform of $f\in{\mathcal {S}}({\mathbb {R}}^d)$ and denote by $\widecheck f\equiv {\mathcal {F}}^{-1}[f]$ the inverse Fourier transform.
For $1\le p \leq q <\infty$ we denote by $M^{p\to q}$ the class of Fourier multipliers for which
$m(D)$ is bounded as an operator from $L^p$ to $L^q$; the norm in $M^{p\to q}$ is just given by the $L^p\to L^q$ operator norm of $m(D)$. A modification is needed for $p=\infty$; then $L^\infty$ is replaced by $C_0$.

\subsection{Scale invariant classes of multipliers: the main results} We shall now formulate our three main theorems on sparse domination involving scale invariant classes of multipliers and subsequently discuss new sharp results for oscillatory multipliers and multiscale radial bump multipliers.
Our conditions are motivated by $p$-sensitive endpoint multiplier theorems in \cite{SeegerMonatshefte, SeegerL12, SeegerStudia90} (see also an earlier result by Baernstein and Sawyer \cite{Baernstein-Sawyer} on $H^p$ multipliers for $p<1$).
The $M^p$ multiplier hypotheses
(in particular the one in \cite{SeegerStudia90})
can be seen as certain localized Besov-conditions where the Besov spaces are built on suitable Fourier multiplier spaces. Here we will formulate similar conditions which will be relevant for endpoint sparse bounds.

Let $\Phi_0\in C^\infty({\mathbb {R}}^d)$ be supported in $\{x \in {\mathbb {R}}^d:|x|<1/2\}$ such that $\Phi_0(x)=1$ for $|x|\le 1/4$. For $\ell\in {\mathbb {Z}}$ define
\begin{equation}
\label{eqn:defofPsiell}
\Psi_\ell(x)= \Phi_0(2^{-\ell} x)-\Phi_0(2^{-\ell+1}x)
\end{equation}
which is supported in $\{x \in {\mathbb {R}}^d:2^{\ell-3}\le |x|\le 2^{\ell-1}\}.$
For a Banach space ${\mathfrak {X}}$ of distributions (here suitable classes of multipliers)
let $B^\alpha_1({\mathfrak {X}})$ be the $B^\alpha_1$-Besov space built on ${\mathfrak {X}}$, with norm
\[\|h\|_{B^\alpha_1({\mathfrak {X}})} = \|h*\widehat \Phi_0\|_{\mathfrak {X}}+ \sum_{\ell> 0} 2^{\ell \alpha}
\|h*\widehat \Psi_\ell\|_{{\mathfrak {X}}}.
\]
The standard Besov-classes $B^\alpha_{u,1}$ can be recovered by taking ${\mathfrak {X}}=L^u({\mathbb {R}}^d)$; however for our results it is most appropriate to take for ${\mathfrak {X}}$ a multiplier space such as $M^{r\to q}$ for $r$ between $p$ and $q$.
Let $\phi$ be a radial $C^\infty$ function supported in $\{\xi\in \widehat {{\mathbb {R}}}^d: 1/2<|\xi|<2\}$ which is not identically zero.\footnote{The assumption that $\phi$ is radial is convenient but not crucial; one can show that one just needs to assume that for all rays emanating from the origin, the restriction of $\phi$ to the ray is not identically zero.}

It was proved
in \cite{SeegerStudia90} that
\begin{equation}\label{studia-estimate}
\|m\|_{M^{p\to p} } \le C_{p,r} \sup_{t>0} \|\phi m(t\cdot)\|_{B_1^{d(\frac 1p-\frac 1r)}(M^{r\to r} )}, \quad 1<p<r\le 2.
\end{equation}
Inequality \eqref{studia-estimate} is related to results in \cite{carbery-revista, See88} but the latter are not applicable to endpoint estimates in many situations; indeed, they do not give satisfactory results for the oscillatory
multipliers in \eqref{def:mab}
below.

We state now versions of these results for $(p,q')$ sparse domination, with three cases, depending on whether $q <2$, $q=2$ or $q>2$.
Multipliers for which the right-hand side of \eqref{studia-estimate} is finite belong at least to the class ${\mathrm{Sp}} (p,r') $; this follows from the special case $r=q$ in our first theorem.
\begin{thm}\label{thm:qle2}
Let $1 < p \leq q \leq 2$.
Then
for $p<r\le q$,
\[\|m(D)\|_{\mathrm{Sp}(p,q') } \le C_{p,r}
\sup_{t>0} \|\phi m(t\cdot) \|_{B^{d(\frac 1p-\frac 1q)}_1(M^{r\to q}) } \,.\]
\end{thm}

For $q=2$ we can also let $r=p$ to get a ${\mathrm{Sp}}(p,2)$ bound.
\begin{thm}\label{thm:qge2}
Let $1 < p< 2$.
Then
\[\|m(D)\|_{\mathrm{Sp}(p,2) } \le C_p
\sup_{t>0} \|\phi m(t\cdot) \|_{B^{d(\frac1p-\frac12)}_1(M^{p\to 2}) }
\,.\]
\end{thm}

For $q> 2$ we have the following version.

\begin{thm} \label{thm:multq>2}
Let $1 < p < 2<q\le p'$. Then for $q'<r\le 2$,
\begin{multline}\label{condmultq>2}
\|m(D)\|_{\mathrm{Sp}(p,q') } \le C_{p,q,r} \times\\
(\sup_{t>0} \|\phi m(t\cdot) \|_{B^{d(\frac1p-\frac1q)}_1(M^{p\to q}) }
+\sup_{t>0} \|\phi m(t\cdot) \|_{B^{d(\frac1{q'}-\frac1r)}_1(M^{r\to r}) })
\,.\end{multline}
\end{thm}

\begin{remarksa}
(i) The spaces of multipliers defined by the conditions in the above three theorems are independent of the choice of the radial non-trivial function $\phi$ and independent of the specific spatial cutoff
function.
This can be shown by routine but somewhat lengthy calculations. We omit the proof but point out that our choice for $\phi$ can always be taken as the cutoff function $\varphi$ used in the Calder\'on reproducing formula \eqref{eq:Caldrepr}.

(ii)
As noted in \cite[p.152]{SeegerMonatshefte} the multipliers in Theorem \ref{thm:qge2} satisfy for $1<p<2$ an inequality involving the Lorentz space $L^{p,2}$,
\begin{equation}\label{eq:Lp2} \notag \|m(D) f\|_{L^{p,2} } \lesssim_p A \|f\|_p, \quad \text{ with $A=\sup_{t>0} \|\phi m(t\cdot) \|_{B_1^{d(\frac 1p-\frac 12)}(M^{p\to 2})}$.}\end{equation}
This is shown to be a consequence of the weighted norm inequality
\begin{equation}\label{eq:Gfunctions} \notag \int {\mathcal {G}}[m(D) f]^2 w \, \mathrm{d} x \lesssim A^2
\int \widetilde {\mathcal {G}}[f] ^2 (M[|w|^s] )^{1/s} \, \mathrm{d} x,\quad s=(p'/2)',
\end{equation} where
$M$ is the Hardy-Littlewood maximal operator, and ${\mathcal {G}}$, $\widetilde {\mathcal {G}}$
are suitable Littlewood--Paley--Stein operators. For earlier closely related variants in non-endpoint cases see \cite[Ch.IV]{Ste70}, \cite{ChristProc}.

(iii) The proofs of Theorems \ref{thm:qle2}, \ref{thm:qge2}, and \ref{thm:multq>2} have a similar structure and, in order to avoid repetitions, we shall present them together. They rely on an iteration argument common in sparse domination; one main novelty in this paper is that at every step of the iteration Calder\'on--Zygmund arguments are combined with atomic decompositions in $L^p$-spaces of functions on certain dyadic cubes. This use of the iterated atomic decompositions is crucial for the proof of Theorem \ref{thm:qle2}, but can be replaced by applications of Littlewood--Paley theory in the proofs of Theorem \ref{thm:qge2} and \ref{thm:multq>2}.

\end{remarksa}

We can use the embedding $L^u\subset M^{p\to 2} $ for $1\le p\le 2$, $1/p-1/2=1/u$ to derive a corollary of Theorem \ref{thm:qge2} which uses standard Besov spaces
$B^{d(1/p-1/2)}_{u,1}=B^{d(1/p-1/2)}_1(L^u)$.

\begin{cor} \label{cor:generalBesov} Suppose $1< p< 2$, $1/p-1/2\ge 1/u$ and $m\in L^\infty$ satisfies
\[\sup_{t>0} \|\phi m(t\cdot) \|_{B^{d(\frac 1p-\frac 12)}_{u,1} }<\infty.\]
Then $m(D)\in {\mathrm{Sp}}(p,2)$.
\end{cor}

Another corollary of Theorem \ref{thm:qge2} involves radial multipliers where $L^u$ is replaced by $L^2$, as a consequence of the Stein--Tomas Fourier restriction theorem.
\begin{cor} \label{cor:studiaradial}
Let $d\ge 2$, $1< p\le \frac{2(d+1)}{d+3}$.
Suppose that \[\sup_{t>0} \|\phi h(t\cdot)
\|_{B_{2,1}^{ d(\frac 1p-\frac 12)} ({\mathbb {R}}) } <\infty.\]
Then $h(|D|)\in \mathrm{Sp}(p,2)$.
\end{cor}

\subsection{Oscillatory Fourier multipliers and Miyachi classes}

For $a>0$, $a\neq 1$ and $b\ge 0$ consider the Fourier multipliers
\begin{equation} \label{def:mab}
m_{a,b} (\xi) = e^{i|\xi|^a} |\xi|^{-b} \chi_\infty(\xi)
\end{equation} where $\chi_\infty \in C^\infty({\mathbb {R}}^d)$, $\chi_\infty(\xi)=1$ for $|\xi|\ge 1$ and $\chi_\infty$ vanishes in a neighborhood of the origin. It is well-known that the operator $m_{a,b} (D)$ is bounded on $L^p({\mathbb {R}}^d) $ for all $p\in (1,\infty)$ if and only if $b\ge ad/2$. Moreover if $0<b<ad/2$, $L^p$ boundedness holds if and only if $\frac{ 2ad}{ad+2b}\leq p \leq \frac{2ad}{ad-2b}$, i.e., equivalently, if $b\ge ad|\frac 1p-\frac 12|$ (see \cite{Stein-bull1971}, \cite{FeffermanStein1972}, \cite {Miyachi}).
Miyachi \cite{Miyachi}
considered classes generalizing the oscillatory multipliers $m_{a,b}$; for $0<a<1$ these correspond to
translation invariant versions of the pseudo-differential operators with $S^{-b}_{1-a,\delta}$ symbols for which Fefferman \cite{Fefferman-IsrealPs} had already proved sharp $L^p$ bounds. We say that $m\in \mathrm{FM}(a,b)$ if $m$ is supported in $\{\xi\in{\mathbb {R}}^d: |\xi|\ge 1\}$ and satisfies the derivative estimates
\begin{equation}\label{Miyachi-deriv-cond} |m^{(\beta)} (\xi) |\le C_\beta |\xi|^{(a-1)|\beta|-b} \end{equation} for all multiindices $\beta \in {\mathbb {N}}_0^d$. It is proved in \cite{Miyachi} that for $1<p\le 2$
the multiplier operators $m(D)$ with $m\in \mathrm{FM}(a, ad(\frac 1p-\frac 12))$ are bounded on $L^p$; this result is optimal.

In \cite{BRS}
it was shown that for $0<b < ad/2$, the operator
$m_{a,b}(D)$ belongs to ${\mathrm{Sp}}(p,p)$ in the open range for $\frac{ 2ad}{ad+2b}<p\le 2$.
For general multipliers in $\mathrm{FM}(a,b) $ it was shown in \cite{beltran-cladek} that the operators belong to ${\mathrm{Sp}}(p,2) $, in the same $p$-range.
We note that no nontrivial $(p_1,p_2)$ sparse bounds with $p_2<p_1'$ was obtained at the endpoint $p_1=\frac{2ad}{ad+2b} $, i.e. $b=ad(\frac 1{p_1}-\frac 12)$.

We provide
a full characterization of the sparse exponent set for the oscillatory multiplier operators $m_{a,b} (D)$ in the relevant parameter case $0<b<ad/2$, thereby settling the open endpoint problem.

\begin{thm} \label{thm:oscmult} Let $a\neq 1$ and $0<b<ad/2$.
Let $\triangle(a,b)$
the closed triangle with vertices
$Q_1=(\frac 12+\frac{b}{da}, \frac 12-\frac{b}{da})$,
$Q_2=(\frac 12-\frac{b}{da}, \frac 12+\frac{b}{da})$, $Q_3=(\frac 12+\frac{b}{da}, \frac 12+\frac{b}{da})$ and $m_{a,b} $ be the oscillatory multiplier in \eqref{def:mab}.
Then
\[ m_{a,b}(D) \in \mathrm{Sp}(p_1,p_2) \, \iff \, (\tfrac 1{p_1}, \tfrac 1{p_2}) \in \triangle(a,b).
\]
\end{thm}
In particular for the oscillatory multipliers we get the $\mathrm{Sp} (p_1,p_2)$ bound for the endpoint $p_1=\frac{2ad}{ad+2b}$ in the optimal range $p_1\le p_2\le p_1'$. We also have a sharp result that applies to the full class $\mathrm{FM}(a,b)$.

\begin{thm} \label{miyachi-thm}
Let $a\neq 1$ and $0<b<ad/2$.
Let $\trapez(a,b) $ be
the closed trapezoid with vertices
$Q_1=(\frac 12+\frac{b}{da}, \frac 12-\frac{b}{da})$, $Q_2=(\frac 12-\frac{b}{da}, \frac 12+\frac{b}{da})$,
$P_3=(\frac 12, \frac 12+\frac{b}{da})$, $P_4= (\frac 12+\frac{b}{da},\frac 12)$. Then
\[ m(D)\in{\mathrm{Sp}}(p_1,p_2)\,\text{for all}\;m\in\mathrm{FM}(a,b)\,\iff\, (\tfrac1{p_1},\tfrac1{p_2})\in\trapez(a,b).\]
\end{thm}

\begin{figure}

\begin{tikzpicture}[scale=2]

\begin{scope}[scale=1.8]

\draw[thick,->] (0,0) -- (1.1,0) node[below] {\small{$ \frac 1 p_1$}};
\draw[thick,->] (0,0) -- (0,1.1) node[left] {\small{$ \frac{1}{p_2}$}};

\draw[loosely dashed] (0,1) -- (1.,1.) -- (1.,0);
\draw[loosely dashed] (0,1) -- (1/6,5/6);
\draw[loosely dashed] (5/6,1/6) -- (1,0);
\draw (1/6,5/6) -- (0.5,0.5) -- (5/6,1/6);

\draw (.5,.02) -- (.5,-.02) node[below] {\small{$ \tfrac 12$}};
\draw (.02,.5) -- (-.02,.5) node[left] {\small{$ \tfrac 12$}};

\draw (5/6,.02) -- (5/6, -.02) node[below] {\small{$\tfrac{1}{2} + \tfrac{b}{ad}$}} ;

\draw[loosely dashed] (0.5,0.5) -- (1,1);

\draw (5/6,1/6) -- (5/6,1/2);
\draw (1/6,5/6) -- (1/2,5/6);
\draw (5/6,1/2) -- (1/2,5/6);
\fill[pattern=north west lines, pattern color=gray] (5/6,1/6) -- (5/6,1/2) -- (1/2, 5/6) -- (1/6,5/6) -- (5/6,1/6);

\begin{scope}[xshift=50]
\draw[thick,->] (0,0) -- (1.1,0) node[below] {\small{$ \frac 1 p_1$}};
\draw[thick,->] (0,0) -- (0,1.1) node[left] {\small{$ \frac{1}{p_2}$}};

\draw[loosely dashed] (0,1) -- (1.,1.) -- (1.,0);
\draw[loosely dashed] (0,1) -- (1/6,5/6);
\draw[loosely dashed] (5/6,1/6) -- (1,0);
\draw (1/6,5/6) -- (0.5,0.5) -- (5/6,1/6);

\draw (.5,.02) -- (.5,-.02) node[below] {\small{$ \tfrac 12$}};
\draw (.02,.5) -- (-.02,.5) node[left] {\small{$ \tfrac 12$}};

\draw (5/6,.02) -- (5/6, -.02) node[below] {\small{$\tfrac{1}{2} + \tfrac{b}{ad}$}} ;

\draw[loosely dashed] (0.5,0.5) -- (1,1);

\draw (5/6,1/6) -- (5/6,5/6)--(1/6,5/6);

\fill[pattern=north west lines, pattern color=gray] (5/6,1/6) -- (5/6,5/6) -- (1/6,5/6) -- (5/6,1/6);

\end{scope}
\end{scope}

\end{tikzpicture}

\caption{Sparse bounds for the general multiplier class $\mathrm{FM}(a,b)$ (left) and for the oscillatory multipliers $m_{a,b}$ (right) for given $a,b >0$ with $0<b<ad/2$.}
\label{fig:exponents}
\end{figure}
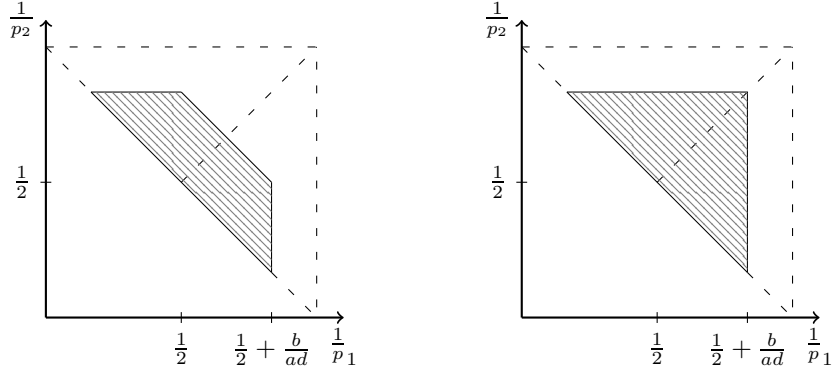

The results of Theorems \ref{thm:oscmult} and \ref{miyachi-thm} are illustrated in Figure \ref{fig:exponents}.
The positive results on the edges $({Q_1Q_3}]$ and $[{Q_3Q_2})$ of the triangle on the right are new for the oscillatory multipliers.
For the $\mathrm{FM}(a,b)$ class the positive results on the edges
$({Q_1P_3}]$, $[{P_3P_4} ]$, $[{P_4Q_2})$ of the trapezoid are new.

\begin{remarksa}
(i)
The general positive result about the multipliers in $\mathrm{FM}(a,b)$ can be derived from Corollary
\ref{cor:generalBesov}. Indeed, this corollary implies sharp results for the classes of subdyadic multipliers considered in
\cite{Beltran-Bennett}, see also a relevant discussion in \cite{BRS}.
For the extended region of the oscillatory multipliers we need to use Theorem \ref{thm:multq>2}.

(ii) The methods in this paper can also be used to strengthen results in \cite{beltran-cladek} on sparse bounds for pseudo-differential operators with symbols in the H\"ormander classes
$S^{\nu}_{\rho, \delta}$, for $0<\delta\le \rho <1$.
By \cite{Fefferman-IsrealPs, MichalowskiRuleStaubach} these operators are bounded on $L^p$ (here $1<p<\infty$) provided that $\nu\le -d(1-\rho)|\frac 1p-\frac 12|$.
In the range $-\frac d2(1-\rho)<\nu<0$
we now get the full endpoint sparse bounds, extending the results for the multiplier classes $\mathrm{FM}(1-\rho,-\nu)$,
that is, the operators belong to ${\mathrm{Sp}}(p_1,p_2)$ for $(1/p_1,1/p_2)\in\trapez(1-\rho, -\nu)$.

(iii) The multiplier class in Theorem \ref{thm:qge2} is also relevant
in the interesting recent work by
Bulj--Kova\v c \cite{BuljKovac22}
and by Stolyarov \cite{Stolyarov22} on lower bounds for other types of oscillatory multipliers; indeed the theorem allows to derive upper ${\mathrm{Sp}}(p,2)$ bounds in their setting.
\end{remarksa}

\subsection{Multiscale radial bump multipliers}
Let $\chi$ be a smooth bump function supported in $(-1/2,1/2)$ and set for small $\delta$
\begin{equation}\label{eq:defhdelta}h_\delta(t)=\chi(\delta^{-1} (1-|t|)).\end{equation} The multiplier $h_\delta(|\xi|)$
occurs naturally as a building block for the Bochner--Riesz multipliers. It is conjectured that in dimension $d\ge 2$ we have
\begin{equation}\label{BR-Lp}
\sup_{0<\delta<1/2} \delta^{ d(\frac 1{p}-\frac 12)-\frac 12}\|h_\delta(|\cdot|) \|_{M^{p\to p}} <\infty ,
\end{equation} for $1\le p<\frac{2d}{d+1}$. This conjecture is well known in two dimensions (the range is then $1\le p<4/3$, see \cite{CarlesonSjolin, FeffermanBR, CordobaBR, seeger-BRwt}), and there are partial results in higher dimensions.
More specifically, by Tao's arguments in \cite{Tao-Indiana1998} the bound \eqref{BR-Lp} for any fixed $p<\frac{2d}{d+1}$ follows from a slightly weaker bound with an additional factor of $c_\epsilon \delta^\epsilon$ for arbitrary $\epsilon>0$ on the left-hand side, and
such estimates with the $\epsilon$-loss have been verified on a partial range of $p$ (see \cite{GuoOhWangWuZhang} for the latest results and more references).

Here, we consider
the multiscale version
\begin{equation}\label{eq:BR-mult} m_\delta(\xi)= \sum_{k\in {\mathbb {Z}}} a_k h_\delta(2^k|\xi|).
\end{equation}
From \cite{SeegerStudia90} we know that if \eqref{BR-Lp} holds for some $p_\circ<\frac{2d}{d+1}$, then we have
\begin{equation}\label{BR-mult-Lp}
\sup_{0<\delta<1/2} \delta^{ d(\frac 1{p}-\frac 12)-\frac 12}\|m_\delta(D) \|_{L^p\to L^p} \lesssim_p \sup_{k\in \mathbb Z}|a_k|.
\end{equation}
for $1< p<p_\circ$.

Our purpose here is to illustrate how Theorems \ref{thm:qle2}, \ref{thm:qge2} and \ref{thm:multq>2} imply various sharp or essentially sharp sparse domination results in a $p$-range that will be optimal in two dimensions; in higher dimensions we limit ourselves to the range $1<p<\frac{2(d+2)}{d+4}$ (i.e. the range dual to Tao's bilinear Fourier extension theorem \cite{tao-bilinear}), as in this range the known sharp $L^p\to L^q$ estimates for $h_\delta(|D|)$ are well documented in the literature \cite{bak-negBR, Gutierrez2000, ChoKimLeeShim2005}.

\begin{thm} \label{cor:BR} Let $m_\delta$ be as in \eqref{eq:BR-mult} and let $d\ge 2$.
Then the inequality
\begin{equation} \label{eq:sparseBRineq} \sup_{0<\delta<1/2} \delta^{d(\frac{1}{p_1}-\frac 12)-\frac 12} \|m_\delta(D)\|_{{\mathrm{Sp}}(p_1,p_2)} \lesssim_{p_1,p_2} \sup_{k \in {\mathbb {Z}}}|a_k|
\end{equation} holds if
\begin{itemize}
\item[(a)] $1<p_1 \le
\frac{2(d+1)}{d+3}$ and $p_2\ge
\frac{(d-1)p_1}{d+1-2p_1} $, or

\item[(b)]
$
\frac{2(d+1)}{d+3} <p_1< \frac{2(d+2)}{d+4}$ and $p_2>\frac{(d-1)p_1}{d+1-2p_1} $.
\end{itemize}

\end{thm}

\begin{remarksa} (i) Using the building block with $a_0=1$ and $a_k=0$ for $k\neq 0$
one sees that this result is sharp in the sense that inequality \eqref{eq:sparseBRineq} fails in general
if
$1<p_1<\frac{2d}{d+1}$ and $p_2<
\frac{(d-1)p_1}{d+1-2p_1} $. This can be deduced directly from a corresponding result for Bochner--Riesz operators in \cite[\S5]{lacey-mena-reguera}.

(ii) For the proof of sufficiency in
Theorem \ref{cor:BR} we rely on Theorems \ref{thm:qge2} and \ref{thm:multq>2} in the range $p_1\le \frac {2(d+1)}{d+3}$ and on
Theorem \ref{thm:qle2} in the range $
\frac{2(d+1)}{d+3} <p_1< \frac{2(d+2)}{d+4}$.
\end{remarksa}
\subsection{ Necessary conditions for sparse domination of convolution operators}
Necessary conditions for general sparse operators were discussed in Chapter 2 of \cite{BRS}; here we point out that they can be put in a simple form for scalar operators.
This result will be convenient for checking the sharpness of several of the results mentioned above.

Let $T: C^\infty_c({\mathbb {R}}^d) \to {\mathcal {D}}'({\mathbb {R}}^d)$ with Schwartz kernel $K\in {\mathcal {D}}'({\mathbb {R}}^d\times {\mathbb {R}}^d)$.
Let $\Psi\in C^\infty_c({\mathbb {R}}^d)$ be supported in $\{x\in {\mathbb {R}}^d:1\le|x|\le 2\}$.
Define the distribution $K_R$ as the multiplication of $K$ with the $C^\infty$ function
$\Psi( R^{-1}(x-y) ) $ and let $T_R$ denote the linear operator with Schwartz kernel $K_R$.
Define the rescaled kernels
\begin{align}
\label{Kjresc}K_R^{\mathrm{resc}}(x,y)&:= R^d K_R(Rx, Ry) =\Psi(x-y) R^d K(Rx, Ry),
\end{align}
interpreted in the sense of distributions, and let
$T_R^{{\mathrm{resc}}}$ be the rescaled version of $T_R$, with convolution kernel $K_R^{\mathrm{resc}}$.

\begin{prop}\label{thm:nec-cond}
Let $T:C^\infty_c({\mathbb {R}}^d)\to {\mathcal {D}}'({\mathbb {R}}^d)$ be a continuous linear operator with Schwartz kernel $K$ and let $T_R^{\mathrm{resc}}$ be the rescaled version defined
in \eqref{Kjresc}. Suppose $1<p_1,p_2<\infty$ and $T\in {\mathrm{Sp}}(p_1,p_2)$, with $p_2<p_1'$. Then $T$ extends to a bounded operator $L^{p_1}\to L^{p_1,\infty}$ and $L^{p_2',1}\to L^{p_2'}$; moreover the operators $T_R^{\mathrm{resc}}$ map $L^{p_1}$ to $L^{p_2'}$ with uniform operator norm and
\[ \|T\|_{L^{p_1}\to L^{p_1,\infty} }
+\|T\|_{L^{p_2',1}\to L^{p_2'} }+
\sup_{R>0} \|T_R^{\mathrm{resc}}\|_{L^{p_1}\to L^{p_2'}} \lesssim \|T\|_{{\mathrm{Sp}}(p_1,p_2)}.\]
\end{prop}

The proof is based on more general results in \cite{BRS} and will be given in \S\ref{sec:necessary}.

\subsection*{Structure of the paper} In \S\ref{sec:induction scheme}, we present the induction scheme that proves the sparse domination Theorems \ref{thm:qle2}, \ref{thm:qge2} and \ref{thm:multq>2}. In \S\ref{sec:decompositions} we discuss the atomic decomposition, which is used in \S\ref{sec:base case} to verify the base case for the induction. In \S\ref{sec:CZ-dec} we present a Calder\'on--Zygmund decomposition based on the atomic decomposition. The plan for the proof of the induction step is outlined in \S\ref{sec:induction step}, with proofs presented in \S\ref{sec:goodpart} and \S\ref{sec:badpart}.
In \S\ref{sec:applications} we discuss the applications of the main theorems and, in particular, how they imply the positive results in Theorems \ref{thm:oscmult} and \ref{miyachi-thm}, and the positive results on radial multipliers in Corollary \ref{cor:studiaradial} and Theorem \ref{cor:BR}. Proposition \ref{thm:nec-cond} and the proof of the necessary conditions for Theorems \ref{thm:oscmult} and \ref{miyachi-thm} are presented in \S\ref{sec:necessary}. Finally, \S\ref{sec:appendix} contains the proofs of some technical facts which are included for the reader's convenience.

\subsection*{Acknowledgements}
\thanks{The authors would like to thank
the Hausdorff Research Institute of Mathematics and the organizers of the trimester program
{\it Interactions between Geometric Measure Theory, Singular Integrals, and PDE}
for a pleasant working environment during a visit in January 2022. The work was supported in part by National Science Foundation grants
DMS-1954479 (D.B.), DMS-2154835 (J.R.), DMS-2054220 (A.S.),
by a grant from the Simons Foundation (ID 855692, J.R.) and by the Agencia Estatal de Investigación through RYC2020-029151-I (D.B.).}

\section{Structure of the induction argument}\label{sec:induction scheme}

In this section we present the proof strategy for Theorems \ref{thm:qle2}, \ref{thm:qge2} and \ref{thm:multq>2}. We will see that sparse bounds for $m(D)$ can be deduced from sparse bounds for finite multi-scale sums of spatially (and frequency) localized pieces of $m(D)$. The proof of the latter is based on an induction on the number of pieces in the multi-scale sums, similarly to our previous work \cite{BRS}. As in \cite{BRS}, it is useful to work with a modified version of our maximal $(p_1,p_2)$-form $\Lambda^*_{p_1,p_2}$.

\begin{definition} Given a dyadic cube $S_0\in {\mathfrak {Q}}$
let
\[\Lambda^{**}_{S_0,p,q'}(f_1,f_2):= \sup\sum_{S \in {\mathfrak {S}}}|S|\jp{f_1}_{S,p} \jp{f_2}_{\tr{S},q'}\] where the supremum is taken over all $\gamma$-sparse
collections ${\mathfrak {S}}$ consisting of cubes in ${\mathfrak {Q}}(S_0)$, which denotes the subset of ${\mathfrak {Q}}$ of cubes contained in $S_0$.
\end{definition}

\subsection{Decomposition as a multi-scale sum}
Consider $\eta, \varphi \in \mathcal{S}$ facilitating the Calder\'on reproducing formula, i.e.,
\begin{subequations}\label{eq:Caldrepr}
\begin{align}\label{eq:suppwidechecketa}
&\text{$\widecheck{\eta}$ has compact support in $\{ |x| \leq 10^{-d}\}$ and $\eta(0)=0;$}
\\
\label{eq:suppvarphi}
&\text{$\varphi$ has compact support in $\{ 1/2< |\xi| < 2\}$;}
\\
&\label{eq:reproducing} \sum_{k \in {\mathbb {Z}}} \eta^2(2^{-k} \xi) \varphi(2^{-k} \xi) = 1, \quad\xi \neq 0.
\end{align}
\end{subequations}
For any $k \in {\mathbb {Z}}$, let $L_k$ and $P_k$ be defined by
\begin{align*}
\widehat{L_k f}(\xi) &= \varphi_k( \xi) \widehat{f}(\xi), \qquad \quad \text{where} \,\, \varphi_k(\xi):=\varphi(2^{-k}\xi), \\
\widehat{P_k f}(\xi) &= \eta_k( \xi) \widehat{f}(\xi), \qquad \quad \,\text{where} \,\, \eta_k(\xi):=\eta(2^{-k}\xi).
\end{align*}
Let $T_k=m(D) L_k$ and denote by $K_k$ its convolution kernel, that is $K_k=\mathcal{F}^{-1}[\varphi_k m]$. We next perform a spatial decomposition of $K_k$. Let $\Phi_0$ and $\Psi_j$
be as in \eqref{eqn:defofPsiell}.
Let
\begin{align*}
K_k^{(-k)}(x) &= {\mathcal {F}}^{-1} [\varphi_k m] (x) \Phi_0(2^{k} x)
\\
K_k^{(j)} (x)&={\mathcal {F}}^{-1} [\varphi_k m] (x) \Psi_{j} (x) \quad \text{ if $j>-k$}
\end{align*}
so that we get $K_k= \sum_{j=-k}^\infty K_k^{(j)} = K_k^{(-k)}+\sum_{\ell>0} K_k^{(\ell-k)}$.
Let $T_k^{(\ell-k)}$ denote the operator with convolution kernel $K_k^{(\ell-k)}$ for $\ell\geq 0$. Note that by \eqref{eq:reproducing} we have
\[
m(D)=\sum_{k \in {\mathbb {Z}}} T_k P_k P_k = \sum_{k \in {\mathbb {Z}}} \sum_{\ell \geq 0} T_k^{(\ell-k)} P_k P_k.
\]

It is also convenient to introduce some notation for the operator norm of $T_k^{(\ell-k)}$.
We first note that
\begin{equation}\label{eq:kernels in the scaled multiplier form}
\widehat{{K}_k^{(\ell-k)}} (2^k\xi)
= (\varphi m(2^k \cdot))\ast \widehat{\Psi_\ell}(\xi)
\end{equation}
and define the quantities
\begin{equation}\label{Aprq} A_{p,r,q}^{k,\ell}[m]\equiv A_{p,r,q}^{k,\ell}=
\begin{cases}
\| \varphi m(2^k \cdot) \ast \widehat{\Phi_0} \|_{M^{r\to q}}, &\text{ if }\ell=0,
\\
\| \varphi m(2^k \cdot) \ast \widehat{\Psi_\ell} \|_{M^{r\to q}}2^{\ell d (\frac{1}{p}-\frac{1}{q})}, &\text{ if }\ell>0,
\end{cases}
\end{equation}
and
\begin{equation}
\mathcal{A}_{p,r,q}[m] = \sup_{k \in {\mathbb {Z}}} \sum_{\ell\ge 0} A_{p,r,q}^{k,\ell}[m],
\end{equation}
assuming that $p<q$ and $p\le r\le q$. Then
\begin{equation}\label{eq:Besovmultnorm}
\mathcal{A}_{p,q,r}[m] \le
\sup_{t>0} \|\varphi m(t\cdot) \|_{ B_1^{d(\frac 1p-\frac 1q)}(M^{r\to q})}.
\end{equation}
Moreover we have for all $k\in {\mathbb {Z}}$,
\begin{equation} \label{eq:A-embeddings}
\sum_{\ell\ge 0} A_{p,r_1,q}^{k,\ell} \lesssim \sum_{\ell\ge 0} A_{p,r_2,q}^{k,\ell} , \quad p\le r_1\le r_2\le q.
\end{equation}
This inequality is an immediate consequence of a slightly stronger statement, Corollary \ref{cor:Aembeddings}.

The sparse bounds for sums of operators $T_k^{(-k)}$ are reduced to standard sparse bounds for singular integral operators. This only requires the assumption $m\in L^\infty$; note that
\begin{equation}\label{eq:LinftyvsBesov}\|m\|_\infty\lesssim \sup_{t>0}\|\phi m(t\cdot)\|_{B^0_1(M^{p\to p} )}
\lesssim \sup_{t>0}\|\phi m(t\cdot)\|_{B^{d(1/p-1/q)}_1(M^{p\to q} )}.
\end{equation}
\begin{lem} \label{lem:Kk-k}
For any $1 < p \leq q < \infty$, and for any finite subset $\digamma \subset {\mathbb {Z}}$
\begin{equation*}
\Big|\biginn{\sum_{k \in \digamma} T_k^{(-k)} P_k f_1}{f_2} \Big| \lesssim \| m \|_\infty \, \Lambda_{p,q'}^{*}(f_1,f_2),
\end{equation*}
uniformly in $\digamma\subset{\mathbb {Z}}$, for all $f_1,f_2\in C^\infty_c$.
\end{lem}
The proof is straightforward and will be given in the auxiliary \S\ref{app:Kk-k}.

We now introduce operators which are local at a fixed spatial scale. For a fixed finite set $\digamma \subseteq {\mathbb {Z}}$, let $k_{\min} := \min \digamma$ and $k_{\max} := \max \digamma$. Given $j \in {\mathbb {Z}}$, it is convenient to define
\begin{equation}\label{defn:Tj}
\mathcal{T}_j f \equiv \mathcal{T}_{j,\digamma} f := \sum_{\substack{k \in \digamma\\ k > -j}} T_k^{(j)}P_kP_k f
\end{equation}
and to note that
\begin{equation}\label{eq:from Tkl to Tj}
\sum_{k \in \digamma} \sum_{\ell=1}^N T_k^{(\ell-k)} P_k P_k= \sum_{-k_{\max} < j \leq -k_{\min} + N} \mathcal{T}_j.
\end{equation}
By construction, the operators ${\mathcal {T}}_j$ are local at scale $2^j$, in the sense that if $S$ is a cube of side length $2^j$,
\begin{equation}\label{supportin3S}
{\mathrm{supp}\,}(f)\subset S \implies {\mathrm{supp}\,} ({\mathcal {T}}_{j} f) \subset 3S.
\end{equation}
Indeed, by our definition of the $\Phi_0$ and the $\Psi_j$ in \eqref{eqn:defofPsiell}, $T_k^{(j)}P_k P_k[f{\mathbbm 1}_{S}]$ is supported in the set
$\{x: {\mathrm{dist}}(x,S) \le 2^{j-1} + 10^{-d} 2^{-k+1} \}$. Thus
${\mathcal {T}}_{j,\digamma} [f{\mathbbm 1}_{S} ]$
is supported where
$ {\mathrm{dist}}(x,S) \le 2^{j-1} (1+ 2\cdot 10^{-d} ) $, and hence in
$\{x: {\mathrm{dist}}(x,S) <2^{j} \} \subseteq 3S$.

The key estimate in proving the sparse bounds for $m(D)$ is the following modified sparse bound for sums of ${\mathcal {T}}_j$, uniformly in the number of terms in the $j$-sum. Throughout the paper we set \[L(Q)=\log_2(\mathrm{sidelength}(Q))\] so that $L(Q)=N$ for a dyadic cube of side length $2^N$.

\begin{thm}\label{thm:main}
Let $1<p<q<\infty$.
Given integers $N_1 \leq N_2$, a dyadic cube $S_0 \in {\mathfrak {Q}}$ such that $L(S_0)=N_2$ and a finite subset $\digamma \subseteq {\mathbb {Z}}$, the inequality
\begin{equation}\label{eq:main sparse triple}
\Big|\biginn{\sum_{j=N_1}^{N_2}\mathcal{T}_{j,\digamma} f_1}{f_2}\Big| \leq c \, {\mathcal {C}} \Lambda^{**}_{S_0,p,q'}(f_1,f_2)
\end{equation}
holds for all $f_1, f_2\in C^\infty_c$
uniformly in $N_1, N_2$, $\digamma$ and $S_0$, where ${\mathcal {C}}$ is given by
\begin{subequations} \label{eq:defofcC}
\begin{align} \label{eq:defofcCThma}
{\mathcal {C}}&:=\mathcal{A}_{p,r,q}[m]
&& \text{ if } 1<p\le q<2, \quad p<r\le q,
\\
\label{eq:defofcCThmb}
{\mathcal {C}}&:=\mathcal{A}_{p,p,2}[m]
&&
\text{ if } 1<p<2 \quad (\text{and } q=2),
\\ \label{eq:defofcCThmc}
{\mathcal {C}}&:=\mathcal{A}_{p,p,q}[m] + \mathcal{A}_{q',r,r}[m]
&& \text{ if }
1<p< 2< q< p', \quad q'<r\le 2,
\end{align}
\end{subequations}
and $c=c(p,q,r,\gamma,d)$ is a constant depending only on $p,q,r,\gamma,d$.
\end{thm}

We note that by the definition of ${\mathcal {T}}_j$ in \eqref{defn:Tj}, we may assume that the set $\digamma$ featuring in the left-hand side of \eqref{eq:main sparse triple} has the property that $k > -N_2=-L(S_0)$ for $k \in \digamma$.

Also note by \eqref{supportin3S} that in order to prove the theorem we may assume without loss of generality that $f_1$ is supported in $S_0$ and $f_2$ is supported in $3S_0$.

We shall use standard arguments in the theory of sparse domination to make the following
\begin{observation} \label{obs:basicobs}
In order to prove Theorems \ref{thm:qle2}, \ref{thm:qge2} and \ref{thm:multq>2}, it suffices to prove Lemma \ref{lem:Kk-k} and Theorem \ref{thm:main}.
\end{observation}
The proof of this observation is included in the auxiliary \S\ref{observation-proof}.

\subsection{Induction scheme for the proof of Theorem \ref{thm:main}}

We will prove \eqref{eq:main sparse triple} by induction on $\mathbf n$ where $\mathbf n+1$ is the number of terms in the $j$-sum.

\begin{definition}
For $\mathbf n=0,1,2,\dots $, let $\mathbf U(\mathbf n)$ be the smallest nonnegative constant $U$ so that for all pairs $(N_1, N_2)$ with $0\le N_2-N_1\le \mathbf n$, for all finite sets $\digamma \subset {\mathbb {Z}}$ and for all dyadic cubes $S_0\in {\mathfrak {Q}}$ with $L(S_0)=N_2$
we have
\[
\big| \biginn{\sum_{j=N_1}^{N_2} \mathcal{T}_{j,\digamma} f_1}{f_2}\big|
\le U \Lambda^{**}_{S_0,p,q'}(f_1,f_2)
\]
whenever ${\mathrm{supp}\,}(f_1)\subset S_0$.
\end{definition}

For the inductive argument we first consider the base case
$\mathbf n=0.$
We distinguish two situations, $q\ge 2$ and $q<2$.
For fixed $j_0$ we let
\begin{equation}\label{defofcCjo}
{\mathcal {C}}_{p,r,q} (j_0)= \sup_{k>-j_0} A_{p,r,q}^{k,j_0+k}.
\end{equation}
It is immediate that
$\sup_{j_0} {\mathcal {C}}_{p,r,q}(j_0) $ is bounded by
${\mathcal {C}}_{\eqref{eq:defofcCThmc}}$ if $q\ge 2$, $r=p$, bounded by ${\mathcal {C}}_{\eqref{eq:defofcCThmb}}$ if $r=p$, $q=2$ and bounded by ${\mathcal {C}}_{\eqref{eq:defofcCThma}}$ if
$q<2$, $p<r\le q$.
\begin{lem}\label{lem:mult single spatial scale qge2}
Let $1<p\le 2\le q\le p'$ and $j_0\in{\mathbb {Z}}$. Let $S_0$ be a dyadic cube with $L(S_0)=j_0$ and let $f_1\in L^p$ be supported in $S_0$. Then we have for $f_2\in L^q_{\mathrm{loc}}$
\begin{equation}\label{singlescalesparseqge2}
|\inn{{\mathcal {T}}_{j_0} f_1}{f_2}|
\lesssim {\mathcal {C}}_{p,p,q}(j_0)
|S_0| \jp{f_1}_{S_0,p}\jp{f_2}_{3S_0,q'}.
\end{equation}

\end{lem}

\begin{lem}\label{lem:mult single spatial scale}
Let $1<p<r\le q\le 2$ and $j_0\in{\mathbb {Z}}$. Let $S_0$ be a dyadic cube with $L(S_0)=j_0$
and let $f_1\in L^p$ be supported in $S_0$. Then we have for $f_2\in L^q_{\mathrm{loc}}$
\begin{equation}\label{singlescalesparseq<2}
|\inn{{\mathcal {T}}_{j_0} f_1}{f_2}|
\lesssim \big({\mathcal {C}}_{p,p,q}(j_0) +{\mathcal {C}}_{p,r,q} (j_0) \big)
|S_0| \jp{f_1}_{S_0,p}\jp{f_2}_{3S_0,q'}.
\end{equation}
\end{lem}

Both lemmata can be reduced to $L^p\to L^q$ estimates for the operators ${\mathcal {T}}_{j_0}$; these and the corresponding reduction are stated in \S\ref{sec:base case} (see \eqref{eq:sparse-singlescale0red}
for the reduction argument).
\begin{cor} \label{induction-hypothesis}
With ${\mathcal {C}}$ as in \eqref{eq:defofcC},
\begin{equation}\label{base-case}
\mathbf U(0) \le c_0(p,r,q,d)\, {\mathcal {C}}.
\end{equation}
\end{cor}
\begin{proof} This is immediate from Lemmata \ref{lem:mult single spatial scale qge2} and \ref{lem:mult single spatial scale}, in combination
with the inequality \eqref{eq:A-embeddings}.
\end{proof}

Corollary \ref{induction-hypothesis} is the base case for our induction. We note that the same argument implies $\mathbf U( \mathbf n) \le c(\mathbf n,p,r,q,d){\mathcal {C}}$ for any $\mathbf n \geq 0$, but one needs a uniform bound in $\mathbf n$.
The key is the verification of the induction step, formulated in the following claim.

\begin{prop}[Inductive claim] \label{inductiveclaim} There is a constant
$c=c(p,q,r,\gamma,d)$ such that for all
$\mathbf n>0$,
\[
\mathbf U(\mathbf n) \le \max \{ \mathbf U(\mathbf n-1), \, c \,\mathcal{C} \}, \]
with
${\mathcal {C}}$ as in \eqref{eq:defofcC}.
\end{prop}

The proof structure for the inductive claim is presented in \S\ref{sec:induction step}, with proofs given in \S\ref{sec:goodpart} and
\S\ref{sec:badpart}.
They are based on
Calder\'on--Zygmund decompositions combined with the atomic decomposition outlined in \S\ref{sec:CZ-dec}.
By induction, the conclusion of Theorem \ref{thm:main} follows by combining Corollary \ref{induction-hypothesis} and the inductive claim Proposition \ref{inductiveclaim}.

\section{Atomic and subatomic decompositions}\label{sec:decompositions}
In this section we fix a dyadic reference cube $S_0$ and outline an atomic decomposition for a function $f$ supported in a dyadic cube $S_0$
based on estimates for a martingale square function on the cube $S_0$.

\subsection{Local square functions} \label{sec:atomic}
Let $\{{\mathbb {E}}_n\}_{n \in {\mathbb {Z}}}$ be the conditional expectation operators associated to the $\sigma$-algebra generated by the subfamily ${\mathfrak {Q}}_n$ of cubes in ${\mathfrak {Q}}$ with $L(Q)=-n$ (i.e of side length $2^{-n}$), that is, ${\mathbb {E}}_n f(x) = \mathrm{av}_Q f$ for every $x\in Q$ with $Q\in {\mathfrak {Q}}_n$.
Define the martingale difference operator ${\mathbb {D}}_n:={\mathbb {E}}_{n+1}-{\mathbb {E}}_{n}$ for $n \in {\mathbb {Z}}$.
We shall frequently use the familiar properties
\begin{equation}\label{eq:Dkorth}
{\mathbb {D}}_k^2={\mathbb {D}}_k, \qquad \text{and} \qquad {\mathbb {D}}_k{\mathbb {D}}_{k'}={\mathbb {D}}_{k'}{\mathbb {D}}_k=0 \text{ for } k\neq k',
\end{equation}
as well as
\begin{subequations}\label{eq:Lpk1k2}
\begin{align}\label{eq:k1>k2}
&\| P_{k_2} {\mathbb {D}}_{k_1} \|_{L^p\to L^p} \lesssim 2^{-(k_1-k_2)/p} \quad \text{ if $k_1 \ge k_2$},
\\
\label{eq:k2>k1}
& \| {\mathbb {E}}_{k_1} P_{k_2} \|_{L^p\to L^p}
\lesssim 2^{-(k_1-k_2)/p'} \quad \text{ if $k_2 \ge k_1$},
\end{align}
for $1\le p\le \infty$. From \eqref{eq:k1>k2}, \eqref{eq:k2>k1} and using duality,
\begin{equation}\label{eq:Lpk1k2D}
\|{\mathbb {D}}_{k_1} P_{k_2} \|_{L^p\to L^p} + \| P_{k_2} {\mathbb {D}}_{k_1} \|_{L^p \to L^p} \lesssim 2^{-|k_1-k_2|\min\{1/p,1/p'\}}
\end{equation}
for $1 \leq p \leq \infty$ and $k_1,k_2 \in {\mathbb {Z}}$.
\end{subequations}
The bounds \eqref{eq:Lpk1k2} follow by standard computations exploiting cancellation of $P_k$ and $\mathbb{D}_k$ (see e.g. \cite[Ch. 3]{BRS}). They will allow us for example to interchange ${\mathbb {D}}_k$ and $P_k$ in $L^p$ bounds for $1<p<\infty$.
We note the reproducing formula
\begin{equation}\label{eq:reproducing formula f}
f={\mathbb {E}}_{1-L(S_0)} f + \sum_{k > -L(S_0)} {\mathbb {D}}_k f.
\end{equation}
We start from ${\mathbb {E}}_{1-L(S_0)}$ rather than from ${\mathbb {E}}_{-L(S_0)}$ because that will be convenient in Section \ref{sec:CZ-dec}.
Consider the localized dyadic square function
\begin{equation}\label{eq:local-sqfct}
{\mathfrak {g}}_{S_0}f(x)= |{\mathbb {E}}_{1-L(S_0)}f(x)| + \Big(\sum_{k > -L(S_0) } |{\mathbb {D}}_k f(x)|^2 \Big)^{1/2}.
\end{equation}
Note that
${\mathfrak {g}}_{S_0}f$ is supported on $ S_0$, by definition.
By a trivial $L^2$-bound and standard Calder\'on--Zygmund decomposition (and using Khintchine's inequality), it is well-known that ${\mathfrak {g}}_{S_0}$ satisfies an $L^p$ bound
for all $1<p<\infty$ with constant only depending on $p,d$. This is also a special case of Burkholder's square-function estimate for more general martingales \cite{Burkholder66}.

It will also be convenient to work with a slightly larger and more robust square function. Let $Q_k(x)$ be the unique dyadic cube of sidelength $2^{-k}$ containing $x$ and define a dyadic square function in the spirit of Peetre \cite{Peetre1975},
\begin{equation}\label{eq:local-maxsqfct}
{\mathbb {G}}_{S_0}f(x)= |{\mathbb {E}}_{1-L(S_0)}f(x)| + \Big(\sum_{k > -L(S_0) } \sup_{y\in Q_k(x)} |{\mathbb {D}}_k f(y)|^2 \Big)^{1/2}.
\end{equation}
Since ${\mathbb {D}}_k f$ is constant on dyadic cubes of sidelength $2^{-k-1}$ it is easy to see that $\sup_{y\in Q_k(x)} |{\mathbb {D}}_k f(y)|\lesssim M_{\mathrm{HL}} {\mathbb {D}}_k f(x)$ where $M_{\mathrm{HL}}$ is the Hardy--Littlewood maximal operator. Using the Fefferman--Stein inequalities \cite{FeffermanStein} for the vector-valued Hardy--Littlewood maximal operator and the bounds for ${\mathfrak {g}}_{S_0}$ we obtain, for $1<p<\infty$, \begin{equation}\label{eq:Lp bound sq fn}
\| {\mathbb {G}}_{S_0} f\|_{L^p(S_0)} \le {\mathscr C_{\mathrm{sq},p}}\| f \|_{L^p(S_0)},
\end{equation}
where $\mathscr C_{\mathrm {sq},p}$ only depends on $p$ and $d$.

\subsection{Atomic decomposition}
We will now perform an atomic decomposition of $f$ using the local square function ${\mathbb {G}}_{S_0} f$, following ideas in \cite{Chang-Fefferman1982} (see also \cite{SeegerStudia90}). Given $\mu \in {\mathbb {Z}}$, consider the level sets
\begin{equation}\label{omega def}
\Omega_\mu\equiv \Omega_\mu[f] := \{ x\in S_0 \,:\, {\mathbb {G}}_{S_0} f(x) > 2^\mu\},
\end{equation}
and the open sets
\begin{equation}\label{omega tilde def} \widetilde \Omega_\mu\equiv
\widetilde{\Omega}_{\mu} [f]:= \{ x\in{\mathbb {R}}^d : M_{\mathrm{HL}} {\mathbbm 1}_{\Omega_\mu[f]}(x) > 2^{-1}(10\sqrt d)^{-d} \}.
\end{equation}
Note that $\widetilde{\Omega}_\mu$ is not necessarily contained in $S_0$. Of course, $\Omega_\mu \subseteq \widetilde{\Omega}_\mu$ and $\widetilde{\Omega}_{\mu_2} \subseteq \widetilde{\Omega}_{\mu_1}$ if $\mu_1 \leq \mu_2$.
By the Hardy--Littlewood theorem, one has
\begin{equation}\label{eq:sizes omegas}
|\widetilde{\Omega}_\mu| \le C_d |\Omega_\mu|,
\text{ with } C_d=5^d 2 (10 \sqrt{d})^d
\end{equation}
and Chebyshev's inequality and \eqref{eq:Lp bound sq fn} imply
\begin{equation}\label{eq:size omega}
|\Omega_\mu| \leq 2^{-\mu p} \| {\mathbb {G}}_{S_0} f \|_p^p \leq 2^{-\mu p} {\mathscr C^p_{\mathrm{sq},p}} \| f \|_{L^p(S_0)}^p
\end{equation}
for all $1 < p < \infty$. Let ${\mathcal {R}}_\mu\equiv {\mathcal {R}}_\mu[f]$ denote the family of all dyadic cubes $R \subsetneq S_0$ satisfying
\begin{align}
\label{big portion mu} |R \cap \Omega_{\mu}| &> |R|/2, \\
\label{small portion mu+1} |R \cap \Omega_{\mu+1}| &\leq |R|/2.
\end{align}

\begin{lemma}\label{lemma:R in Omega tilde}
For all $\mu\in {\mathbb {Z}}$ and all $R \in {\mathcal {R}}_\mu$ we have
$10\sqrt d R \subset \widetilde{\Omega}_\mu$.
\end{lemma}

\begin{proof} Let $c=10\sqrt{d}$. For every $x\in cR$, $R\in {\mathcal {R}}_\mu$, we have by \eqref{big portion mu}
\[
\frac{1}{2 (10 \sqrt{d})^d} < \frac{|R \cap \Omega_\mu|}{c^d |R|} \leq \frac{|cR \cap \Omega_\mu|}{|cR|} = \frac{1}{|cR|} \int_{cR} {\mathbbm 1}_{\Omega_{\mu}} \leq M_{\mathrm{HL}} {\mathbbm 1}_{\Omega_\mu} (x).
\]
By the definition \eqref{omega tilde def}, this implies $cR \subset \widetilde{\Omega}_\mu$ for all $R \in {\mathcal {R}}_\mu$.
\end{proof}

The lemma implies in particular that $R \subseteq \widetilde{\Omega}_\mu$ for $R \in {\mathcal {R}}_\mu$. This and \eqref{small portion mu+1} further imply
\begin{equation}\label{big portion mu-mu+1}
|R \cap (\widetilde{\Omega}_\mu \backslash \Omega_{\mu+1})| = |R| - |R \cap \Omega_{\mu+1}| \geq |R|/2.
\end{equation}
We also note that for every dyadic cube $R\subsetneq S_0$ there exists a unique $\mu\in{\mathbb {Z}}$ such that $R\in{\mathcal {R}}_\mu$.

Fix $\mu\in{\mathbb {Z}}$. Let $\widetilde{\mathcal{W}}_\mu=\{W\}$ denote a family of standard dyadic Whitney cubes
\cite[\S VI.1]{Ste70}
whose union is the open set $\widetilde{\Omega}_\mu$, which satisfy
\begin{equation}\label{eq:Whitney W}
{\,\text{\rm diam}} (W) \leq {\mathrm{dist}} (W, \widetilde{\Omega}_\mu^\complement) \leq 4 {\,\text{\rm diam}} (W).
\end{equation}
If there exists $W\in\widetilde{\mathcal{W}}_\mu$ with $S_0\subseteq W$, then we set
\[\mathcal{W}_\mu = \{S_0\}. \]
Otherwise we have $W\subsetneq S_0$ for all $W\in\widetilde{\mathcal{W}}_\mu$ intersecting $S_0$ and we set
\[ \mathcal{W}_\mu = \{W\in \widetilde{\mathcal{W}}_\mu\,:\, W\cap S_0\not=\emptyset\}. \]
The cubes in ${\mathcal {R}}_\mu$ have a unique ancestor in ${\mathcal {W}}_\mu$.
\begin{lemma}\label{lem:RinW}
For each $R \in {\mathcal {R}}_\mu$, there exists a unique $W(R) \in {\mathcal {W}}_\mu$ containing $R$.
\end{lemma}

\begin{proof}
If $\mathcal{W}_\mu=\{S_0\}$, then there is nothing to prove. Otherwise, let $c=10 \sqrt{d}$ and $R \in {\mathcal {R}}_\mu$. By Lemma \ref{lemma:R in Omega tilde}, we have $cR \subseteq \widetilde{\Omega}_\mu$. Let $x_R$ denote the center of $R$ and let $W=W(R) \in {\mathcal {W}}_\mu$ such that $x_R \in W$. With this setup and \eqref{eq:Whitney W} we have
\begin{equation*}
{\mathrm{dist}}(x_R, (cR)^\complement) \leq {\mathrm{dist}} (x_R, \widetilde{\Omega}_\mu^\complement) \leq {\,\text{\rm diam}}(W) + {\mathrm{dist}}(W, \widetilde{\Omega}_\mu^\complement) \leq 5 {\,\text{\rm diam}} (W).
\end{equation*}
Noting that ${\mathrm{dist}}(x_R, (cR)^\complement)=\frac{c {\,\text{\rm diam}}(R)}{2 \sqrt{d}}= 5 {\,\text{\rm diam}}(R)$, we obtain that \[5 {\,\text{\rm diam}}(R) \leq 5 {\,\text{\rm diam}} (W).\] As $R$ and $W$ are dyadic cubes containing $x_R$, we conclude that $R \subseteq W$. The uniqueness follows from the disjointess of $W \in {\mathcal {W}}_\mu$.
\end{proof}

Given $\mu \in {\mathbb {Z}}$ and $W \in {\mathcal {W}}_\mu$, the sets
\begin{equation*}
\mathcal{R}_{W,\mu} := \{ R\in\mathcal{R}_\mu \,:\, R \subseteq W\}
\end{equation*}
are disjoint for different $W$, by disjointness of the $W$. Also, by Lemma \ref{lem:RinW}
\[{\mathcal {R}}_\mu = \bigcup_{W \in {\mathcal {W}}_\mu} {\mathcal {R}}_{W,\mu},\]
where the union is disjoint.
We are now ready to define the atoms.
First, for each dyadic cube $R \subsetneq S_0$ with $L(R)=-k$ let
\[ e_R\equiv e_R[f] := (\mathbb{D}_k f) {\mathbbm 1}_{R} = \mathbb{D}_k (f {\mathbbm 1}_R). \]
We refer to the $e_R$ as {\it subatoms}; they are pairwise orthogonal and $\int e_R=0$. The subatoms are building blocks of larger {\it atoms} which are associated to cubes $W$. Given $\mu \in {\mathbb {Z}}$ and $W \in {\mathcal {W}}_\mu$, these are defined as
\begin{equation}
\label{eq:atoms-combined}
a_{W,\mu} \equiv a_{W,\mu}[f] := \sum_{R\in {\mathcal {R}}_{W,\mu} } e_R[f].
\end{equation}
We refer to the $a_{W,\mu}$ as {\em atoms}, but note that they have a non-standard normalization with respect to other sources in the literature. Indeed, if we define the coefficients
\begin{equation}\label{defn:gammaWmu} \gamma_{W,\mu} \equiv
\gamma_{W,\mu}[f] := \Big( |W|^{-1} \sum_{R\in\mathcal{R}_{W,\mu}} \|e_R[f]\|_2^2 \Big)^{1/2},
\end{equation}
then using orthogonality,
\begin{equation}\label{eq:atom L2 norm}
\|a_{W,\mu} \|_2 = |W|^{1/2} \gamma_{W,\mu}.
\end{equation}
Note that
$|W|^{-1/p} (\gamma_{W,\mu})^{-1} a_{W,\mu} $ corresponds to a {\it $(p,2) $-atom} in the classical atomic Hardy-space theory developed for $p\le 1$ (see e.g. \cite{coifman-weiss-bams}). Note that for $p \leq 2$, H\"older's inequality and \eqref{eq:atom L2 norm} imply
\begin{equation}\label{eq:atom Lp norm}
\| a_{W,\mu} \|_p \le |W|^{1/p} \gamma_{W,\mu}.
\end{equation}

In view of \eqref{eq:reproducing formula f} and the above discussion, we can write the atomic decomposition as
\begin{equation}\label{eq:atomic dec}
f= {\mathbb {E}}_{1-L(S_0)}f + \sum_{\mu \in {\mathbb {Z}}} \sum_{W \in {\mathcal {W}}_\mu} a_{W,\mu}.
\end{equation}
In applications it will be useful to use the fine structure of the $a_{W,\mu} $ and further group subatoms that are at the same scale (see \eqref{eq:atomskatoms}).

\subsection{Properties of the atomic decomposition}

The square function ${\mathbb {G}}_{S_0}$ allows summation of the coefficients $|W|^{1/2}\gamma_{W,\mu}$ in $\ell^2$ over the collection ${\mathcal {W}}_\mu$.

\begin{lemma}\label{lemma:sum of atoms}
Let $\mu \in {\mathbb {Z}}$. Then
\begin{equation}\label{eq:sum of atoms}
\Big(\sum_{W \in {\mathcal {W}}_{\mu}} |W| (\gamma_{W,\mu})^2 \Big)^{1/2} = \Big(\sum_{R\in{\mathcal {R}}_\mu} \|e_R\|_2^2\Big)^{1/2} \leq 2^{\mu+3/2} |\widetilde{\Omega}_\mu|^{1/2}.
\end{equation}
\end{lemma}

\begin{proof} The first identity is by definition.
Using \eqref{big portion mu-mu+1},
\[ \|e_R\|_2^2 \le |R| \|e_R\|_\infty^2 \le 2 |R\cap (\widetilde{\Omega}_\mu\setminus \Omega_{\mu+1})|\,\|e_R\|_\infty^2. \]
Observe that
\[ \sum_{\substack{R\in\mathcal{R}_\mu,\\L(R)=-k}} \|e_R\|_\infty^2 {\mathbbm 1}_R(x) = \sup_{y\in Q_k(x)}|{\mathbb {D}}_k f (y)|^2. \]
Thus, the left-hand side of the square of \eqref{eq:sum of atoms}
is
\[ \sum_{k>-L(S_0)} \sum_{\substack{R\in\mathcal{R}_\mu,\\L(R)=-k}} \|e_R\|_2^2 \le 2 \int_{\widetilde{\Omega}_\mu\setminus \Omega_{\mu+1}} \sum_{k>-L(S_0)}
\sup_{y\in Q_k(x)}|{\mathbb {D}}_k f (y)|^2 \, \mathrm{d} x. \]
which by definition of $\Omega_{\mu+1}$ is bounded by
$2^{2\mu+3} |\widetilde{\Omega}_\mu|$, as claimed.
\end{proof}

Even though the coefficients $\gamma_{W,\mu}$ incorporate $\ell^2$ in their definition, there is an $\ell^p$-analogue of the above lemma for $1 < p < 2$. For notational convenience, define the auxiliary function
\begin{equation}\label{eq:F1 defn}
F_{p} (x):= \Big(\sum_{\mu \in {\mathbb {Z}}}\sum_{W \in {\mathcal {W}}_{\mu}} (\gamma_{W,\mu} [f])^p {\mathbbm 1}_{W}(x) \Big)^{1/p}.
\end{equation}
The following lemma says that $\|F_{p}\|_p$ is controlled by $\|f\|_p$.

\begin{lem}\label{lem:U1-Lp}
Let $1 < p \leq 2$ and $\mu \in {\mathbb {Z}}$. Then
\begin{equation}\label{eq:Lp sum coefficients}
\Big(\sum_{W \in {\mathcal {W}}_{\mu}} |W| (\gamma_{W,\mu})^p \Big)^{1/p} \le 2^{\mu+3/2} |\widetilde{\Omega}_\mu|^{1/p} \le C_d^{1/p} 2^{\mu+3/2} |\Omega_\mu|^{1/p}
\end{equation} with $C_d $ as in \eqref{eq:sizes omegas}.
Moreover,
\begin{equation}\label{eq:Lp sum Omega}
\Big(\sum_{\mu \in {\mathbb {Z}}} 2^{\mu p} |\Omega_\mu|\Big)^{1/p} \le (2C_d)^{1/p} {\mathscr C_{\mathrm{sq},p}} \|f\|_p
\end{equation}
and
\begin{equation}\label{eq:Lp sum coefficients function}
\| F_{p} \|_p
\le 2^{3/2} (2C_d)^{1/p} {\mathscr C_{\mathrm{sq},p}}
\| f \|_{p}.
\end{equation}
\end{lem}

\begin{proof}
Fix $\mu \in {\mathbb {Z}}$. By H\"older's inequality, the definition of ${\mathcal {W}}_\mu$, Lemma \ref{lemma:sum of atoms} and the estimate \eqref{eq:sizes omegas} we have
\begin{align*}
\sum_{W \in {\mathcal {W}}_\mu} |W| (\gamma_{W,\mu})^p &= \sum_{W \in {\mathcal {W}}_\mu} |W|^{1-p/2} (|W|^{1/2}\gamma_{W,\mu})^p \\
& \leq \Big( \sum_{W \in {\mathcal {W}}_\mu} |W| \Big)^{1-p/2} \Big( \sum_{W \in {\mathcal {W}}_\mu} |W| (\gamma_{\mu,W})^2 \Big)^{p/2}
\\&\leq |\widetilde{\Omega}_\mu|^{1-p/2} (2^{2\mu+3} |\widetilde{\Omega}_\mu|)^{p/2} \le C_d |\Omega_\mu| 2^{\mu p+3p/2}.
\end{align*}
This proves \eqref{eq:Lp sum coefficients}. To prove \eqref{eq:Lp sum coefficients function} we use
the definition of $\Omega_\mu$ and \eqref{eq:Lp bound sq fn} to estimate
\begin{align*}
&\sum_{\mu \in {\mathbb {Z}}} 2^{\mu p} |\Omega_\mu|
= \sum_{\mu\in {\mathbb {Z}}} (1-2^{-p})^{-1} \int_{2^{\mu-1}}^{2^\mu } p\alpha^{p-1} \, \mathrm{d}\alpha \big |\{x:{\mathbb {G}}_{S_0} f(x)> 2^\mu\}\big|
\\
&\le (1-2^{-p})^{-1} \int_0^\infty p\alpha^{p-1} |\{{\mathbb {G}}_{S_0} f(x)>\alpha\}| \, \mathrm{d}\alpha
\le 2\|{\mathbb {G}}_{S_0} f\|_p^p\le 2C_d {\mathscr C^p_{\mathrm{sq},p}} \|f\|_p^p
\end{align*}
which gives \eqref{eq:Lp sum Omega} and further implies
\begin{align*}\label{eq:Lp-sum-coefficients function}
\|F_p\|_p^p &=\sum_{\mu\in {\mathbb {Z}}} \sum_{W\in{\mathcal {W}}_\mu} |W|(\gamma_{W,\mu})^p\le 2^{3p/2} C_d
\sum_{\mu \in {\mathbb {Z}}} 2^{\mu p} |\Omega_\mu|
\le 2C_d 2^{3p/2}
{\mathscr C^p_{\mathrm{sq},p}}\|f\|_p^p,
\end{align*}
as desired.
\end{proof}

\subsection{Fine structure analysis of atomic decompositions} \label{sec:fine-structure}
Note that by \eqref{eq:atom L2 norm}
\begin{equation}\label{eq:L2 identity F1}
\Big\| \sum_{\mu \in {\mathbb {Z}}} \sum_{W \in {\mathcal {W}}_\mu} a_{W,\mu} \Big\|_2 = \| F_{2} \|_2
\end{equation}
via \eqref{eq:atom L2 norm}, so for $p=2$ Lemma \ref{lem:U1-Lp} recovers the trivial inequality
\begin{equation}\label{eq:L2 trivial atoms}
\Big\| \sum_{\mu \in {\mathbb {Z}}} \sum_{W \in {\mathcal {W}}_\mu} a_{W,\mu} \Big\|_2 \lesssim \| f \|_2,
\end{equation}
which follows directly from \eqref{eq:atomic dec}. There does not seem to be an $L^p$ analogue of this inequality for $1<p<2$, because there appears to be no immediate relation between the $L^p$ norms of $\sum_{\mu \in {\mathbb {Z}}} \sum_{W \in {\mathcal {W}}_\mu} a_{W,\mu}$ and $F_{p}$ of the type \eqref{eq:L2 identity F1}.
However, we shall rely on other useful analogues where either the atoms or subatoms are of a fixed scale (see Lemmata \ref{lemma:l2 sum eff_k} and \ref{lemma:key atomic} below); these will then be used in conjunction with a weak orthogonality or Littlewood--Paley type argument based on properties of the operator that is estimated.

We first consider variants where the subatoms correspond to a fixed dyadic scale.
For $k\in{\mathbb {Z}}$, $\mu \in {\mathbb {Z}}$ and $W \in {\mathcal {W}}_\mu$, define the families of cubes
\begin{align*}
\mathcal{R}^k_\mu & := \{ R\in\mathcal{R}_\mu \,:\, L(R)=-k\}, \\
\mathcal{R}^k_{W,\mu} & := \{ R\in\mathcal{R}_\mu^k \,:\, R \subseteq W\},
\end{align*}
the coefficients
\begin{equation}\label{defn:gammakWmu}
\gamma^k_{W,\mu }\equiv
\gamma_{W,\mu}^k[f] := \Big( |W|^{-1} \sum_{R\in\mathcal{R}^k_{W,\mu}} \|e_R[f]\|_2^2 \Big)^{1/2},
\end{equation}
and
the fixed scale atoms
\begin{align}\label{eq:atomskatoms} a_{W,\mu}^k \equiv a_{W,\mu}^k[f] &:=\sum_{R\in {\mathcal {R}}^k_{W,\mu} } e_R[f].
\end{align}
Note that if $k\le -L(S_0)$, then $\mathcal{R}^k_\mu=\mathcal{R}^k_{W,\mu}=\emptyset$, $\gamma^k_{W,\mu}=0$ and $a^k_{W,\mu}=0$ by definition.
We have
\begin{equation*}
\gamma_{W,\mu} = \Big(\sum_{k > -L(S_0)}(\gamma_{W,\mu}^k)^2\Big)^{1/2}\qquad \text{and} \qquad
a_{W,\mu} =\sum_{k > - L(S_0)} a_{W,\mu}^k.
\end{equation*}

We observe that the inequality \eqref{eq:atom Lp norm} continues to hold when all subatoms are at a fixed scale.
\begin{lemma}\label{lemma:Lp norm aWk}
Let $k > - L(S_0)$, $\mu \in {\mathbb {Z}}$ and $W \in {\mathcal {W}}_\mu$. If $1 < p \leq 2$, then
\begin{equation*}
\| a_{W,\mu}^k \|_p \leq |W|^{1/p} \gamma_{W,\mu}^k.
\end{equation*}
\end{lemma}

\begin{proof}
By two applications of H\"older's inequality,
\begin{align*}
\Big(\sum_{\substack{R \in {\mathcal {R}}_{W,\mu}^k }} \| e_R \|_{L^p}^p \Big)^{1/p} & \leq \Big(\sum_{\substack{R \in {\mathcal {R}}_{W,\mu}^k }} \| e_R \|_{L^2}^2 \Big)^{1/2} \Big( \sum_{\substack{R \in {\mathcal {R}}_{W,\mu}^k}} |R| \Big)^{\frac{1}{p}-\frac{1}{2}} \\ & \leq \Big(\sum_{\substack{R \in {\mathcal {R}}_{W,\mu}^k}} \| e_R \|_{L^2}^2 \Big)^{1/2} |W|^{\frac{1}{p}-\frac{1}{2}} \,= \gamma_{W,\mu}^k |W|^{1/p}.
\end{align*}
The result now follow by disjointness of the cubes in $R \in {\mathcal {R}}^k_{W,\mu}$.
\end{proof}

For the remainder of this section we fix parameters
\begin{equation}\label{eq:Q-mumin-def}Q\in\mathfrak{Q}(S_0)\quad\text{and}\quad\mu_{\mathrm{min}}\in {\mathbb {Z}}\cup \{-\infty\}.\end{equation}
For the time being the reader may pretend that $Q=S_0$ and $\mu_{\mathrm{min}}=-\infty$, but we will need the additional localization when combining the atomic decomposition with an appropriate Calder\'on--Zygmund decomposition in \S \ref{sec:CZ-dec}.
With this in mind, set $\mathcal{W}_{Q,\mu}=\{W\in \mathcal{W}_\mu\,:\,W\subseteq Q\}$ and
\begin{subequations}
\begin{align}\label{defn:bQ}
b_{Q} &= \sum_{\mu>\mu_{\mathrm{min}}} \sum_{\substack{W\in\mathcal{W}_{Q,\mu}}} a_{W,\mu},
\\
b^k_{Q} &= \sum_{\mu>\mu_{\mathrm{min}}} \sum_{\substack{W\in\mathcal{W}_{Q,\mu}}} a^k_{W,\mu}.
\end{align}
\end{subequations}
Define also
\begin{subequations}
\begin{align}\label{defn:betaQp}
\beta_{Q,p} &= \Big( \sum_{\mu\in{\mathbb {Z}}} \sum_{\substack{W\in\mathcal{W}_{Q,\mu}}} |W| (\gamma_{W,\mu}[f])^p\Big)^{1/p},\\
\beta^k_{Q,p} &= \Big( \sum_{\mu\in{\mathbb {Z}}} \sum_{\substack{W\in\mathcal{W}_{Q,\mu}}} |W| (\gamma_{W,\mu}^k[f])^p\Big)^{1/p}
\end{align}
\end{subequations}
and observe that $\|F_p\mathbf{1}_Q\|_p = \beta_{Q,p}$.
Note that $\mathcal{W}_{Q,\mu}$ and $b_Q$, $b_Q^k$, $\beta_{Q,p}$, $\beta_{Q,p}^k$ all depend on the function $f$. Also observe that the truncation in $\mu$ is omitted in the definitions of $\beta_{Q,p}, \beta^k_{Q,p}$.
Our first observation is a variant of \eqref{eq:L2 identity F1} in $L^p$ for a fixed scale $k$.

\begin{lemma}\label{lemma:f to eff}\label{lemma:b_Q to beta_Q}
Let $1 < p \leq 2$ and $k > - L(S_0)$.
Then
\begin{equation*}
\|b_Q^k\|_p \leq \beta_{Q,p}^k.
\end{equation*}
\end{lemma}

\begin{proof}
Note that
\begin{equation*}
\|b_Q^k\|_{L^p}= \Big( \sum_{\mu>\mu_{\mathrm{min}}} \sum_{\substack{W \in {\mathcal {W}}_{\mu,Q}}} \sum_{\substack{R \in {\mathcal {R}}_{W,\mu}^k}} \| e_R \|_{L^p}^p \Big)^{1/p}
\end{equation*}
as all the cubes occurring in the definition of $b^k_Q$ are disjoint. By Lemma \ref{lemma:Lp norm aWk},
\begin{equation}
\| b^k_Q \|_{L^p} \leq \Big( \sum_{\mu \in {\mathbb {Z}}} \sum_{\substack{W \in {\mathcal {W}}_{Q,\mu}}} |W| (\gamma_{W,\mu}^k)^p \Big)^{1/p} = \beta_{Q,p}^k.
\qedhere
\end{equation}
\end{proof}
We can sum the coefficients $\{\beta^k_{Q,p}\}_{k\in {\mathbb {Z}}}$ in $\ell^2$.
\begin{lemma}\label{lemma:l2 sum eff_k}
If $1< p \leq 2$, then
\begin{equation*}
\Big(\sum_{k > -L(S_0)} (\beta_{Q,p}^{k})^2\Big)^{1/2} \le \|F_p {\mathbbm 1}_Q\|_p= \beta_{Q,p}.
\end{equation*}
\end{lemma}

\begin{proof}
By the definitions and Minkowski's inequality,
\begin{align*}
\Big(\sum_{k > - L(S_0)} (\beta_{Q,p}^{k})^2\Big)^{1/2} & = \Big( \sum_{k > - L(S_0)} \Big( \sum_{\mu>\mu_{\mathrm{min}}} \sum_{\substack{W \in {\mathcal {W}}_{Q,\mu}}} (\gamma_{W,\mu}^k)^{p} |W| \Big)^{2/p}
\Big)^{1/2} \\
& \leq \Big( \sum_{\mu>\mu_{\mathrm{min}}} \sum_{\substack{W \in {\mathcal {W}}_{Q,\mu}}} \big(\sum_{k \in {\mathbb {Z}}} (\gamma_{W,\mu}^k)^2 \big)^{p/2} |W| \Big)^{1/p} \\
& = \Big( \sum_{\mu>\mu_{\mathrm{min}}} \sum_{\substack{W \in {\mathcal {W}}_{Q,\mu}}} (\gamma_{W,\mu} )^{p} |W| \Big)^{1/p} \leq \|F_p{\mathbbm 1}_Q\|_p= \beta_{Q,p},
\end{align*}
as desired.
\end{proof}

There is a second variant which consists in fixing the scale of the atoms rather in addition to that of the subatoms. Given integers $k\in{\mathbb {Z}}$ and $n \geq 0$, define
\begin{equation}\label{defn:f1kn}
b^{k,n}_Q = \sum_{\mu>\mu_\mathrm{min}} \sum_{\substack{W\in \mathcal{W}_{Q,\mu},\\L(W)=-k+n}} a_{W,\mu}^k,
\end{equation}
and
\begin{equation}\label{defn:effkn}
\beta^{k,n}_{Q,p} = \Big( \sum_{\mu\in {\mathbb {Z}}} \sum_{\substack{W\in\mathcal{W}_{Q,\mu},\\L(W)=-k+n}} |W| (\gamma_{W,\mu}^k[f])^p\Big)^{1/p}.
\end{equation}
Note that by definition, $b^{k,n}_Q=0$ and $\beta_{Q,p}^{k,n}=0$ unless $k>-L(S_0)$.
Lemma \ref{lemma:f to eff} continues to hold for these fixed-scale $W$ versions. A crucial observation is that we obtain a gain if we move to a larger Lebesgue exponent $r > p$. This will allow us to think of the case $W=R$ as the dominant contribution.
This observation will be crucial in later proofs.

\begin{lemma} \label{lem:bkn}
Let $1 < p \le r \leq 2$.
Then for
$k > -L(S_0)$ and $n \geq 0$,
\[ \|b^{k,n}_Q\|_r \le 2^{-nd(\frac1p-\frac1r)} 2^{kd(\frac1p-\frac1r)} \beta^{k,n}_{Q,p}. \]
\end{lemma}

\begin{proof}
Arguing as in the proof of Lemma \ref{lemma:f to eff},
\begin{equation*}
\| b^{k,n}_Q \|_{L^r} \leq \Big( \sum_{\mu \in {\mathbb {Z}}} \sum_{\substack{W \in {\mathcal {W}}_{Q,\mu} \\ L(W)=-k+n }} |W| (\gamma_{W,\mu}^k)^r \Big)^{1/r}.
\end{equation*}
Using the embedding $\ell^p \subseteq \ell^r$ for $p < r$,
\begin{align*}
\| b^{k,n}_Q \|_{L^r} & \leq \Big( \sum_{\mu \in {\mathbb {Z}}} \sum_{\substack{W \in {\mathcal {W}}_{Q,\mu} \\ L(W)=-k+n}} (\gamma_{W,\mu}^k)^p |W| |W|^{\frac{p}{r}-1} \Big)^{1/p}
\\
& \leq 2^{(k-n)d(\frac{1}{p}-\frac{1}{r})} \beta^{k,n}_{Q,p}.
\qedhere
\end{align*}
\end{proof}
We also have the following variant of Lemma \ref{lemma:l2 sum eff_k}.
\begin{lemma}\label{lemma:key atomic}
For every $n\ge 0$,
\begin{equation*}
\Big(\sum_{k > -L(S_0)} (\beta_{Q,p}^{k,n})^{p}\Big)^{1/p} \le \| F_p{\mathbbm 1}_Q \|_p = \beta_{Q,p}.
\end{equation*}
\end{lemma}

\begin{proof}
This follows since $k$ and $n$ are coupled.
\end{proof}

This lemma allows us to sum in $\ell^p$ rather than $\ell^2$ provided that the quantity $L(W)-L(R)$ is constant. In applications, this constitutes a great advantage which permits us to prove endpoint bounds.

\section{The base case}\label{sec:base case}
Recalling \eqref{eq:Q-mumin-def} we set
\[ \mu_{\mathrm{min}} = -\infty \]
throughout this section.
We first note the following observations.
Let $p\le r\le q$. By rescaling (recall \eqref{eq:kernels in the scaled multiplier form} and \eqref{Aprq}), we have for $j_0>-k$,
\begin{align} \label{Tkjscaling} \|T_k^{(j_0)} \|_{L^r\to L^q} &= 2^{kd(\frac 1r-\frac 1q)} \|\varphi m(2^k\cdot)*\widehat {\Psi_{j_0+k}}\|_{M^{r\to q}} \\\notag&= 2^{-j_0d(\frac 1p-\frac 1q)} 2^{-kd(\frac 1p-\frac 1r)} A_{p,r,q}^{k,j_0+k}.\end{align}
It is well-known that sparse domination for single spatial scale operators follows from certain rescaled $L^p\to L^q$ estimates (see, for instance, \cite[\S 3.1]{BRS}). In our case, it suffices to verify the
$L^p\to L^q$ estimates in the following two lemmata; in both the implicit constants do not depend on $j_0$ and $\digamma$. Recall the definition ${\mathcal {C}}_{p,r,q}(j_0) =\sup_{k>-j_0} A_{p,r,q}^{k, j_0+k}$
(see \eqref{defofcCjo}).

\begin{lem}\label{lem:mult-single-spatial-scale-qge2}
Let $1<p\le 2\le q\le p'$ and $j_0\in{\mathbb {Z}}$. Let $S_0$ be a dyadic cube of side length $2^{j_0}$ and let $f\in L^p$ be supported in $S_0$. Then
\begin{equation}\label{singlescalesparse-qge2}
\|{\mathcal {T}}_{j_0} f\|_q
\lesssim 2^{-j_0 d(\frac 1p-\frac 1q)} {\mathcal {C}}_{p,p,q}(j_0) \|f\|_p.
\end{equation}
\end{lem}

\begin{lem}\label{lem:mult-single-scale-pq}
Let $1<p<r\le q\le 2$ and $j_0\in{\mathbb {Z}}$. Let $S_0$ be a dyadic cube of side length $2^{j_0}$ and let $f\in L^p$ be supported in $S_0$. Then
\begin{equation}\label{single-scale-sparse}
\|{\mathcal {T}}_{j_0} f\|_q\lesssim 2^{-j_0 d(\frac 1p-\frac 1q)} \big({\mathcal {C}}_{p,p,q}(j_0)+ {\mathcal {C}}_{p,r,q}(j_0) \big)\|f\|_p.
\end{equation}
\end{lem}
\begin{proof}[Reduction of Lemmata \ref{lem:mult single spatial scale qge2}, \ref{lem:mult single spatial scale} to Lemmata
\ref{lem:mult-single-spatial-scale-qge2}, \ref{lem:mult-single-scale-pq}] \label{sec:Reduction-of-lemmas}Both reductions use the same argument; we therefore abbreviate by ${\mathcal {C}}(j_0)$ the respective constants
${\mathcal {C}}_{p,p,q}(j_0) $, for $q\le 2$, and ${\mathcal {C}}_{p,p,q}(j_0)+{\mathcal {C}}_{p,r,q}(j_0)$, for $q\ge 2$.
Keeping in mind that $f_1$ vanishes in $S_0^\complement$ we estimate
\begin{equation}\label{eq:sparse-singlescale0red}\begin{aligned}
&|\inn{{\mathcal {T}}_{j_0} f_1}{f_2} |=
|\inn{{\mathcal {T}}_{j_0} f_1}{f_2{\mathbbm 1}_{3S_0} } |\le
\|{\mathcal {T}}_{j_0} f_1\|_q \|f_2{\mathbbm 1}_{3S_0} \|_{q'}
\\&\lesssim {\mathcal {C}}({j_0}) 2^{-j_0 d(\frac 1p-\frac 1q)} \|f_1\|_p \|f_2{\mathbbm 1}_{3S_0} \|_{q'}
\lesssim {\mathcal {C}}({j_0}) \jp{f_1}_{S_0,p} \jp{f_2}_{3S_0,q'} |3S_0|,
\end{aligned}
\end{equation}
which gives the desired sparse bounds \eqref{singlescalesparseqge2} and \eqref{singlescalesparseq<2}.
\end{proof}
It remains to prove Lemmata \ref{lem:mult-single-spatial-scale-qge2} and
\ref{lem:mult-single-scale-pq}. The proof of the former is a short standard argument based on Littlewood--Paley inequalities. The proof of the latter is longer and relies on the atomic decomposition discussed in \S\ref{sec:decompositions}.

\subsection{The case \texorpdfstring{$q \geq 2$}{q>=2}: Proof of Lemma \ref
{lem:mult-single-spatial-scale-qge2}}

By the Littlewood-Paley inequality $\| \sum_k L_k P_k F_k \|_q \lesssim \|(\sum_k | F_k|^2)^{1/2} \|_q$
and setting $F_k =P_k f$
and applying
\eqref{eq:reproducing}, we get
\begin{equation}\label{eqn:lemsinglescaleqge2pf1} \|\mathcal{T}_{j_0} f\|_q \lesssim \Big\|\Big(\sum_{\substack{k\in \digamma}} | T^{(j_0)}_k P_k f|^2\Big)^{1/2} \Big\|_q
\lesssim \Big( \sum_{\substack{k\in \digamma}} \| T^{(j_0)}_k P_k f\|_q^2\Big)^{1/2}
\end{equation}
where we applied Minkowski's inequality in $L^{q/2}$.
By \eqref{Tkjscaling} with $r=p$ we get
\[ \|\mathcal{T}_{j_0} f\|_q \lesssim {\mathcal {C}}_{p,p,q}(j_0) 2^{-j_0 d (\frac 1p - \frac 1q)}\Big( \sum_{k\in{\mathbb {Z}}} \|P_k f\|_p^2\Big)^{1/2}. \]
Using $p\le 2$ and Minkowski's inequality in $L^{2/p}$ we obtain
\begin{align*} \|\mathcal{T}_{j_0} f\|_q &
\lesssim {\mathcal {C}}_{p,p,q}(j_0) 2^{-j_0 d (\frac 1p - \frac 1q)} \Big\|\Big( \sum_{k\in{\mathbb {Z}}} |P_k f|^2\Big)^{1/2}\Big\|_p
\\ &\lesssim
{\mathcal {C}}_{p,p,q}(j_0) 2^{-j_0 d (\frac 1p - \frac 1q)} \|f\|_p\end{align*}
and \eqref{singlescalesparse-qge2} is proved.
\qed

\subsection{The case \texorpdfstring{$q<2$}{q<2}: Proof of Lemma \ref{lem:mult-single-scale-pq}}\label{sec:PfofLemma}

Let $f\in L^p(S_0)$, with $L(S_0)=j_0$, and decompose using \eqref{eq:reproducing formula f}, for fixed $k \in {\mathbb {Z}}$,
\begin{align*}
f&={\mathbb {E}}_{1-j_0}f +\sum_{m > -j_0-k} {\mathbb {D}}_{k+m} f.
\end{align*}
Next we split
\[ {\mathcal {T}}_{j_0} f
= I+II_1+II_2,\]
where
\begin{align*}
I &= \sum_{\substack{k\in \digamma}} T_k^{(j_0)} P_kP_k
[{\mathbb {E}}_{1-j_0} f]
\end{align*}
and $II_1$ and $II_2$ are defined in terms of the additional decomposition ${\mathbb {D}}_{k+m} f=\sum_{n \geq0} b_{S_0}^{k+m,n}$ as
\[
II_1 = \sum_{m\ge 0} \sum_{0\le n \leq 2m} \sum_{\substack{k\in \digamma}}
II_{m,n,k},
\]
\[
II_2 =\sum_{m< 0} \sum_{n \geq 0} \sum_{\substack{k\in \digamma}} II_{m,n,k} \,+\, \sum_{m\ge 0} \sum_{n > 2m} \sum_{\substack{k\in \digamma}} II_{m,n,k},\,\text{where}
\]
\[
II_{m,n,k} = P_k
T_k^{(j_0)} P_k {\mathbb {D}}_{k+m} b_{S_0}^{k+m,n}.
\]

It is useful to keep in mind that $II_{m,n,k}=0$ unless $k>-j_0$ and $k+m>-j_0$.
Our goal is to control the $L^q$ norm of the three terms by a constant times ${\mathcal {C}}_{p,r,q}(j_0)2^{-j_0 d(1/p-1/q)}\|f\|_{L^p(S_0)}$. For the first two terms we get the better bound with $r=p$.

\subsubsection*{Estimation of $\|I\|_q$} For the term $I$ we take $r=p$ and estimate, using \eqref{eq:k2>k1} and \eqref{Tkjscaling} with $r=p$,
\begin{align*}\|I\|_q &\lesssim \sum_{k\in\digamma}\|T_k^{(j_0)} \|_{L^p\to L^q} \|P_k {\mathbb {E}}_{1-j_0} \|_{L^p\to L^p} \|f\|_p
\\&\lesssim 2^{-j_0 d(\frac 1p-\frac 1q)} \sum_{k\in\digamma}A_{p,p,q}^{k, j_0+k} 2^{-(k+j_0)/p'} \|f\|_p
\\&\lesssim 2^{-j_0d(\frac 1p-\frac 1q)} {\mathcal {C}}_{p,p,q} (j_0)\sum_{\ell>0} 2^{-\ell/p'}\|f\|_p
\end{align*}
and
we get the desired bound.

\subsubsection*{Estimation of $\|II_1\|_q$}
By the almost-orthogonality of the $P_k$ and the
resulting inequality $\|\sum_kP_k F_k\|_q\lesssim(\sum_k\|F_k\|_q^q)^{1/q} $ we get
\begin{align} \label{eq:IIIafterorth}\|II_1\|_q\lesssim&
\sum_{m\ge 0}\sum_{0\le n\le 2m} \Big(\sum_{\substack {k\in \digamma}} \|II_{m,n,k} \|_q^q\Big)^{1/q} .
\end{align}
Now, by \eqref{Tkjscaling}, \eqref{eq:Lpk1k2} and Lemma \ref{lem:bkn}, we have for all $p \leq r \leq q$ that
\begin{multline} \label{eq:III mnk}
\|II_{m,n,k}\|_q \lesssim \|T_k^{(j_0)} \|_{L^r\to L^q} \|P_k{\mathbb {D}}_{k+m}\|_{L^r\to L^r}
\| b_{S_0}^{k+m,n}\|_r
\\ \lesssim {\mathcal {C}}_{p,r,q}(j_0) 2^{-j_0d(\frac 1p-\frac 1q)} 2^{-kd(\frac 1p-\frac 1r)}
2^{-|m|/r'} 2^{-nd(\frac1p-\frac1r)} 2^{(k+m)d(\frac1p-\frac1r)} \beta_{S_0,p}^{k+m,n}.
\end{multline}
We use this with $r=p$ and estimate the right-hand side of \eqref{eq:IIIafterorth}
by a constant times
\begin{multline*}
2^{-j_0d(\frac 1p-\frac 1q)}{\mathcal {C}}_{p,p,q}(j_0) \sum_{m\ge 0} \sum_{0\le n\le 2m} 2^{-m/p'} \Big(\sum_{k> -j_0} (\beta_{S_0,p}^{k+m,n})^q\Big)^{1/q} \\
\lesssim
2^{-j_0d(\frac 1p-\frac 1q)} {\mathcal {C}}_{p,p,q}(j_0) \sum_{m\ge 0} \sum_{0\le n\le 2m} 2^{-m/p'} \| f \|_p
\end{multline*}
using $\ell^q \subseteq \ell^p$, Lemma \ref{lemma:key atomic} and \eqref{eq:Lp sum coefficients function}. Altogether \begin{align*}
\|II_1\|_{q} &\lesssim\sum_{m\ge 0}\sum_{0\le n\le 2m } \Big( \sum_{k> -j_0} \|II_{m,n,k} \|_q^q\Big)^{1/q}
\\
& \lesssim {\mathcal {C}}_{p,p,q} (j_0) 2^{-j_0d(\frac 1p-\frac 1q)}\sum_{m\ge 0} (1+2m)2^{-m/p'} \|f\|_p \\
& \lesssim {\mathcal {C}}_{p,p,q} (j_0) 2^{-j_0d(\frac 1p - \frac 1q)} \|f\|_p,
\end{align*}
which finishes the estimation of $\|II_1\|_q$.

\subsubsection*{Estimation of $\|II_2\|_q$}
By the almost orthogonality of the $P_k$ we have
\begin{align*} \Big \|\sum_{\substack{k\in \digamma}} P_k II_{m,n,k} \Big\|_q
&\lesssim \Big (\sum_{\substack{k\in \digamma}} \| II_{m,n,k} \|_q^q\Big)^{1/q}
\\ &\lesssim \Big (\sum_{k\in\digamma}\| II_{m,n,k} \|_q^r\Big)^{1/r}
\end{align*}
We use \eqref{eq:III mnk} with $r > p$, followed by the embedding $\ell^r \subseteq \ell^p$, Lemma \ref{lemma:key atomic} and \eqref{eq:Lp sum coefficients function} to deduce
\begin{align*}
& \Big (\sum_{k\in\digamma}\| II_{m,n,k} \Big\|_q^r\Big)^{1/r}
\\
& \lesssim {\mathcal {C}}_{p,r,q} (j_0) 2^{-j_0d(\frac 1p-\frac 1q) } 2^{-|m|/r'} 2^{(m-n)d(\frac 1p-\frac 1r)} \Big(\sum_{k\in\digamma} (\beta_{S_0,p}^{k+m,n})^r
\Big)^{1/r} \\
& \lesssim {\mathcal {C}}_{p,r,q} (j_0) 2^{-j_0 d(\frac 1p-\frac 1q) } 2^{-|m|/r'} 2^{(m-n)d(\frac 1p-\frac 1r)} \| f\|_p
\end{align*}
and therefore, using $r>p$,
\[ \|II_2\|_q\lesssim \sum_{m\in {\mathbb {Z}}} \sum_{n\ge \max\{0, 2m\}}
\Big(\sum_{k> -j_0} \|II_{m,n,k}\|_q ^q\Big)^{1/q}
\lesssim {\mathcal {C}}_{p,r,q} (j_0) 2^{-j_0d(\frac 1p-\frac 1q)}\|f\|_p .
\]
This finishes the estimation of $\|II_2\|_q$ and the proof of the lemma. \qed

\section{Combination of atomic and \texorpdfstring{\\}{} Calder\'on--Zygmund decompositions} \label{sec:CZ-dec}

Let $S_0$ be a dyadic cube, let $f_1$ and $f_2$ be given functions. Assume that $f_1$ is supported in $S_0$ and that $f_2$ is supported in $3S_0$.
In analogy to the Calder\'on--Zygmund decomposition we decompose the functions $f_1$, $f_2$ given some threshold parameters $\alpha_1, \alpha_2 >0$; these will be defined as
\begin{equation}\label{alpha-definition}\alpha_1= \langle f_1 \rangle_{S_0, p},
\qquad \alpha_2= \langle f_2 \rangle_{3S_0,q'}.
\end{equation}
The decomposition of $f_2$ will be essentially based on a Calder\'on--Zygmund decomposition at level $\alpha_2$, see \eqref{good2}, \eqref{b2Q} below.
We describe the decomposition of $f_1$ which is more involved and essentially based on the atomic decomposition introduced in \S \ref{sec:decompositions}. The idea of combining atomic and Calder\'on--Zygmund decompositions was previously used in \cite{SK}, and can be traced back to \cite{Christ-rough}, although here we need a different variant.

In the proof we will use two large constants $U_1$, $U_2$ which need to significantly
exceed various constants in standard maximal or square function inequalities, or combinations thereof; we shall see that any choice of $U_1$, $U_2$ with
\begin{subequations}
\begin{align}
\label{lwbdfA1}
U_1&\ge (1-\gamma)^{-1/p} (2^{100}d)^{d/p} {\mathscr C_{\mathrm{sq},p}}\,,
\\
\label{lwbdfA2}
U_2 &\ge (1-\gamma)^{-1/q'} (2^{100}d)^{d/q'}\,,
\end{align}
\end{subequations}
and ${\mathscr C_{\mathrm{sq},p}}$ as in \eqref{eq:Lp bound sq fn},
will work.

We start writing
\begin{equation}\label{eqn:goodbaddecomp}
f_1 = g_1 + b_1,
\end{equation}
with the ``good'' function defined as
\begin{subequations}\begin{equation}\label{eq:goodfct-expansion}
g_1=\mathbb{E}_{-L(S_0)+1} f_1 + \sum_{k> -L(S_0)} g^k_1,\end{equation} where
\begin{equation}\label{defn:good} g^k_1 = \sum_{\mu\in{\mathbb {Z}}\,:\,2^\mu\le U_1\alpha_1} \sum_{R\in\mathcal{R}_\mu^k} e_R
\end{equation} \end{subequations}
and the ``bad" function defined as $b_1 = f_1 - g_1$, i.e. $b_1=\sum_{k > -L(S_0)} b^k_1$ with
\begin{equation}\label{defn:bad}
b^k_1 = \sum_{\mu\in{\mathbb {Z}}\,:\, 2^\mu>U_1 \alpha_1} \sum_{R\in\mathcal{R}_\mu^k} e_R.
\end{equation}
Note that $b_1^k=g_1^k=0$ for $k\le -L(S_0)$.
Clearly, \[ |{\mathbb {E}}_{-L(S_0) +1}f_1(x)| \leq 2^{d/p}\alpha_1\le U_1\alpha_1 \quad \text{ for all $x \in S_0$.}
\] Furthermore, the square function associated with the $\{g_1^k\}_{k > -L(S_0)}$ is pointwise bounded by $2 U_1 \alpha_1$. This is analogous to the $L^\infty$ estimate for the ``good'' function in a standard Calder\'on--Zygmund decomposition.

\begin{lemma}\label{lemma:Linfty good} For almost every $x\in S_0$ we have
\[ \Big(\sum_{k > -L(S_0)} |g^k_1(x)|^2\Big)^{1/2} \le 2 U_1\alpha_1. \]
\end{lemma}

\begin{proof}

Fix $x \in S_0$. Let $\digamma$ be a finite family of indices with $k > -L(S_0)$. It suffices to show that
\begin{equation}\label{finitefamily-F}
\Big(\sum_{k\in \digamma} |g^k_1(x)|^2\Big)^{1/2} \le 2 U_1 \alpha_1. \end{equation}
Let \[{\mathfrak {R}}_x=\Big\{ R: R\in \bigcup_{2^\mu\le U_1 \alpha_1} {\mathcal {R}}_\mu, \,\, x\in R, \,\, L(R)=-k \text{ for some } k\in \digamma\Big\} .\] Note that these are the only cubes contributing to $\sum_{k \in \digamma} |g_1^k(x)|^2$. We can assume that ${\mathfrak {R}}_x \neq \emptyset$, as otherwise $\sum_{k \in \digamma} |g_1^k(x)|^2=0$ and the inequality is trivial. Next, let ${\mathfrak {R}}_{x,k}=\{R\in {\mathfrak {R}}_x, L(R)=-k\}$ and let $k_\circ(x) $ be the maximal integer $k\in\digamma$ for which ${\mathfrak {R}}_{x,k}$ is non-empty. Note $k_\circ(x)$ exists as ${\mathfrak {R}}_x$ is non-empty and $\digamma$ is finite. Moreover, observe that ${\mathfrak {R}}_{x,k}$ is either empty or consists only of one (half-open) cube. Let $R_{x, k_\circ(x)} \in {\mathfrak {R}}_{x,k_\circ(x)}$. By definition, there exists a unique $\mu_{x}$ with $2^{\mu_{x}} \leq U_1 \alpha_1$ and $R_{x,k_\circ(x)} \in {\mathcal {R}}_{\mu_{x}}$. Moreover, in view of \eqref{small portion mu+1}, there exists $w_{x} \in R_{x,k_\circ(x)} \backslash \Omega_{\mu_{x}+1}$. Thus, ${\mathbb {G}}_{S_0}f_1(w_x) \leq 2^{\mu_x+1} \leq 2 U_1\alpha_1$. Note that by the maximality of $k_\circ(x)$, we have $w_x \in R$ for all $R \in {\mathfrak {R}}_x$.

Consequently,
\[
\Big( \sum_{k \in \digamma} |g_1^k (x)|^2 \Big)^{\frac 12} = \Big( \sum_{k \in \digamma} \Big| \sum_{R \in {\mathfrak {R}}_{x,k}} e_R (x)\Big|^2 \Big)^{\frac 12} \le {\mathbb {G}}_{S_0} f_1(w_x) \le 2 \, U_1 \alpha_1. \qedhere
\]
\end{proof}

The above lemma will be used in the proof of the sparse bound for dealing with the term that involves $g_1$: see \S\ref{sec:goodpart q<2} for the case $q<2$ and \S\ref{sec:goodpart q>2} for the case $q \geq 2$.

We need a further refined decomposition of the bad parts $b^k_1$.
Recall that by Lemma \ref{lem:U1-Lp} the function $F_{1,p}=F_p$ satisfies
\begin{equation}\label{eq:norm B1}
\|F_{1,p}\|_{L^p(S_0)} \le 2^{3/2} (2C_d)^{1/p} {\mathscr C_{\mathrm{sq},p}}
|S_0|^{1/p} \alpha_1,
\end{equation}
where $C_d=5^d2(10\sqrt d)^d$ (defined in \eqref{eq:sizes omegas}).

Our next goal is to perform a Calder\'on--Zygmund decomposition so that this inequality continues to hold for smaller cubes.
We now bring in the second function $f_2$.
Let $\mu(\alpha_1)$ be the smallest integer $\mu$ such that $2^\mu> U_1 \alpha_1$.
Define
\begin{equation} \label{defofO}{\mathcal {O}}={\mathcal {O}}_1\cup{\mathcal {O}}_2, \end{equation}
where
\begin{subequations} \label{defofO12}
\begin{align}
{\mathcal {O}}_{1} &:= \widetilde \Omega_{\mu(\alpha_1)}[f_1] \,\cup \, \{ x : M_{\mathrm{HL}}(F_{1,p}^p)(x)>
U_1^p \alpha_1^p\},
\\
{\mathcal {O}}_2&:=\{x : M_{\mathrm{HL}}(|f_2|^{q'})(x) >
U_2^{q'}\alpha_2^{q'} \}.
\end{align}
\end{subequations}
Then set
\[\widetilde {{\mathcal {O}}}=\{x: M_{\mathrm{HL}} {\mathbbm 1}_{{\mathcal {O}}}(x)>2^{-10 d} (\sqrt{d})^{-d}\} \,.\]

The following relation between the sizes of $\widetilde{{\mathcal {O}}}$ and $S_0$ is key in order to prove sparse bounds.

\begin{lem}\label{lemma:O less than S_0}
If in the definitions \eqref{defofO}, \eqref{defofO12} we make the choices of $U_1, U_2$ as in \eqref{lwbdfA1}, \eqref{lwbdfA2}, then
\[|\widetilde {\mathcal {O}} | <(1-\gamma) |S_0|.
\]
\end{lem}

\begin{proof}
By the weak type inequality for the Hardy--Littlewood maximal function
\[|\widetilde {\mathcal {O}}|\le 5^d 2^{10d} (\sqrt{d})^d |{\mathcal {O}}|.\] Moreover, by the definition of $\alpha_2$,
\[ |{\mathcal {O}}_2|\le \frac{5^d \|f_2\|_{q'}^{q'}}{ U_2^{q'} \alpha_2^{q'} } = \frac{5^d 3^d |S_0|}{U_2^{q'} }. \]
Furthermore, by \eqref{eq:sizes omegas} and \eqref{eq:size omega},
\begin{equation*}
|\widetilde{\Omega}_{\mu(\alpha_1)} [f_1]| \leq C_d 2^{-\mu(\alpha_1)p} {\mathscr C^p_{\mathrm{sq},p}} \| f_1 \|_{L^p(S_0)}^p <
C_dU_1^{-p} {\mathscr C^p_{\mathrm{sq},p}} |S_0|\,;
\end{equation*}
here we used that $2^{-\mu(\alpha_1) p } < U_1^{-p} \alpha_1^{-p}$.
Finally, using \eqref{eq:norm B1}, we obtain
\[\big|
\{ x: M_{\mathrm{HL}}(F_{1,p}^p)(x)>U_1^p \alpha_1^p\}\big| \leq \frac{5^d \| F_{1,p} \|_p^p}{U_1^p \alpha_1^p} \leq
\frac{5^d 2^{3p/2} 2C_d {\mathscr C^p_{\mathrm{sq},p}} |S_0|}{U_1^p}\,.
\]
Altogether,
\begin{equation}\label{eq:determ-of-constants}
|\widetilde{{\mathcal {O}}}| \leq |{\mathcal {O}}_1|+|{\mathcal {O}}_2|\leq
5^d 2^{10 d} d^{d/2}\Big(
\frac{C_d{\mathscr C^p_{\mathrm{sq},p}}}{U_1^p} + \frac{5^d 2^{1+3p/2} C_d {\mathscr C^p_{\mathrm{sq},p}}}{U_1^p} + \frac{15^d}{U_2^{q'}} \Big) |S_0|.
\end{equation}
For large choices of $U_1$, $U_2$ we get the conclusion of the lemma, and one checks that the choices of $U_1$, $U_2$ made in
\eqref{lwbdfA1}, \eqref{lwbdfA2} achieve this.
\end{proof}
Let $\mathcal{Q}:=\{Q\}$ denote the family of dyadic Whitney cubes whose union is the open set $\widetilde{{\mathcal {O}}}$, which satisfy
\begin{equation}\label{eq:Whitney Q}
5{\,\text{\rm diam}} (Q) \leq {\mathrm{dist}} (Q, \widetilde{{\mathcal {O}}}^\complement) \leq 12 {\,\text{\rm diam}} (Q).
\end{equation}
We note that here we adapt the standard Whitney decomposition with different constants - it will be important that the constant on the left-hand side is greater than $3$ which ensures the family of triple dilates of Whitney cubes has bounded overlap (see \cite{roberlin} and \cite[\S4.4]{BRS} for more details).
Note that by Lemma \ref{lemma:O less than S_0} we have $|Q| < |S_0|$ and thus either $Q \cap S_0 = \emptyset$ or $Q \subseteq S_0$, since $Q$ and $S_0$ are dyadic cubes.

We describe a decomposition of $f_2$ into a good and a bad part which is analogous to the usual Calder\'on--Zygmund decomposition at level $\alpha_2$. Define
\begin{equation}\label{good2} g_2(x) = f_2(x){\mathbbm 1}_{{\mathcal {O}}^\complement} (x)
+\sum_{Q\in {\mathcal {Q}}} \Big( \frac{1}{|Q|}\int_Qf_2(w)\, \mathrm{d} w \Big)\,{\mathbbm 1}_Q(x)
\end{equation}
and let $b_2=f_2-g_2$ which gives $b_2=\sum_{Q\in {\mathcal {Q}}} b_{2,Q}$ with
\begin{equation}\label{b2Q} b_{2,Q}(x) = \Big( f_2(x) - \frac{1}{|Q|}\int_Qf_2(w) \, \mathrm{d} w\Big){\mathbbm 1}_Q(x) .
\end{equation}
We have the standard Calder\'on--Zygmund properties.

\begin{lem} \label{f2CZdec}
(i) For all $Q\in {\mathcal {Q}}$,
$\big(\frac{1}{Q}\int_Q |f_2(x)|^{q'} \, \mathrm{d} x\big)^{1/q'} \lesssim \alpha_2.$

(ii) For almost every $x\in 3Q$, $ |g_2(x)|\lesssim \alpha_2.$
\end{lem}
The proof is immediate from the definition of ${\mathcal {O}}_2$, by the standard reasoning from Calder\'on--Zygmund theory (see for example \cite{Ste70}). We omit the details.

Next we record the following relation between cubes in ${\mathcal {W}}_\mu$ for $2^\mu > U_1\alpha_1$ and cubes in ${\mathcal {Q}}$; note that the family ${\mathcal {Q}}$ does not depend on $\mu$.

\begin{lemma}\label{lemma:W contained in Q}
Let $\mu \in {\mathbb {Z}}$ such that $2^\mu > U_1 \alpha_1$.
For every $W\in \mathcal{W}_\mu[f_1]$ there exists a unique $Q\in \mathcal{Q}$ such that $W\subset Q$.
\end{lemma}

\begin{proof}
We first note that if $W \in {\mathcal {W}}_\mu$, then $W \subseteq {\mathcal {O}}_1$. This follows from the definition of ${\mathcal {W}}_\mu$, as $W \subseteq \widetilde{\Omega}_{\mu} \subseteq \widetilde{\Omega}_{\mu(\alpha_1)}$ for any $\mu \in {\mathbb {Z}}$ with $2^\mu > U_1 \alpha_1$. Furthermore, $cW \subseteq \widetilde{{\mathcal {O}}}$ for sufficiently small $c \geq 1$. This follows because if $y \in cW$, then
\begin{equation*}
M_{\mathrm{HL}}{\mathbbm 1}_{{\mathcal {O}}}(y) \geq \frac{1}{|cW|} \int_{cW} {\mathbbm 1}_{{\mathcal {O}}}(w) \, \mathrm{d} w \geq \frac{1}{|cW|} \int_{W} {\mathbbm 1}_{{\mathcal {O}}} (w) \, \mathrm{d} w = \frac{|W|}{|cW|}=c^{-d}
\end{equation*}
where we used that $W \subseteq {\mathcal {O}}$. Thus, the claim holds provided $1 \leq c < 2^{10} \sqrt{d}$. This claim implies that if $x_W$ denotes the center of $W$, then
\begin{equation*}
{\mathrm{dist}}(x_W, \widetilde{{\mathcal {O}}}^\complement) \geq {\mathrm{dist}} (x_W, (cW)^\complement) = \frac{{\,\text{\rm diam}}(cW)}{2\sqrt{d}} = \frac{c}{2\sqrt{d}} {\,\text{\rm diam}}(W).
\end{equation*}
Furthermore, as $x_W \in W \subseteq \widetilde{{\mathcal {O}}}$, there exists $Q \in {\mathcal {Q}}$ such that $x_W \in Q$. As $Q \subseteq \widetilde{{\mathcal {O}}}$,
\begin{equation*}
{\mathrm{dist}}(x_W, \widetilde{{\mathcal {O}}}^\complement) \leq {\,\text{\rm diam}}(Q) + {\mathrm{dist}}(Q, \widetilde{{\mathcal {O}}}^\complement) \leq 13 {\,\text{\rm diam}}(Q)
\end{equation*}
where in the last inequality we have used \eqref{eq:Whitney Q}. Consequently,
\begin{equation*}
\frac{c}{2\sqrt{d}} {\,\text{\rm diam}} (W) \leq 13 {\,\text{\rm diam}} (Q).
\end{equation*}
As long as $13 \leq \frac{c}{2\sqrt{d}}$, we have that ${\,\text{\rm diam}}(W) \subseteq {\,\text{\rm diam}} (Q)$. Thus, we require a choice of $c$ such that $26 \sqrt{d} \leq c < 2^{10} \sqrt{d}$. Since $W$ and $Q$ are dyadic, this implies that $W \subseteq Q$, and as the cubes in ${\mathcal {Q}}$ have disjoint interior, the cube $Q$ is unique.
\end{proof}

At this point we set once and for all (throughout the proof of Proposition \ref{inductiveclaim} in \S \ref{sec:induction step}-\S \ref{sec:badpart}),
\[ \mu_{\mathrm{min}} = \log_2(U_1 \alpha_1) \]
and recalling the definitions of $ b_{1,Q}^k=b_{Q}^k$, $b_{1,Q}^{k,n}=b_{Q}^{k,n}$ from \S \ref{sec:fine-structure} (with $f=f_1$) we then have
\begin{equation}\label{eq:bk decomp}
b_1^k = \sum_{Q\in\mathcal{Q}} b_{1,Q}^k = \sum_{n\ge 0} \sum_{Q\in\mathcal{Q}} b_{1,Q}^{k,n}. \end{equation}
Note that the families \[\mathcal{W}_{Q,\mu}=\{ W\in\mathcal{W}_\mu\,:\, W\subset Q\}\] are disjoint for different $Q$. For the cubes $Q \in {\mathcal {Q}}$, we have a standard stopping time condition for the function $F_{1,p}$.

\begin{lem}\label{lem:stopping-time-boundf1}
For every $Q\in {\mathcal {Q}}$, we have
\begin{equation}\label{eq:stopping time condition} \beta_{1,Q,p}=\| F_{1,p} {\mathbbm 1}_{Q} \|_{p}=\Big( \sum_{\mu\in{\mathbb {Z}}} \sum_{\substack{ W\in {\mathcal {W}}_{Q,\mu}} } (\gamma_{W,\mu}[f_1] )^p|W| \Big)^{1/p} \lesssim |Q|^{1/p} \alpha_1.
\end{equation}
\end{lem}

\begin{proof}
Let $Q \in {\mathcal {Q}}$. By \eqref{eq:Whitney Q} we have that $c\, Q \cap \widetilde{{\mathcal {O}}}^\complement \neq \emptyset$ provided that $c$ is sufficiently large, say $c=100\sqrt{d}$. Let $x^* \in c\,Q \cap \widetilde{O}^\complement \subseteq c\,Q \cap O^\complement$. Then we have
\begin{equation*}
\frac{1}{c^d|Q|} \int_Q |F_{1,p}|^p \leq \frac{1}{|c\,Q|}\int_{c\,Q} |F_{1,p}|^p \leq M_{\mathrm{HL}}(F_{1,p}^p)(x^*) \leq U_1^p \alpha_1^p,
\end{equation*}
as desired.
\end{proof}

Combining this with Lemma \ref{lemma:l2 sum eff_k} and Lemma \ref{lemma:key atomic} we obtain the key estimates
\begin{equation}\label{eq:l2 sum betakQ}
\Big(\sum_{k> - L(S_0)} (\beta_{1,Q,p}^{k})^{2}\Big)^{1/2} \le \beta_{1,Q,p} \lesssim |Q|^{1/p} \alpha_1.
\end{equation}
and
\begin{equation}\label{eq:key estimate}
\Big(\sum_{k> - L(S_0)} (\beta_{1,Q,p}^{k,n})^{p}\Big)^{1/p} \le \beta_{1,Q,p} \lesssim |Q|^{1/p} \alpha_1.
\end{equation}

In the proof of the sparse bounds, the case $q\geq 2$ will only require the decomposition in $k$ but not in $n$. Correspondingly, Lemma \ref{lemma:b_Q to beta_Q} and \eqref{eq:l2 sum betakQ} will be essential in the proof of Proposition \ref{prop:badpartq>2} in \S\ref{sec:badpart q>2}. The case $q < 2$ is more subtle and requires decomposition in the $n$-parameter. It will be essential in our argument that for $r>p$ the $L^r$ norms of $b^{k,n}_{1,Q}$ exhibit exponential decay in $n$. Correspondingly, Lemma \ref{lem:bkn} and \eqref{eq:key estimate} will be of central importance in the proof of Proposition \ref{prop:badpart} in \S\ref{sec:badpart q<2}.

\section{The induction step}\label{sec:induction step}
Let $\mathbf n\ge 1$. This section is devoted to reducing the proof of the inductive claim (Proposition \ref{inductiveclaim}) to a couple of main estimates.

Recall from \eqref {defn:Tj} that $\mathcal{T}_j f = \sum_{\substack{k \in \digamma\\ k > -j}} T_k^{(j)}P_k^2 f$ and
define
\begin{equation}\label{eq:cTQdef} \mathcal{T} = \sum_{j=N_1}^{N_2} \mathcal{T}_j \qquad \text{and} \qquad \mathcal{T}^Q f = \sum_{j=N_1}^{L(Q)} \mathcal{T}_j[f {\mathbbm 1}_Q]\end{equation}
for $Q\in\mathcal{Q}\cup \{S_0\}$, and note that $\mathcal{T}^{S_0}=\mathcal{T}$. Note that by Lemma \ref{lemma:O less than S_0}, if $Q \in {\mathcal {Q}}$ is such that $Q\cap S_0\not=\emptyset$ then $Q\subsetneq S_0$. In particular $L(Q)<L(S_0)=N_2$, so $L(Q)-N_1<\mathbf{n}$ which puts us in the position to apply the induction hypothesis to the operators ${\mathcal {T}}^Q.$

Next, the decomposition $f_1=g_1+b_1$ as in \eqref{eqn:goodbaddecomp} and \eqref{defn:bQ} give
\begin{align}
\notag &|\langle \mathcal{T} f_1, f_2\rangle| \le |\langle \mathcal{T} g_1, f_2\rangle| + |\langle \mathcal{T} b_1, f_2\rangle| \\
\notag
& \,\,\leq |\langle \mathcal{T}^{S_0} g_1, f_2\rangle|+ \Big|\big\langle \sum_{Q\in\mathcal{Q}} \mathcal{T}^Q b_1, f_2\big\rangle\Big| + \Big|\big\langle \sum_{N_1\le j\le N_2} \sum_{\substack{Q\in\mathcal{Q},\\ L(Q)<j}} \mathcal{T}_j b_{1,Q}, f_2\big\rangle\Big|
\\ \label{eq:three-terms}& \,\,\leq \sum_{Q\in\mathcal{Q}\cup \{S_0\}} |\langle\mathcal{T}^Q g_1, f_2\rangle| + \sum_{Q\in\mathcal{Q}}|\langle \mathcal{T}^Q f_1, f_2\rangle| +\Big |\big\langle \sum_{N_1\le j\le N_2} \sum_{\substack{Q\in\mathcal{Q},\\ L(Q)<j}} \mathcal{T}_j b_{1,Q}, f_2\big\rangle\Big|,
\end{align}
where in the last step we used again $b_1=f_1-g_1$.
We state four propositions that will be proved in the four subsequent sections.
For the good part, i.e. the first term in \eqref{eq:three-terms}
we have the following propositions. Here we use the notation $j_0=L(S_0)$ as in \S\ref{sec:base case}.

\begin{prop}\label{prop:goodpart}
Let $1<p\le q\le 2$.
For all $Q \in {\mathcal {Q}} \cup \{S_0\}$,
\[ |\langle \mathcal{T}^Q g_1, f_2\rangle| \lesssim \|m\|_\infty \,
|Q| \langle f_1\rangle_{S_0,p} \langle f_2\rangle_{3S_0,q'} \,.
\]
\end{prop}

\begin{prop} \label{prop:goodpartq>2}
Let $2<q\le p'<\infty$, $q'<r\le 2$. For all $Q \in {\mathcal {Q}} \cup \{S_0\}$, \[ |\langle \mathcal{T}^Q g_1, f_2\rangle| \lesssim \mathcal{A}_{q',r,r}[m]
\,|Q| \langle f_1\rangle_{S_0,p} \langle f_2\rangle_{3S_0,q'} \,.\]
\end{prop}

Propositions \ref{prop:goodpart} and \ref{prop:goodpartq>2} will be proved in \S \ref{sec:goodpart q<2} and \S \ref{sec:goodpart q>2}, respectively. Note that disjointness of the cubes in $Q \in {\mathcal {Q}}$ implies $\sum_{Q \in {\mathcal {Q}}} |Q| \leq |S_0|$ and thus Propositions \ref{prop:goodpart} and \ref{prop:goodpartq>2} yield
\begin{equation*}
\sum_{Q\in\mathcal{Q}\cup \{S_0\}} |\langle\mathcal{T}^Q g_1, f_2\rangle| \lesssim |S_0| \langle f_1\rangle_{S_0,p} \langle f_2\rangle_{3S_0,q'}.
\end{equation*}
The terms involving ${\mathcal {T}}^Q f_1$ for $Q \in {\mathcal {Q}}$ are estimated using the inductive hypothesis exactly as described in \cite[\S 4.4]{BRS}, as $L(Q)-N_1 < n$. More precisely, given any $\epsilon>0$, for each $Q \in {\mathcal {Q}}$, there exists a $\gamma$-sparse family of cubes ${\mathcal {S}}_Q^\epsilon \subseteq {\mathcal {D}}(Q)$ such that
\begin{equation*}
|\langle\mathcal{T}^Q f_1, f_2\rangle| \leq ( \mathbf U( \mathbf{n}-1) + \epsilon ) \sum_{\tilde{Q} \in {\mathcal {S}}_Q^\epsilon} |\tilde{Q}|\jp{f_1}_{\tilde{Q},p} \jp{f_2}_{3\tilde{Q},q'}
\end{equation*}
holds. By disjointness of the $Q \in {\mathcal {Q}}$ and Lemma \ref{lemma:O less than S_0}, the resulting family \[{\mathcal {S}}^\epsilon = \{S_0\} \cup \bigcup_{Q \in {\mathcal {Q}}: Q \subseteq S_0} {\mathcal {S}}^\epsilon_Q\] is $\gamma$-sparse.
We then get the desired result from the following propositions which take care of the third term in
\eqref{eq:three-terms}.

\begin{prop}\label{prop:badpart}
Let $1<p<r\le q\le 2$. Then
\begin{equation*}
\Big|\Biginn{ \sum_{N_1\le j\le N_2} \sum_{\substack{Q\in\mathcal{Q},\\ L(Q)<j}} \mathcal{T}_j [b_{1,Q}]}{ f_2}\Big| \lesssim
\mathcal{A}_{p,r,q}[m]
|S_0| \jp{f_1}_{S_0,p} \jp{f_2}_{3S_0,q'}.
\end{equation*}
\end{prop}

\begin{prop} \label{prop:badpartq>2} Let $2\le q\le p'<\infty$. Then
\begin{equation*}
\Big|\Biginn{ \sum_{N_1\le j\le N_2} \sum_{\substack{Q\in\mathcal{Q},\\ L(Q)<j}} \mathcal{T}_j [b_{1,Q}]} { f_2}\Big| \lesssim
\mathcal{A}_{p,p,q}[m]
\,
|S_0| \jp{f_1}_{S_0,p} \jp{f_2}_{3S_0,q'}.
\end{equation*}
\end{prop}
Propositions \ref{prop:badpart} and \ref{prop:badpartq>2} will be proved in \S \ref{sec:badpart q<2} and \S \ref{sec:badpart q>2}, respectively.
Notice that the main induction step for Theorem \ref{thm:qle2} follows from Propositions \ref{prop:goodpart} and \ref{prop:badpart}.
For Theorem \ref{thm:qge2} it follows from Propositions \ref{prop:goodpart} and \ref{prop:badpartq>2} and for Theorem \ref{thm:multq>2} it follows from Propositions
\ref{prop:goodpartq>2} and \ref{prop:badpartq>2}.

\section{The good part}\label{sec:goodpart}

Here we prove Proposition \ref{prop:goodpart} and Proposition \ref{prop:goodpartq>2}.
By the definition of $\mathcal{T}_j$ in \eqref{defn:Tj}, using \eqref{supportin3S} and substituting $j$ by $\ell-k$,
\[ \langle \mathcal{T}^Q g_1,f_2\rangle =\sum_{\substack{k\in\digamma}} \sum_{\substack{\ell>0\\N_1\le\ell-k\le L(Q)}} \langle T_k^{(\ell-k)} P_k P_k [g_1 {\mathbbm 1}_Q], f_2{\mathbbm 1}_{3Q}\rangle. \]
Next, using the decomposition $g_1 {\mathbbm 1}_Q= ({\mathbb {E}}_{1-j_0} f_1){\mathbbm 1}_Q + \sum_{k' >-L(Q)} g_1^{k'} {\mathbbm 1}_Q$
as in \eqref{eq:goodfct-expansion},
\[
|\langle \mathcal{T}^Q g_1, f_2\rangle|\le I+II,
\]
where
\begin{equation}\label{goodparttoestimate1}
I= \sum_{\substack{k\in\digamma}} \Big|\sum_{\substack{\ell>0\\N_1\le\ell-k\le L(Q)}} \langle P_kT_k^{(\ell-k)} P_k [{\mathbb {E}}_{1-j_0} (f_1 {\mathbbm 1}_Q)],f_2 {\mathbbm 1}_{3Q} \rangle \Big|,
\end{equation}
\begin{equation}\label{goodparttoestimate}
II=\Big| \sum_{\substack{k\in\digamma}} \sum_{k'>-L(Q)} \sum_{\substack{\ell>0\\N_1\le\ell-k\le L(Q)}} \langle P_kT_k^{(\ell-k)} P_k [g_1^{k'} {\mathbbm 1}_{Q}], f_2{\mathbbm 1}_{3Q} \rangle \Big|
\end{equation}
and $g_1^{k'}$ is as in \eqref{defn:good}.
The main contribution to $II$ is given by the terms with $|k-k'|\lesssim 1$. Therefore we substitute $k'=k+\mathrm{\nu}$ with $\mathrm{\nu}\in{\mathbb {Z}}$ and estimate
\[
II \le \sum_{\mathrm{\nu}\in{\mathbb {Z}}} \Big| \sum_{\substack{k\in\digamma}} \sum_{\substack{\ell>0\\N_1\le\ell-k\le L(Q)}} \langle P_kT_k^{(\ell-k)} P_k [g_1^{k+\mathrm{\nu}} {\mathbbm 1}_{Q}], f_2{\mathbbm 1}_{3Q} \rangle \Big|.
\]

From here on the terms $I$ and $II$ will each be estimated differently depending on whether $q\le 2$ or $q>2$.

\subsection{The case \texorpdfstring{$q\le 2$}{q<=2}: Proof of Proposition \ref{prop:goodpart}}\label{sec:goodpart q<2}
By \eqref{eq:kernels in the scaled multiplier form} and \eqref{eqn:defofPsiell} we have for each fixed $k\in\digamma$ and integers $0<L_1\le L_2$ that
\begin{equation}\label{eq:sum ell telescoping} \Big\|\sum_{\substack{L_1<\ell\le L_2}} T_k^{(\ell-k)} \Big\|_{2\to 2} \le 2
\sup_{\ell>0} \|\varphi m(2^k\cdot)*2^{\ell}\widehat{\Phi_0}(2^{\ell}\cdot)\|_\infty \lesssim \|m\|_\infty, \end{equation}
uniformly in $k,L_1,L_2$.
Using this and the Cauchy--Schwarz inequality
we bound
\[ II\lesssim \|m\|_\infty \sum_{\mathrm{\nu}\in{\mathbb {Z}}}\sum_{\substack{k\in\digamma}} \|P_k {\mathbb {D}}_{k+\mathrm{\nu}} [g_1^{k+\mathrm{\nu}} {\mathbbm 1}_{Q}]\|_2 \|P_k [f_2{\mathbbm 1}_{3Q}]\|_2, \]
which by another application of the Cauchy--Schwarz inequality and \eqref{eq:Lpk1k2} is
\[ \lesssim \|m\|_\infty \sum_{\mathrm{\nu}\in{\mathbb {Z}}} 2^{-|\mathrm{\nu}|/2} \Big\|\Big(\sum_{k>-j_0} |g_1^{k}|^2\Big)^{1/2}\Big\|_{L^2(Q)}
\Big\| \Big( \sum_{k\in{\mathbb {Z}}} |P_k[f_2 \mathbf{1}_{3Q}]|^2 \Big)^{1/2} \Big\|_2.
\]
By Lemma \ref{lemma:Linfty good} we have
\[ \Big\|\Big(\sum_{k>-j_0} |g_1^{k}|^2\Big)^{1/2}\Big\|_{L^2(Q)}\lesssim \alpha_1 |Q|^{\frac12}. \]
Moreover using $q'\ge 2$,
\begin{equation}\label{eq:sum Pkf2Q}
\Big\| \Big( \sum_{k\in{\mathbb {Z}}} |P_k[f_2 \mathbf{1}_{3Q}]|^2 \Big)^{1/2} \Big\|_2\lesssim \|f_2\|_{L^2(3Q)} \lesssim |Q|^{\frac12} \langle f_2\rangle_{3Q,q'} \lesssim |Q|^{\frac12} \alpha_2,
\end{equation}
where the last step follows from Lemma \ref{f2CZdec} in the case $Q\in{\mathcal {Q}}$ (and is void if $Q=S_0$).
Thus we obtain $II\lesssim \|m\|_\infty |Q| \langle f_1\rangle_{S_0,p} \langle f_2\rangle_{3S_0,q'}$, as desired.

To bound $I$ we again use \eqref{eq:sum ell telescoping}, the Cauchy--Schwarz inequality and \eqref{eq:Lpk1k2} to arrive at
\begin{align*}
I &\lesssim \|m\|_\infty \sum_{\substack{k\in\digamma}} \|P_k [({\mathbb {E}}_{1-j_0} f_1){\mathbbm 1}_Q]\|_2 \|P_k [f_2{\mathbbm 1}_{3Q}]\|_2\\
&\le \|m\|_\infty \Big\|\Big(\sum_{k\in{\mathbb {Z}}} |P_k({\mathbb {E}}_{1-L(Q)} f_1){\mathbbm 1}_Q|^2\Big)^{1/2}\Big\|_{L^2(Q)}
\Big\| \Big( \sum_{k\in{\mathbb {Z}}} |P_k[f_2 \mathbf{1}_{3Q}]|^2 \Big)^{1/2} \Big\|_2\\
&\lesssim \|m\|_\infty \|{\mathbb {E}}_{1-j_0} f_1\|_{L^2(Q)} |Q|^{\frac12} \alpha_2 \lesssim \|m\|_\infty |Q| \langle f_1\rangle_{S_0,p} \langle f_2\rangle_{3S_0,q'},
\end{align*}
where in the last two steps we have used \eqref{eq:sum Pkf2Q} and that
\begin{equation*}
\|{\mathbb {E}}_{1-j_0} f_1\|_{L^2(Q)}\lesssim |Q|^{\frac12} \langle f_1\rangle_{S_0,1}\le |Q|^{\frac12} \langle f_1\rangle_{S_0,p}.
\end{equation*}
This concludes the proof of Proposition \ref{prop:goodpart}. \qed

\subsection{The case \texorpdfstring{$q>2$}{q>2}: Proof of Proposition \ref{prop:goodpartq>2}}\label{sec:goodpart q>2}
Here we assume $2<q\le p'<\infty$ and let $r\in (q',2]$.
We begin with estimating the term $II$. Note that by Fubini's theorem,
\[ II\le \sum_{\mathrm{\nu}\in{\mathbb {Z}}} \Big| \int \sum_{\substack{k\in\digamma}} P_k [g_1^{k+\mathrm{\nu}} {\mathbbm 1}_{Q}](x) \sum_{\substack{\ell>0\\N_1\le\ell-k\le L(Q)}}
(T^{(\ell-k)})^* P_k[f_2{\mathbbm 1}_{3Q}](x)\, \, \mathrm{d} x\Big|. \]
By the Cauchy--Schwarz inequality applied to the summation in $k$ and H\"older's inequality applied to the integration in $x$, we obtain that the previous display is no greater than
\[ \sum_{\mathrm{\nu}\in{\mathbb {Z}}} {\mathcal {E}}^1_{q,\mathrm{\nu}} \cdot {\mathcal {E}}_{q'}^2, \]
where
\begin{subequations}
\begin{align}
{\mathcal {E}}^1_{q,\mathrm{\nu}}=&\,\Big\| \Big( \sum_{k\in \digamma} \big| P_k [g_1^{k+\mathrm{\nu}}{\mathbbm 1}_Q]\big|^2\Big)^{1/2} \Big\|_q, \\
{\mathcal {E}}^2_{q'}=&\,\Big\| \Big(\sum_{k\in\digamma} \Big|\sum_{\substack{\ell>0\\N_1\le\ell-k\le L(Q)}} (T_k^{(\ell-k)})^* P_k[f_2 {\mathbbm 1}_{3Q}]\Big|^2 \Big)^{1/2} \Big\|_{q'}.
\end{align}
\end{subequations}
By Lemma \ref{lemma:Linfty good} we have ${\mathcal {E}}^1_{q,\mathrm{\nu}} \le C_q \alpha_1|Q|^{1/q}$, uniformly in $m$. Moreover for $q=2$ we can use Fubini's theorem and \eqref{eq:Lpk1k2} to see that
${\mathcal {E}}^1_{2,\mathrm{\nu}} \lesssim 2^{-|\mathrm{\nu}|/2} \alpha_1|Q|^{1/2}$. By log-convexity of the $L^q$-norm, we deduce
${\mathcal {E}}^1_{q,\mathrm{\nu}} \lesssim_q 2^{-|\mathrm{\nu}|\varepsilon(q)} \alpha_1|Q|^{1/q}$ with $\varepsilon(q)<1/q$
for $2<q<\infty$ and thus
\begin{equation}\label{eq:bound Eq1}
\sum_{\mathrm{\nu}\in{\mathbb {Z}}} {\mathcal {E}}^1_{q,\mathrm{\nu}} \lesssim_q \alpha_1|Q|^{1/q},\quad 2\le q<\infty .
\end{equation}

We shall now prove that
\begin{equation}\label{eq:E2q' bd}
{\mathcal {E}}^2_{q'} \lesssim {\mathcal {A}}_{q',r,r} \alpha_2|Q|^{1/q'}.
\end{equation}
By averaging with Rademacher functions the desired bound will follow if we show that for any sequence $\{a_k\}_{k\in {\mathbb {Z}}}$ with $\sup_k |a_k|\le 1$ and subsets $\Lambda(k)$ of nonnegative integers
we can show that
\begin{align*}
\Big\|\sum_{k\in\digamma} a_k \sum_{ \ell \in \Lambda(k)} (T_k^{(\ell-k)})^* P_k\Big\|_{q'\to q'} \lesssim {\mathcal {A}}_{q',r,r}.
\end{align*}
This can be established by showing that the associated multipliers
\begin{equation}\label{eq:h-multipliers} h(\xi) =\sum_{k\in \digamma} a_k \sum_{\ell\in \Lambda(k)} \big([\phi m(2^k\cdot)] *\widehat{\Psi_\ell} \big) (2^{-k} \xi) \eta(2^{-k}\xi)\end{equation} have $M^{q'\to q'} $
norm bounded by a constant times ${\mathcal {A}}_{q',r,r}$, which by \eqref{studia-estimate} (with $p$ replaced by $q'$) follows from the following lemma.
\begin{lem} \label{lem:studia-verification} Suppose $q'<r\le 2$.
Then the multipliers in \eqref{eq:h-multipliers} satisfy
\begin{equation} \label{mult-verification}
\sup_{t>0} \|\phi h(t\cdot) \|_{B^{d(\frac 1{q'}-\frac 1r)}_1(M^{r\to r} )} \lesssim {\mathcal {A}}_{q',r,r} ,
\end{equation}
with the implicit constant independent of $\digamma$, of the sets $\Lambda(k)\subset {\mathbb {N}}_0$ and of $\{a_k\}$ in the unit ball of $c_0$.
\end{lem}
The proof is straightforward but somewhat technical, and therefore postponed to \S \ref{app:hmultipliers}.
This finishes the proof of \eqref{eq:E2q' bd} and therefore we obtain the desired estimate for the term $II$.
The bound for the term $I$ is slightly simpler.
We again use Fubini's theorem, the Cauchy--Schwarz inequality and H\"older's inequality to obtain that
\[ I \le {\mathcal {E}}_{q}^1\cdot {\mathcal {E}}_{q'}^2, \]
where
\begin{equation*}
{\mathcal {E}}_{q}^1 = \Big\| \Big( \sum_{k\in \digamma} | P_k[({\mathbb {E}}_{1-j_0} f_1){\mathbbm 1}_Q]|^2\Big)^{1/2} \Big\|_q.
\end{equation*}
By Littlewood--Paley theory and H\"older's inequality,
\[ {\mathcal {E}}_q^1 \lesssim \|{\mathbb {E}}_{1-j_0} f_1\|_{L^q(Q)} \lesssim \langle f_1\rangle_{S_0,1} |Q|^{1/q} \le \langle f_1\rangle_{S_0,p} |Q|^{1/q}.\]
Combined with \eqref{eq:E2q' bd} we obtain the desired bound for $I$ which concludes the proof of Proposition \ref{prop:goodpartq>2}. \qed

\section{The bad part} \label{sec:badpart}
Here we prove Proposition \ref{prop:badpart} and Proposition \ref{prop:badpartq>2}.
By the definition of $\mathcal{T}_j$ in \eqref{defn:Tj}, using $b_1=\sum_{k'>-j_0} \sum_{Q\in\mathcal{Q}} b_{1,Q}^{k'}$ and substituting $j$ by $\ell-k$ and $k'$ by $k+\mathrm{\nu}$,
\begin{multline*}
\Big|\Biginn{ \sum_{N_1\le j\le N_2} \sum_{\substack{Q\in\mathcal{Q},\\ L(Q)<j}} \mathcal{T}_j [b_{1,Q}]}{ f_2}\Big|\le \\
\sum_{\mathrm{\nu}\in{\mathbb {Z}}} \sum_{\ell>0} \sum_{k\in\digamma}\Big| \sum_{\substack{Q\in\mathcal{Q}\\L(Q)< \ell-k\le N_2}} \langle T_k^{(\ell-k)} P_k b_{1,Q}^{k+\mathrm{\nu}}, P_kf_2 \rangle \Big|.
\end{multline*}
For fixed $j\le j_0$ we tile $S_0$ with a family ${\mathcal {B}}_{j}=\{B\}$ of dyadic cubes $B$ such that $L(B)=j$. For convenience we also set ${\mathcal {B}}_{j}=\emptyset$ if $j>j_0$.
Then the previous display is
\begin{equation}\label{eq:badpart pre}
\le \sum_{\mathrm{\nu}\in{\mathbb {Z}}} \sum_{\ell>0} \sum_{k\in\digamma} \sum_{B\in{\mathcal {B}}_{\ell-k}} \Big| \Big\langle T_k^{(\ell-k)} P_k \Big[ \sum_{\substack{Q\in\mathcal{Q},\\ Q\subsetneq B}} b_{1,Q}^{k+\mathrm{\nu}} \Big], (P_kf_2){\mathbbm 1}_{3B} \Big\rangle \Big|.
\end{equation}
Let us also recall the scaling relation
\begin{align}\label{eq:Tkl scaling}
\|T_k^{(\ell-k)}\|_{r\to q} &= 2^{kd(\frac1r-\frac1q)}\|\varphi m(2^k\cdot)*\widehat{\Psi_\ell}\|_{M^{r\to q}}\\
\notag &= 2^{-\ell d(\frac1p-\frac1q)} 2^{kd(\frac1r-\frac1q)} A_{p,r,q}^{k,\ell}.
\end{align}

\subsection{The case \texorpdfstring{$q\le 2$}{q<=2}: Proof of Proposition \ref{prop:badpart}}\label{sec:badpart q<2}
In view of \eqref{eq:A-embeddings}
it is no loss of generality to assume that $r>p>1$ is chosen very close to $p$; indeed, it will be convenient to assume
\begin{equation}\label{eq:convenience} d(1/p-1/r)<1-1/r,\end{equation}
which is admissible since $p>1$.
By H\"older's inequality and \eqref{eq:Tkl scaling} we can bound \eqref{eq:badpart pre} by
\begin{multline*}\sum_{\mathrm{\nu}\in{\mathbb {Z}}} \sum_{\ell>0} \sum_{k\in\digamma} \sum_{B\in{\mathcal {B}}_{\ell-k}} A_{p,r,q}^{k,\ell} 2^{-\ell d(\frac1p-\frac1q)} 2^{kd(\frac1r-\frac1q)} \times\\
\Big \|P_k {\mathbb {D}}_{k+\mathrm{\nu}} \Big[\sum_{\substack{Q\in\mathcal{Q},\\ Q\subsetneq B}} b_{1,Q}^{k+\mathrm{\nu}}\Big] \Big \|_r \|(P_k f_2){\mathbbm 1}_{3B}\|_{q'}. \end{multline*}
We next use \eqref{eq:Lpk1k2} and
write $b_{1,Q}^{k+\mathrm{\nu}}=\sum_{n \geq 0} b_{1,Q}^{k+\mathrm{\nu},n}$ so that the above is
\begin{multline}\label{eqn:badpartpf1}
\lesssim \sum_{n\ge 0} \sum_{\mathrm{\nu} \in {\mathbb {Z}}} \sum_{\ell>0} \sum_{k\in\digamma} \sum_{\substack{B\in{\mathcal {B}}_{\ell-k}}} A_{p,r,q}^{k,\ell} 2^{-\ell d(\frac1p-\frac1q)} 2^{kd(\frac1r-\frac1q)} 2^{-|\mathrm{\nu}|/r'}
\\ \times \Big \|\sum_{\substack{Q\in\mathcal{Q},\\ Q\subsetneq B}} b_{1,Q}^{k+\mathrm{\nu},n} \Big \|_r \|(P_k f_2){\mathbbm 1}_{3B}\|_{q'}. \end{multline}
By Lemma \ref{lemma:lift lp to lq} and disjointness of the $Q\in\mathcal{Q}$, we have for $B\in {\mathcal {B}}_{\ell-k}$
\begin{equation*}
\Big\| \sum_{\substack{Q\in\mathcal{Q},\\ Q\subsetneq B}} b_{1,Q}^{k+\mathrm{\nu},n} \Big \|_r \lesssim |B|^{\frac{1}{r}-\frac{1}{q}} \Big( \sum_{Q \in {\mathcal {Q}}} |Q|^{1-\frac{q}{r}} \| b_{1,Q}^{k+\mathrm{\nu},n} \|_r^q \Big)^{1/q}.
\end{equation*}
By Lemma \ref{lemma:b_Q to beta_Q},
$\|b_{1,Q}^{k+\mathrm{\nu},n}\|_r \le 2^{(k+\mathrm{\nu}-n)d(\frac1p-\frac1r)} \beta^{k+m,n}_{1,Q,p}$.
Noting that
\[ 2^{-\ell d(\frac1p-\frac1q)} 2^{kd(\frac1r-\frac1q)} |B|^{\frac{1}{r}-\frac{1}{q}}\times 2^{kd(\frac1p-\frac1r)} |Q|^{\frac1p}
|Q|^{-\frac{1}{r}}= (|Q|/|B|)^{\frac{1}{p}-\frac{1}{r}} \leq 1 \]
for $Q \subseteq B$ and $r>p$, we obtain that \eqref{eqn:badpartpf1} is bounded by a constant times
\begin{flalign*}& \sum_{n\ge 0} \sum_{\mathrm{\nu} \in {\mathbb {Z}}} 2^{-(n-\mathrm{\nu})d(\frac1p-\frac1r)} 2^{-|\mathrm{\nu}|/r'}
\times &
\end{flalign*}
\begin{equation}
\sum_{k\in\digamma} \sum_{\ell>0} \sum_{\substack{B\in\mathcal{B}_{\ell-k}}} A_{p,r,q}^{k,\ell}
\Big( \sum_{\substack{Q\in\mathcal{Q},\\
Q\subsetneq B}} |Q|^{1-\frac qp} (\beta^{k+\mathrm{\nu},n}_{1,Q,p})^q \Big)^{1/q} \|(P_k f_2){\mathbbm 1}_{3B}\|_{q'}.
\label{eq:bad fixed n nu}
\end{equation}
In view of \eqref{eq:convenience} we now fix $n\ge 0$ and $\mathrm{\nu}\in{\mathbb {Z}}$.
Using H\"older's inequality with exponents $(\tfrac1q,\tfrac1{q'})$ on the summation over $B$, \eqref{eq:bad fixed n nu} is estimated by
\[ {\mathcal {A}}_{p,r,q} \sum_{k\in\digamma} \Big( \sum_{Q\in\mathcal{Q}} |Q|^{1-\frac qp} (\beta^{k+\mathrm{\nu},n}_{1,Q,p})^q \Big)^{1/q} \|P_k f_2\|_{q'}. \]
Using H\"older's inequality again with exponents $(\tfrac1q,\tfrac1{q'})$ on the summation over $k$, this is then estimated by
\[{\mathcal {A}}_{p,r,q}\Big(\sum_{k\in\digamma} \sum_{Q\in\mathcal{Q}}
|Q|^{1-\frac qp} (\beta^{k+\mathrm{\nu},n}_{1,Q,p})^q \Big)^{1/q} \Big(\sum_{k\in{\mathbb {Z}}}\|P_k f_2\|_{q'}^{q'}\Big)^{1/q'}.\]
Since $q'\ge 2$ we have
\[ \Big(\sum_{k\in{\mathbb {Z}}}\|P_k f_2\|_{q'}^{q'}\Big)^{1/q'} \le \Big\|\Big(\sum_{k\in{\mathbb {Z}}}|P_k f_2|^{2}\Big)^{1/2}\Big\|_{q'} \lesssim |S_0|^{\frac1{q'}} \langle f_2\rangle_{3S_0,q'}.\]
By the embedding $\ell^p \subseteq \ell^q$ for $p \leq q$ in the $k$-sum and \eqref{eq:key estimate}, we estimate
\[\Big( \sum_{\substack{k\in\digamma}} \sum_{Q\in\mathcal{Q}} |Q|^{1-\frac qp} (\beta^{k+\mathrm{\nu},n}_{1,Q,p})^q \Big)^{1/q} \\ \le \Big( \sum_{Q\in\mathcal{Q}} |Q|^{1-\frac qp} (\sum_{\substack{k\in\digamma}} (\beta^{k+\mathrm{\nu},n}_{1,Q,p})^p)^{q/p} \Big)^{1/q}
\lesssim |S_0|^{\frac1q} \alpha_1.
\]
Summing over $n\ge 0$ and $\mathrm{\nu}\in{\mathbb {Z}}$ using \eqref{eq:convenience} concludes the argument. \qed

\subsection{The case \texorpdfstring{$q\ge 2$}{q>=2}: Proof of Proposition \ref{prop:badpartq>2}}\label{sec:badpart q>2}

We decompose $f_2=g_2+\sum_{Q' \in {\mathcal {Q}}} b_{2,Q'}$, as in \eqref{good2}, \eqref{b2Q},
so that \eqref{eq:badpart pre} is bounded by $I+II$, where
\begin{align*}
I &= \sum_{\mathrm{\nu}\in{\mathbb {Z}}} \sum_{\ell>0} \sum_{k\in\digamma} \sum_{B\in{\mathcal {B}}_{\ell-k}} \Big| \Big\langle T_k^{(\ell-k)} P_k \Big[ \sum_{\substack{Q\in\mathcal{Q},\\ Q\subsetneq B}} b_{1,Q}^{k+\mathrm{\nu}} \Big], (P_k g_2){\mathbbm 1}_{3B} \Big\rangle \Big|,
\\
II &= \sum_{\mathrm{\nu}\in{\mathbb {Z}}} \sum_{\ell>0} \sum_{k\in\digamma} \sum_{B\in{\mathcal {B}}_{\ell-k}} \Big| \Big\langle T_k^{(\ell-k)} P_k \Big[ \sum_{\substack{Q\in\mathcal{Q},\\ Q\subsetneq B}} b_{1,Q}^{k+\mathrm{\nu}} \Big], {\mathbbm 1}_{3B} \sum_{\substack{Q' \in {\mathcal {Q}}}} P_kb_{2,Q'} \Big\rangle \Big|.
\end{align*}
\subsubsection{The term \texorpdfstring{$I$}{I}}
By the Cauchy--Schwarz inequality,
\begin{equation*}
|I| \leq \sum_{\mathrm{\nu}\in {\mathbb {Z}}} \sum_{\ell > 0}\sum_{\substack{k \in \digamma}} \sum_{B \in {\mathcal {B}}_{\ell-k}}
\Big\| T_k^{(\ell-k)} \big[ P_k{\mathbb {D}}_{k+\mathrm{\nu}} \sum_{\substack{Q \in {\mathcal {Q}}, \\ Q \subsetneq B}}b_{1,Q}^{k+\mathrm{\nu}} \big] \Big\|_{2} \| (P_k g_{2} ){\mathbbm 1}_{3B} \|_{2}.
\end{equation*}
By a standard localization argument, for $p \leq 2 \leq q$,
\begin{equation*}
\| \varphi m(2^k \cdot) \ast \widehat{\Psi_\ell} \|_{M^{p \to 2}} \lesssim 2^{\ell d (\frac{1}{2}-\frac{1}{q})} \| \varphi m(2^k \cdot) \ast \widehat{\Psi_\ell} \|_{M^{p \to q}} = 2^{-\ell d (\frac{1}{p}-\frac{1}{2})} A_{p,p,q}^{k,\ell}
\end{equation*}
and thus by the scaling relation \eqref{eq:Tkl scaling},
\begin{equation} \label{eq:Tkellminusk}
\|T_k^{(\ell-k)} \|_{p\to 2} \leq 2^{-(\ell-k) d (\frac{1}{p}-\frac{1}{2})} A_{p,p,q}^{k, \ell}.
\end{equation}
By \eqref{eq:Lpk1k2} and Lemma \ref{lemma:lift lp to lq} with the exponent pair $(p,2)$,
\begin{align}
\notag &\Big\| P_k{\mathbb {D}}_{k+\mathrm{\nu}}\Big[\sum_{\substack{Q \in {\mathcal {Q}}, \\ Q \subsetneq B}} b_{1,Q}^{k+\mathrm{\nu}}\Big] \Big\|_p
\lesssim 2^{-|\mathrm{\nu}|/p' } \Big\| \sum_{\substack{Q \in {\mathcal {Q}}, \\ Q \subsetneq B}} b_{1,Q}^{k+\mathrm{\nu}} \Big\|_p
\\
\notag &\lesssim 2^{-|\mathrm{\nu}|/p'}
|B|^{\frac{1}{p}-\frac{1}{2}} \Big( \sum_{\substack{Q \in {\mathcal {Q}}, \\ Q \subsetneq B}} |Q|^{1-\frac 2p} \|b_{1,Q}^{k+\mathrm{\nu}} \|_{p}^2 \Big)^{1/2}
\\ \label{eq:fixedkplusm} &\lesssim 2^{-|\mathrm{\nu}|/p'} 2^{(\ell-k)d(\frac 1p-\frac 12)}\Big( \sum_{\substack{Q \in {\mathcal {Q}}, \\ Q \subsetneq B}} |Q|^{1-\frac 2p} (\beta_{1,Q,p}^{k+\mathrm{\nu}})^2 \Big)^{1/2},
\end{align}
where for the last line we used that $|B|=2^{(\ell-k)d}$ and Lemma \ref{lemma:b_Q to beta_Q}.
Combining \eqref{eq:Tkellminusk} and \eqref{eq:fixedkplusm}
we obtain
\begin{multline*}
|I| \leq \sum_{\mathrm{\nu}\in {\mathbb {Z}}} 2^{-|\mathrm{\nu}|/p'} \sum_{\ell>0} \sum_{\substack{k \in \digamma}} A_{p,p,q}^{k,\ell}
\quad \times \\
\sum_{B \in \mathcal{B}_{\ell-k}} \Big( \sum_{\substack{Q \in {\mathcal {Q}}, \\ Q \subsetneq B}} |Q|^{1-\frac 2p} (\beta_{1,Q,p}^{k+\mathrm{\nu}})^2 \Big)^{1/2} \| (P_kg_2) {\mathbbm 1}_{3B} \|_2
\end{multline*}
and after applying the Cauchy--Schwarz inequality to the sum over $B$ we get
\[
| I| \leq {\mathcal {A}}_{p,p,q} \sum_{\mathrm{\nu}\in {\mathbb {Z}}} 2^{-|\mathrm{\nu}|/p'} \sum_{\substack{k \in \digamma}}
\Big(\sum_{\substack{Q \in {\mathcal {Q}}}} |Q|^{1-\frac 2p} (\beta_{1,Q,p}^{k+\mathrm{\nu}})^2 \Big)^{1/2} \| P_kg_2 \|_2.
\]
Applying the Cauchy--Schwarz inequality in $k$ and using $\sum_{\mathrm{\nu}\in{\mathbb {Z}}} 2^{-|\mathrm{\nu}|/p'}\lesssim 1$,
\[
\le {\mathcal {A}}_{p,p,q}
\Big(\sum_{\substack{Q \in {\mathcal {Q}}}} \sum_{k>-j_0} |Q|^{1-\frac 2p} (\beta_{1,Q,p}^{k})^2 \Big)^{1/2} \Big\|\Big( \sum_{k\in{\mathbb {Z}}} |P_kg_2|^2\Big)^{1/2}\Big\|_2.
\]
We apply \eqref{eq:l2 sum betakQ} to get
\begin{equation}\label{eq:bad sumQk beta2}
\Big(\sum_{\substack{Q \in {\mathcal {Q}} }} |Q|^{1-\frac 2p} \sum_{k>-j_0} (\beta_{1,Q,p}^{k})^2 \Big)^{1/2} \le
\Big(\sum_{\substack{Q \in {\mathcal {Q}} }} |Q|^{1-\frac 2p} |Q|^{\frac{2}{p}} (\alpha_1)^2 \Big)^{1/2}\lesssim \alpha_1|S_0|^{\frac12}.
\end{equation}
By the almost orthogonality of the $P_k$ and Lemma \ref{f2CZdec} (ii),
\[\Big(\sum_{k\in{\mathbb {Z}}} \|P_k g_2\|_2^2\Big)^{1/2} \lesssim \|g_2 \|_{L^2(3S_0)}\lesssim |S_0|^{\frac12} \alpha_2.\]
In summary we obtain
$|I|\lesssim {\mathcal {A}}_{p,p,q} |S_0| \alpha_1\alpha_2$ as claimed.

\subsubsection{The term \texorpdfstring{$II$}{II}}
By the scaling relation \eqref{eq:Tkl scaling} with $r=p$, \eqref{eq:Lpk1k2} and using $|B|=2^{d(\ell-k)}$,
\begin{flalign*}&
\Big|\Big\langle T_k^{(\ell-k)} P_k \Big[ \sum_{\substack{Q\in\mathcal{Q},\\ Q\subsetneq B}} b_{1,Q}^{k+\mathrm{\nu}} \Big], {\mathbbm 1}_{3B} \sum_{\substack{Q' \in {\mathcal {Q}}}} P_kb_{2,Q'} \Big\rangle\Big| &
\end{flalign*}
\[
\lesssim 2^{-|\mathrm{\nu}|/p'} |B|^{-(\frac1p-\frac1q)} A_{p,p,q}^{k,\ell} \Big\|\sum_{\substack{Q\in\mathcal{Q},\\ Q\subsetneq B}} b_{1,Q}^{k+\mathrm{\nu}}\Big\|_p
\Big\|\sum_{Q' \in {\mathcal {Q}}} P_k b_{2,Q'}\Big\|_{L^{q'}(3B)}.
\]
By disjointness of the cubes in ${\mathcal {Q}}$, Lemma \ref{lemma:lift lp to lq} and Lemma \ref{lemma:b_Q to beta_Q},
\begin{equation}\label{eq:q>2 Holder in sum 1}
\Big\| \sum_{\substack{Q\in {\mathcal {Q}},\\ Q \subsetneq B}} b_{1,Q}^{k+\mathrm{\nu}} \Big\|_{p} \lesssim |B|^{\frac{1}{p}-\frac{1}{2}} \Big( \sum_{\substack{Q\in {\mathcal {Q}},\\ Q \subsetneq B}} |Q|^{1-\frac 2p}\, (\beta_{1,Q,p}^{k+\mathrm{\nu}})^2 \Big)^{1/2}
\end{equation}
and similarly, applying H\"older's inequality as in the proof of Lemma \ref{lemma:lift lp to lq} we also get
\begin{align*} &\Big\| \sum_{\substack{Q'\in {\mathcal {Q}}}} P_kb_{2,Q'}\Big\|_{L^{q'}(3B)}
\lesssim |B|^{\frac1{q'}-\frac 12}
\Big( \sum_{\substack{Q'\in {\mathcal {Q}},\\Q'\subseteq 3B}} |Q'|^{1-\frac{2}{q'}}
\|P_kb_{2,Q'}\|_{q'}^{2} \Big)^{1/2}.
\end{align*}
Combining these estimates we get
\begin{multline}
|II|\lesssim \sum_{\mathrm{\nu}\in {\mathbb {Z}}} 2^{-|\mathrm{\nu}|/p'} \sum_{\substack{k\in \digamma}} \sum_{\ell>0} A_{p,p,q}^{k,\ell} \,\times\\\sum_{B\in {\mathcal {B}}_{\ell-k}}
\Big( \sum_{\substack{Q\in {\mathcal {Q}},\\ Q \subsetneq B}} |Q|^{1-\frac 2p}\, (\beta_{1,Q,p}^{k+\mathrm{\nu}})^2 \Big)^{1/2}
\Big( \sum_{\substack{Q'\in {\mathcal {Q}}:\\Q'\subseteq 3B} }|Q'|^{1-\frac{2}{q'}}
\|P_kb_{2,Q'}\|_{q'}^{2} \Big)^{1/2}
\end{multline}
and by the Cauchy--Schwarz inequality applied to the sums over $B$ and $k$,
\[
|II|\lesssim {\mathcal {A}}_{p,p,q} \Big( \sum_{\substack{Q\in {\mathcal {Q}}}} |Q|^{1-\frac 2p}\,\sum_{k>-j_0} (\beta_{1,Q,p}^{k})^2 \Big)^{1/2} \Big( \sum_{\substack{Q'\in {\mathcal {Q}}} } \sum_{k\in{\mathbb {Z}}} |Q'|^{1-\frac{2}{q'}}
\|P_kb_{2,Q'}\|_{q'}^{2} \Big)^{1/2}
\]
Using $q'\le 2$ and Lemma \ref{f2CZdec} (i),
\[ \Big(\sum_{k\in{\mathbb {Z}}}
\|P_kb_{2,Q'}\|_{q'}^{2}\Big)^{1/2} \lesssim \Big\| \Big(\sum_{k\in{\mathbb {Z}}}|P_k b_{2,Q' }|^2\Big)^{1/2}\Big\|_{q'} \lesssim \|b_{2,Q'}\|_{q'} \lesssim |Q'|^{\frac1{q'}}\alpha_2 \]
and thus
\[\Big( \sum_{\substack{Q'\in {\mathcal {Q}}}}\sum_{k\in{\mathbb {Z}}} |Q'|^{1-\frac{2}{q'}}
\|P_kb_{2,Q'}\|_{q'}^{2} \Big)^{1/2} \lesssim
\Big( \sum_{\substack{Q'\in {\mathcal {Q}}}} |Q'|
\alpha_2^{2} \Big)^{1/2} \lesssim\alpha_2|S_0|^{\frac12} .
\]
Together with \eqref{eq:bad sumQk beta2} this gives
$|II|\lesssim {\mathcal {A}}_{p,p,q} |S_0|\alpha_1\alpha_2$ as desired. \qed

\section{Applications} \label{sec:applications}

\subsection{The classes \texorpdfstring{$\mathrm{FM}(a,b)$}{Miy(a,b)}} {We prove the positive result in Theorem \ref{miyachi-thm}.}
Condition \eqref{Miyachi-deriv-cond} can be reformulated as
\[ \sum_{|\alpha|\le n} \|\partial^\alpha[\phi m(t\cdot) ]\|_\infty\lesssim_n t^{-b+na}
\] for all $n\in {\mathbb {N}}$ and for all $t> 1/8$ (and we have $\phi m(t\cdot)=0$ for small $t$).
Using a standard interpolation result for Sobolev spaces (\cite[Chapter 5.4]{bennett-sharpley}) we also get for any $s\ge 0$
\begin{equation}\label{eq:mlocalBsinf1 bd}
\|\phi m(t\cdot) \|_{B^s_{\infty,1}}\lesssim_s t^{-b+sa}\,,
\end{equation}
for all $t>1/8$,
with the implicit constant independent of $t$.

We now verify $m(D)\in {\mathrm{Sp}}(\rho_1,\rho_2)$ for $(1/\rho_1,1/\rho_2)$ on the edge $[P_3,P_4]$, thus satisfying $1/\rho_1-1/\rho_2'=b/da$. Here $P_4=(1/p_4,1/2)$ with $1/p_4=1/2+b/ad$.
The results in the remaining parts of the trapezoid then follow by H\"older's inequality.

Observe that $\tfrac12\le\tfrac1{\rho_1}\le \tfrac1{p_4}$.
First, we claim that the cases $\rho_1=2$, $\rho_1=p_4$ follow from Theorem \ref{thm:qge2}. By duality we only need to discuss the case $\rho_1=p_4$.
Then, taking $b=ad(1/p_4-1/2)$ and $s=d(1/p_4-1/2)$ in \eqref{eq:mlocalBsinf1 bd} we get
\begin{equation}\label{eq:FMendpt cond}
\sup_{t>0} \|\phi m(t\cdot)\|_{B^{d(1/p_4-1/2)}_{\infty,1}} < \infty.
\end{equation}
Next, using the compact support of $\phi$ and some calculations about the contributions away from the support of $\phi$ we have
\begin{equation}\label{eq:FMendpt cond2}
\sup_{t>0}\|\phi m(t\cdot)\|_{B^{d(1/p_4-1/2)}_{1} (M^{p_4\to 2}) }
\lesssim \|m\|_\infty+\sup_{t>0}\|\phi m(t\cdot)\|_{B^{d(1/p_4-1/2)}_{\infty, 1} },
\end{equation}
which establishes that $m\in {\mathrm{Sp}}(p_4,2)$, as desired.
In the remaining case $\tfrac12<\tfrac1{\rho_1}<\tfrac1{p_4}$ we use
Theorem
\ref{thm:multq>2}
(with the parameters $p=\rho_1$,
$q'=\rho_2$; note that $\rho_2< 2$).
To this end it suffices to verify that
\begin{align}
\label{rho1rho2'}
&\sup_{t>0}\|\phi m(t\cdot)\|_{B_1^{d(1/\rho_1-1/\rho_2')}(M^{\rho_1\to \rho_2'})}<\infty,
\\
\label{rho1-impr} &\sup_{t>0}\|\phi m(t\cdot)\|_{B^{d(1/\rho_2-1/2)}_{\infty,1}} <\infty.
\end{align}
Note that \eqref{rho1-impr} follows from \eqref{eq:FMendpt cond} since $B^{d(1/p_4-1/2)}_{\infty,1}\hookrightarrow B^{d(1/\rho_2-1/2)}_{\infty,1} $.
Finally, \eqref{rho1rho2'} follows from \eqref{eq:FMendpt cond2} and \eqref{eq:FMendpt cond}. This is because $1/\rho_1-1/\rho_2'=1/p_4-1/2$ and
since $M^{p_4\to 2}=M^{2\to p_4'} \hookrightarrow M^{\rho_1\to \rho_2'}$ by interpolation
(observe that $(1/\rho_1, 1/\rho_2')$ lies on the line between $(1/p_4,1/2)$ and $(1/2, 1/p_4')$).

\subsection{{\it Oscillatory multipliers}}
\label{sec:oscmultsec}
{We next turn to the proof of the positive result in Theorem \ref{thm:oscmult}.} Consider the oscillatory multipliers $m_{a,b}$. It is our goal to establish the endpoint ${\mathrm{Sp}}(p_1,p_1)$ bound, for $1/p_1=1/2+b/(ad)$, {as the remaining bounds then follow by H\"older's inequality}. Since
$m_{a,b}$ belongs to $\mathrm{FM}(a,b)$ we have
\[\sup_{t>0} \|\phi m_{a,b} (t\cdot)
\|_{B^{d(1/p_1-1/2)}_1(M^{2\to 2})}<\infty\] and in order to apply Theorem \ref{thm:multq>2} it remains to verify that
\begin{equation} \label{eq:oscmult-curv}
\sup_{t>0}\big[\| \phi m_{a,b}(t\cdot)*\widehat{\Phi_0}\|_{M^{p_1\to p_1'}}+ \sum_{\ell>0} 2^{\ell d(1/p_1-1/p_1')}\|\phi m_{a,b}(t\cdot)*\widehat{\Psi_\ell} \|_{M^{p_1\to p_1'}}\big]<\infty.
\end{equation}

We sketch the argument; a similar calculation appears in \cite[Chapter 7.2.2]{BRS}.
Note that $\phi m_{a,b}(t\cdot)=0$ for $t\ll 1$ and that the inequality is trivial for $t\approx 1$. For $t\gg 1$ we have
\[K_t(x):={\mathcal {F}}^{-1} [ \phi m_{a,b}(t\cdot)](x)=(2\pi)^{-d}\int \frac{\phi(\xi)}{t^b|\xi|^b} e^{i t^a|\xi|^a-i\inn x\xi} \, \mathrm{d} \xi\,.
\]
For $a\ne 1$ the Hessian of $\xi\mapsto |\xi|^a$ has full rank. Thus we get by stationary phase the bound
$|K_t(x)| \lesssim t^{-b-ad/2} $ if $|x|\approx t^a$, moreover $|K_t(x)|\lesssim_N t^{-N-b} $ if $|x|\ll t^a$ and $|K_t(x)|\lesssim_N t^{-b} |x|^{-N} $ for $|x|\gg t^a$ for all $N \geq 0$. These estimates give $M^{1\to\infty}$ bounds for $\phi m_{a,b}(t\cdot) *\widehat {\Psi_\ell}$ while we also have the trivial bound $O(t^{-b})$ for the $M^2$ norm. Interpolation shows that (for suitable constants $c(a)<C(a)$)
\[\|(\phi m_{a,b}(t\cdot))*\widehat{\Psi_\ell} \|_{M^{p\to p'} } \lesssim_N \begin{cases} t^{-b-da(\frac 1p-\frac 12)}
&\text{ if } c(a) t^a \le 2^{\ell} \le C(a) t^a
\\
t^{-b-Na (\frac 1p-\frac 12)} &\text { if } 2^\ell \le c(a) t^a
\\
t^{-b} 2^{ -\ell N(\frac 1p-\frac 12)} &\text { if } 2^\ell \ge C(a) t^a
\end{cases}
\] and since $1/p_1= 1/2+b/(ad)$ the inequality \eqref{eq:oscmult-curv} follows.

\subsection{The results for radial multipliers}
\label{sec:cor-radial}
We use the Fefferman--Stein argument \cite{fefferman69} based on the $L^2$ Stein--Tomas restriction theorem. For radial $m(\xi)=a(|\xi|)$, $1\le p\le \frac{2(d+1)}{d+3}$, by Plancherel's theorem we may write $\|m(D)f\|_2^2$ as a constant times
\[ \int |m(\xi)\widehat{f}(\xi)|^2 \, \mathrm{d}\xi = \int_0^\infty |a(\rho)|^2 \rho^{d} \int_{\mathbb{S}^{d-1}} |\widehat{f}(\rho\theta)|^2 \, \mathrm{d}\sigma(\theta) \frac{\, \mathrm{d} \rho}\rho. \]
Therefore, by the Stein--Tomas theorem applied to the integral over $\mathbb{S}^{d-1}$,
\begin{equation}\label{eq:Feff-reduction}\|m\|_{M^{p\to 2}} \lesssim \Big(\int_0^\infty |a(\rho)|^2 \rho^{2d(\frac 1p-\frac 12) } \frac{\, \mathrm{d} \rho}{\rho}\Big)^{1/2}=
c_d \Big(\int |m(\xi)|^2 |\xi|^{2d(\frac 1p-1)}\, \mathrm{d} \xi \Big)^{1/2}
\end{equation}
where $c_d$ is the surface measure of $\mathbb{S}^{d-1}$ raised to the power $-\frac12$.

\subsubsection{Proof of Corollary \ref{cor:studiaradial}}
We use Theorem \ref{thm:qge2} in conjunction with \eqref{eq:Feff-reduction}. We set $m_t(\xi)=\phi(|\xi|)h(t|\xi|)
$.
We let $\, \mathrm{d}\mu_p(\xi)= |\xi|^{2d(\frac 1p-1) } \, \mathrm{d}\xi$.
It suffices to prove the estimate
\begin{equation}\label{eq:Cor-redux}
\sup_{t>0} \|\phi(\cdot)h(t|\cdot|) \|_{B^{\alpha}_{1} (L^2(\, \mathrm{d}\mu_p))} \lesssim \sup_{t>0} \|\phi h(t\cdot)\|_{B^{\alpha }_{2,1}({\mathbb {R}})}
\end{equation}
for $\alpha>0$, $1\le p<2$ and apply it for $\alpha=d(1/p-1/2)$ and $1< p\le \frac{2(d+1)}{d+3}$.

We are assuming that $\phi$ is supported in $(1/2,2)$ and it is convenient to choose $\chi$ to be a radial function supported in $\{\xi: 1/4<|\xi|<4\}$ such that $0\le\chi\le 1$ and $\chi(\xi)=1$ for $1/3\le |\xi|\le 3.$
For $\ell>0$, let
\begin{align*} I_{\ell,t}&= \Big(\int \big|(\phi(|\cdot|) h(t|\cdot|)) *\widehat{\Psi_\ell} (\xi) \big |^2 \chi(\xi) |\xi|^{2d(\frac 1p-1)}\, \mathrm{d}\xi \Big)^{1/2},
\\
II_{\ell,t}&= \Big(\int \big|(\phi(|\cdot|) h(t|\cdot|)) *\widehat{\Psi_\ell} (\xi)\big |^2 (1-\chi(\xi)) |\xi|^{2d(\frac 1p-1)}\, \mathrm{d}\xi \Big)^{1/2}
\end{align*}
and let $I_{0,t}$, $II_{0,t}$ be the analogous expressions with $\Phi_0$ in place of $\Psi_\ell$.

First note that $|\xi|^{2d(1/p-1)}\approx 1$ on the support of $\chi$, and therefore
\[\sum_{\ell\ge 0} 2^{\ell \alpha} I_{\ell,t} \le C_\alpha
\|\phi(|\cdot|)h(t|\cdot|) \|_{B^\alpha_{2,1} ({\mathbb {R}}^d)}.\]
We observe the inequality
\begin{equation}\label{eq:norm Rd to R}
\|\chi g(|\cdot|) \|_{W^m_2({\mathbb {R}}^d)} \le C_m\|g\|_{W^m_2({\mathbb {R}})}
\end{equation}
which follows by application of the product and chain rules; by real interpolation we get for all $\alpha>0$
\[\|\chi g(|\cdot|) \|_{B^\alpha_{2,1}({\mathbb {R}}^d)} \lesssim_\alpha C_\alpha \|g\|_{B^\alpha_{2,1}({\mathbb {R}})}\] and hence,
\[
\sum_{\ell\ge 0} 2^{\ell \alpha} I_{\ell,t} \lesssim_\alpha \|\phi h(t\cdot) \|_{B^\alpha_{2,1} ({\mathbb {R}})}.\]
It remains to estimate the term $II_{\ell,t}$.
Note that for $|\xi|\ge 3$ and $N>d$,
\begin{align*} \big|\phi(|\cdot|)h(t&|\cdot|) )*\widehat{\Psi_\ell}(\xi) \big|
\le \int_{|y|\ge |\xi|-2} |\phi(|\xi-y|)h(t|\xi-y|)| 2^{\ell d } (2^\ell|y|)^{-N} \, \mathrm{d} y
\\
&\lesssim \Big(\int_{|y|\ge |\xi|-2} |\phi( |\xi-y |)h(t|\xi-y|)|^2\, \mathrm{d} y\Big)^{1/2} 2^{\ell (d-N)} |\xi|^{\frac{d-2N}{2}}
\end{align*}
by the Cauchy--Schwarz inequality. Since $2d(\tfrac1p-1)$ is negative, we get, for $N>d$,
\begin{multline*}\Big(\int_{|\xi|\ge 3} \big| (\phi(|\cdot|)h(t|\cdot|) )*\widehat{\Psi_\ell}(\xi) \big| ^2|\xi|^{2d(\frac1p-1)} \, \mathrm{d}\xi\Big)^{1/2}
\\
\lesssim_N 2^{\ell(d-N)} \|\phi(|\cdot|) h(t\cdot)\|_{L^2({\mathbb {R}}^d)}
\lesssim_N 2^{\ell(d-N)} \|\phi h(t\cdot)\|_{L^2({\mathbb {R}})}.
\end{multline*}
Similarly, we have
\begin{align*}
&\Big(\int_{|\xi|\le \frac 13} \big| (\phi(|\cdot|)h(t|\cdot|) )*\widehat{\Psi_\ell}(\xi) \big|^2 |\xi|^{2d(\frac1p-1)} \, \mathrm{d}\xi\Big)^{\frac 12}\\
&\lesssim_N \|\phi h(t\cdot)\|_{L^2({\mathbb {R}})}
\Big(\int_{|\xi|\le \frac 13} \int_{|y|\ge \frac 16}
2^{2\ell d} |y2^\ell|^{-2N} \, \mathrm{d} y\, |\xi|^{-\frac{2d}{p'} }\, \mathrm{d}\xi\Big)^{\frac 12}
\\&\lesssim_N 2^{\ell(d-N)} \|\phi h(t\cdot)\|_{L^2({\mathbb {R}})}.
\end{align*}
Altogether we get for $p<2$ (i.e. $-2d/p'>-d$),
\[ 2^{\ell\alpha} II_{\ell,t} \lesssim 2^{\ell(\alpha+d-N)} \|\phi h(t\cdot)\|_{L^2({\mathbb {R}})}.\]
The same applies for the term $II_{0,t}$. Combining the estimates we obtain \eqref{eq:Cor-redux} which concludes the proof of the corollary.
\qed

\subsubsection{Proof of Theorem \ref{cor:BR}}

Let $u_\delta(\xi)=h_\delta(|\xi|)=\chi(\delta^{-1}(1-|\xi|))$ in what follows.
To use the notation in our main theorems we are setting
$p_1=p$ and $p_2=q'$.

\subsubsection*{The case $p\le \frac{2(d+1)}{d+3}$ } We are seeking to prove a ${\mathrm{Sp}}(p, q')$ bound, under the assumption $1<p\le \frac{2(d+1)}{d+3}$ and
$\frac 1{q'}\le \frac{d+1}{d-1}\frac 1p- \frac 2{d-1}$
(which is equivalent with the condition
$\frac 1q\ge \frac{d+1}{d-1} \frac{1}{p'}$). Since ${\mathrm{Sp}}(p_1,p_2) \subset {\mathrm{Sp}}(p_1,p_3)$ for $p_3\ge p_2$ we just need to consider the endpoint line with
$\frac 1{q'}= \frac{d+1}{d-1}\frac 1p- \frac 2{d-1}$; note that under this assumption we have
$2\le q < \infty$ for $1 < p\le \frac{2(d+1)}{d+3}$ and, moreover, $q=2$ if and only $p=\frac{2(d+1)}{d+3}$. In the latter case we use Theorem \ref{thm:qge2} while for $1<p<\frac{2(d+1)}{d+3}$ we have $q>2$ and use Theorem \ref{thm:multq>2}.

In order to establish the assertion we have to prove, for $1<p\le \frac{2(d+1)}{d+3}$ and
$\frac 1{q'}=\frac{d+1}{d-1}\frac 1p- \frac 2{d-1}$
the inequality
\begin{equation} \label{eq:first-cond}
\sum_{\ell>0} 2^{\ell d(\frac 1p-\frac 1q)}
\|u_\delta*\widehat {\Psi_\ell} \|_{M^{p\to q}} \lesssim \delta^{-d(\frac 1p-\frac 12)+\frac 12}.
\end{equation}
Furthermore for $p$ in the open range $1<p<\frac{2(d+1)}{d+3}$ (where $q'<2$) we use Theorem \ref{thm:multq>2} and also have to prove,
for suitable $q'<r\le 2$,
\begin{equation}\label{eq:second-cond}
\sum_{\ell>0} 2^{\ell d(\frac{1}{q'}-\frac 1r)} \|u_{\delta}*\widehat{\Psi_\ell}\|_{M^{r\to r} }\lesssim \delta^{-d(\frac 1p-\frac 12)+\frac 12}.
\end{equation}

From the $L^2$ restriction theorem and \eqref{eq:Feff-reduction} we get
\begin{equation}\label{eq:ST} \|u_\delta*\widehat {\Psi_\ell}\|_{M^{p_\circ\to 2}} \lesssim_N \delta^{1/2} \min\{1, (2^\ell \delta)^{-N}\}, \quad
p_\circ= \tfrac{2(d+1)}{d+3}.
\end{equation}
By stationary phase and integration by parts arguments we get
\begin{equation}\label{eq:statphase}\|u_\delta*\widehat {\Psi_\ell}\|_{M^{1\to \infty}} \lesssim 2^{-\ell\frac{d-1}{2} } \delta
\min\{1, (2^\ell \delta)^{-N}\}.
\end{equation}
This implies
\begin{multline} \label{eq:mdeltainterpol} 2^{\ell d(\frac 1p-\frac 1q)} \delta^{d(\frac 1p-\frac 12)-\frac 12}
\|u_\delta*\widehat{\Psi_\ell}(D)\|_{L^p\to L^q}\lesssim (2^\ell\delta)^{\alpha(q)} \min\{1, (2^\ell \delta)^{-N}\}, \\ \text{ for }
2\le q\le \infty, \,\, \tfrac 1q=\tfrac {d+1}{d-1} \tfrac{1}{p'}, \,\, \alpha(q)>0.
\end{multline}
The relation between $p$ and $q$ can be rewritten as
$\frac 1{q'}=\frac{d+1}{d-1}\frac 1p- \frac 2{d-1}$.
Thus, given
\eqref{eq:mdeltainterpol}
we obtain \eqref{eq:first-cond} for $1\le p< \frac{2(d+1)}{d+3}$ after summing in $\ell$.

To verify \eqref{eq:mdeltainterpol} observe that
by \eqref{eq:ST} we have \eqref{eq:mdeltainterpol}
for $(\frac 1p, \frac 1q) $ equal to $(\frac{d+2}{2(d+1)}, \frac 12) $, with $\alpha (2)=d(\frac 1{p_\circ}-\frac 12)=\frac{d}{d+1} $.
By \eqref{eq:statphase} we have \eqref{eq:mdeltainterpol}
for $(\frac 1p, \frac 1q) $ equal to $(1,0)$, with $\alpha(\infty)=\frac{d+1}{2}$.
Thus by interpolation we get
\eqref{eq:mdeltainterpol} for all $2\le q\le \infty$, with $\frac{1}{q}=\frac{d+1}{d-1}\frac{1}{p'}$, and $\alpha(q)>0$; more precisely
$\alpha(q)=\frac {2d}{q(d+1)}+ (d+1) (\frac 12-\frac 1q)$.

We still have to verify \eqref{eq:second-cond}, but only when $p<\frac{2(d+1)}{d+3} $.
Observe that $q'<2$ for $p<\frac{2(d+1)}{d+3}$. Choose $r$ with
$q'<r\le 2$, and $r$ very close to $q'$.
Here it suffices to use classical non-endpoint estimates which give
\begin{multline} \label{eq:L-est}
2^{\ell d(\frac{1}{q'}-\frac 1r)} \|u_\delta*\widehat{\Psi_{\ell}} (D)\|_{L^r\to L^r} \le\\ C_{N,\varepsilon} 2^{\ell d(\frac{1}{q'}-\frac 1r)} \min\{1, (2^\ell \delta)^{-N} \}\begin{cases}
\delta^{-d(\frac 1r-\frac 12)+\frac 12-\varepsilon} &\text{ if } 1\le r\le \tfrac {2(d+1)}{d+3}
\\
\delta^{-(d+1)(\tfrac 1r-\tfrac 12) -\varepsilon} &\text{ if $\tfrac {2(d+1)}{d+3} \le r\le 2$}
\end{cases},
\end{multline}
with a better result in two dimensions:
\begin{multline} \label{eq:Lr-est2d}
2^{\ell 2(\frac{1}{q'}-\frac 1r)} \|u_\delta*\widehat{\Psi_{\ell}} (D)\|_{L^r({\mathbb {R}}^2)\to L^r({\mathbb {R}}^2)} \le\\ C_N 2^{\ell 2(\frac{1}{q'}-\frac 1r)} \min\{1, (2^\ell \delta)^{-N} \}\begin{cases}
\delta^{-2(\frac 1r-\frac 12)+\frac 12-\varepsilon} &\text{ if } 1\le r<4/3
\\
\delta^{-\varepsilon} &\text{ if $4/3 \le r\le 2$}
\end{cases}.
\end{multline}

In dimension $d\ge 3$ we have to show that for $r$ sufficiently close to $q'$ the right-hand side of \eqref{eq:L-est}
is
dominated by \begin{equation}\label{eq:Thm15bdsec}
\min\{(2^{\ell}\delta )^{-\epsilon_1},
(2^{\ell}\delta)^{\epsilon_1} \} \delta^{-d(\frac 1p-\frac 12)+\frac 12}
\end{equation}
under the assumption that $\frac 1{q'}=\frac{d+1}{d-1}\frac 1p- \frac 2{d-1}$, for some $\epsilon_1>0$ depending on $p,q $. Since $p<q'$ this is immediate for $q'<\frac{2(d+1)}{d+3}$.
When $2>q'\ge \frac{2(d+1)}{d+3}$ the goal is accomplished once \[d(\tfrac 1p-\tfrac 12)-\tfrac 12-(d+1)(\tfrac{1}{q'}-\tfrac 12)>0 , \quad \text{ for
$\tfrac 1p=\tfrac{d-1}{(d+1)q'}+\tfrac 2{d+1}$}
\] and the given range of $q'$. The displayed inequality holds for all $q'\in [\frac{2(d+1)}{d+3},2)$ if and only if $q'>\frac{6d+2}{5d+1}$. Since $\frac{2(d+1)}{d+3} >
\frac{6d+2}{5d+1}$ for $d\ge 3$ we get the bound \eqref{eq:Thm15bdsec}
for $r$ close to $q'$ in dimension $d\ge 3$.

In dimension $d=2$, the previous argument is not strong enough to cover the range $\frac 65\le q'\le \frac{14}{11}$. We need to use the better estimate \eqref{eq:Lr-est2d}. Now the bound \eqref{eq:Thm15bdsec} is immediate for $q'<4/3$ and if $4/3\le q'<2$ and $\frac 1p=\frac{1}{3q'}+\frac 23$ we have
$2(\frac 1p-\frac 12)-\frac 12= \frac{2}{3q'}-\frac 16>0$ so that the desired bound follows in this case as well.

We have thus checked the assumptions of Theorem \ref{thm:qge2} (when $p=\frac{2(d+1)}{d+3}$) and Theorem \ref{thm:multq>2} (when $p<\frac{2(d+1)}{d+3} $) and the theorem is proved for the case $p\le \frac{2(d+1)}{d+3}$.

\subsubsection*{The case $\frac{2(d+1)}{d+3}<p<\frac{2(d+2)}{d+4}$}
The proof relies on an inequality in \cite[Proposition 2.4]{ChoKimLeeShim2005}, which in its dual formulation says
\begin{equation}\label{eq:choetal}
\|u_\delta\|_{M^{\rho\to q}} \lesssim \delta^{-d(\frac 1q-\frac 12)+\frac 12}, \,\,\tfrac 1{q'}=\tfrac{d+1}{d-1}\tfrac 1\rho-\tfrac{2}{d-1},\,\, \tfrac{2(d+1)}{d+3} \le \rho<\tfrac{2(d+2)}{d+4}.
\end{equation}
By averaging we can replace $u_\delta$ with $u_\delta*\widehat{\Psi_\ell} $ or $u_\delta*\widehat {\Phi_0}$ in \eqref{eq:choetal}, and after an additional interpolation with \eqref{eq:ST} (the case with $r=\frac{2(d+1)}{d+3}$ and $q=2$) we also get
\begin{equation}\label{eq:choetalPsi}
\|u_\delta*\widehat{\Psi_\ell}\|_{M^{r\to q}} \lesssim_{N,q}
\min\{1,(2^\ell \delta)^{-N} \}\delta^{-d(\frac 1q-\frac 12)+\frac 12},
\end{equation}
first for $ \tfrac 1{q'}=\tfrac{d+1}{d-1}\tfrac 1r-\tfrac{2}{d-1}$, $\tfrac{2(d+1)}{d+3} \le r<\tfrac{2(d+2)}{d+4}.$ Then, using the compact support of $u_\delta$ and the auxiliary Lemma \ref{lem:r1r2} we also get \eqref{eq:choetalPsi}
in the range $\tfrac 1{q'}\le \tfrac{d+1}{d-1}\tfrac 1r-\tfrac{2}{d-1} $ and again $\tfrac{2(d+1)}{d+3} \le r<\tfrac{2(d+2)}{d+4}$.

Now assuming $q'>\frac{(d-1)p}{d+1-2p}$ we can find
$r>p$ such that $r<\frac{2(d+2)}{d+4} $ and $q'\ge \frac{(d-1)r}{d+1-2r}$
(which is equivalent with $\tfrac 1{q'}\le \tfrac{d+1}{d-1}\tfrac 1r-\tfrac{2}{d-1} $) and then (choosing $N$ large enough)
\[\sum_{\ell>0} 2^{\ell d(\frac 1p-\frac 1q)}
\|u_\delta*\widehat{\Psi_\ell}\|_{M^{\rho\to q} } \lesssim \delta^{-d(\frac 1p-\frac 12)+\frac 12}.
\]

This leads to
\[\sup_{t>0} \|\phi m_\delta (t\cdot) \|_{B^{d/p-d/q}_1(M^{r\to q} )} \lesssim \delta^{-d(\frac 1p-\frac 12) +\frac 12} \sup_{k \in {\mathbb {Z}}}|a_k|
\]
and thus Theorem \ref{thm:qle2} can be applied to complete the proof of part (i) in Theorem \ref{cor:BR}.
\qed
\section{Necessary conditions}\label{sec:necessary}
In this section we first give a proof of Proposition \ref{thm:nec-cond} and then discuss the sharpness of the results on oscillatory multipliers and the classes $\mathrm{FM}(a,b)$.

\subsection{Proof of Proposition \ref{thm:nec-cond}}

The statements about the $L^{p_1}\to L^{p_1,\infty}$ and $L^{p_2',1} \to L^{p_2'}$ operator norms are already proved in \cite[Theorem 2.5]{BRS}. Moreover by \cite[Lemma 2.4]{BRS} we have the bound
$\|T_R\|_{L^{p_1} \to L^{p_2'} } \lesssim
\|T_R\|_{{\mathrm{Sp}}(p_1,p_2)},$
with the implicit constant independent of $R$. It thus suffices to prove
\begin{equation} \label{Sparse-localized-necessity}
\|T_R\|_{{\mathrm{Sp}}(p_1,p_2)} \le (2\pi)^{-d}\|\widehat \Psi\|_1 \|T\|_{{\mathrm{Sp}}(p_1,p_2)}.
\end{equation}
We have
$
|\inn{Tf_1} {f_2}|\le \|T\|_{{\mathrm{Sp}}(p_1,p_2)}
\Lambda^*_{p_1,p_2}(f_1,f_2) $
for all $f_1,f_2 \in C^\infty_c$, and the same inequality holds with $T$ replaced by $T_R$.
To see \eqref{Sparse-localized-necessity} we write \[\Psi(R^{-1}(x-y))= (2\pi)^{-d} \int\widehat \Psi(\omega) e^{iR^{-1}\inn{\omega}x}e^{ -i R^{-1}\inn {\omega}{ y}} \, \mathrm{d}\omega\] and thus we have for $f_1, f_2\in C^\infty_c$
\[\inn{T_R f_1}{f_2}=(2\pi)^{-d}
\int \widehat \Psi(\omega)
\inn{T[f_{1,\omega}]}{f_{2,\omega}} \, \mathrm{d} \omega\] where
$f_{1,\omega}(y)= f_1(y) e^{-iR^{-1}\inn{\omega}{ y}}$ and $f_{2,\omega} (x)= f_2(x) e^{iR^{-1}\inn \omega x}$. Since $\jp{f_{i, \omega}}_{p_i}=\jp{f_i}_{p_i}$ we get
\[|\inn{T_R f_1}{f_2}| \le (2\pi)^{-d}\|\widehat\Psi\|_1 \|T\|_{{\mathrm{Sp}}(p_1,p_2)} \Lambda^*_{p_1,p_2} (f_1,f_2)
\]
which
shows
\eqref{Sparse-localized-necessity}.
\qed

\subsection{Sharpness of results on Miyachi and oscillatory multipliers}\label{sec: sharpness-Miy-osc}
Proposition \ref{thm:nec-cond} and stationary phase calculations such as in \S\ref{sec:oscmultsec}
can be used to show that the condition $(1/p_1,1/p_2) \in \Delta(a,b) $ in Theorem \ref{thm:oscmult} is necessary for $m_{a,b}(D)\in {\mathrm{Sp}}(p_1,p_2)$. See also calculations in the proof of \cite[Prop. 7.10] {BRS} and related arguments for the multiplier in \eqref{def2ofex} below.
We now construct an example completing the proof of Theorem
\ref{miyachi-thm}.

\begin{prop} Let $0<b<ad/2, a\not=1$. There is $m\in \mathrm{FM}(a,b) $ such that $m(D) \in {\mathrm{Sp}}(p_1,p_2)$ if and only if $(1/p_1, 1/p_2) \in \trapez(a,b)$.
\end{prop}

\begin{proof}
Let $\varphi_\circ\in C^\infty_c$ be supported in $\{\xi: \tfrac 34<|\xi|< \tfrac 54\}$ such that $\varphi_\circ(\xi)=1$ for $\tfrac 78<|\xi|<\tfrac 98$. Let $\eta_\circ={\mathcal {F}}^{-1}[\varphi_\circ]$.
Let ${\mathcal {N}}_1$, ${\mathcal {N}}_2$ be two infinite disjoint subsets of ${\mathbb {N}}$ such that ${\mathcal {N}}:={\mathcal {N}}_1\cup {\mathcal {N}}_2$ is well separated in the sense that
$n\ge 1+\frac{10}{|1-a|}$ and $|n-\tilde n|> 1+\frac{10}{|1-a|}$ if $n,\tilde n\in {\mathcal {N}}$ and $n\neq \tilde n$.
Define \[m=m_1+m_2,\]
where
\begin{align}\label{def1ofex}
m_1(\xi)&=\sum_{k\in {\mathcal {N}}_1} 2^{-kb} \varphi_\circ(2^{-k}\xi) e^{-i2^{-k(1-a)}\xi_1},\\ \label{def2ofex}
m_2(\xi)&= \sum_{k\in {\mathcal {N}}_2} 2^{-kb} \varphi_\circ(2^{-k}\xi)e^{ i2^{-k(2-a)}|\xi|^2/2}
\end{align}
and note that $m\in \mathrm{FM}(a,b).$
We remark that for the purpose of $L^p\to L^p$ inequalities $m_1$ behaves better than the oscillatory multipliers $m_{a,b}$, indeed if $b>0$ then $m_1(D)$ maps $L^p\to L^p$ for all $1<p<\infty$; yet $m_1(D)$ provides an example for the sharpness of the line through $P_3$ and $P_4$ in Theorem \ref{miyachi-thm}.

Let $K={\mathcal {F}}^{-1}[m]$. Assuming that $m$ belongs to $ {\mathrm{Sp}}(p,q') $ we must show that $(1/p, 1/q')$ are on or below the line connecting $P_3=(\tfrac 12, \tfrac 12+\tfrac{b}{da})$, and
$P_4=(\tfrac 12+\tfrac{b}{da}, \tfrac 12) $, that is
$1/p+1/q'\le 1+b/(da)$
or equivalently, $b\ge da(1/p-1/q)$. Moreover, we must show that $(1/p, 1/q')$ lies on or to the left of the segment $\overline {Q_1P_4}$ i.e. satisfies $b\ge da(\frac 1p-\frac 12)$.

Let $T_R^{\mathrm{resc}}$ be the convolution operator with kernel $\Psi(x) R^dK(Rx)$ then by Proposition \ref{thm:nec-cond} we get that
$T_R^{\mathrm{resc}}$
is bounded from $L^p\to L^q$ with operator norm uniformly bounded in $R$. Here we may use a suitable $\Psi\in C^\infty_c$ supported in $\{x: 1/2<|x|<2\} $ such that $\Psi(x)=1$ for $2^{-1/2}\le |x|\le 2^{1/2} $.
We shall use this for the parameters
\begin{equation}\label{choiceofR} R_n=2^{-n(1-a)}.\end{equation}
We also let $\kappa_n$ be the convolution kernel of $T_{R_n}^{\mathrm {resc}}$.
We shall show the following lower bounds:
\begin{align}\label{kappanlowerbd1}
\text{For $n\in {\mathcal {N}}_1$:} \quad \|T_{R_n}^{\mathrm {resc}} \|_{L^p\to L^q}
&\gtrsim 2^{ -n
(b-ad(\frac 1p-\frac 1q))}.
\\ \label{kappanlowerbd2} \text{For $n\in {\mathcal {N}}_2$:} \quad \|T_{R_n}^{\mathrm {resc}} \|_{L^p\to L^q}
&\gtrsim 2^{ -n
(b-ad(\frac 1p-\frac 12))}.
\end{align}
These imply after letting $n\to \infty$ within ${\mathcal {N}}_1$, ${\mathcal {N}}_2$, that the conditions
$b\ge ad(1/p-1/q)$, $b\ge ad(1/p-1/2)$, are indeed necessary for $m(D)\in {\mathrm{Sp}}(p,q')$.
Since for convolution operators the ${\mathrm{Sp}}(p_1,p_2)$ and ${\mathrm{Sp}}(p_2, p_1) $ norms coincide we get that $(1/p_1, 1/p_2) \in \trapez(a,b)$ is necessary for $m(D)$ to belong to the class ${\mathrm{Sp}}(p_1,p_2)$.
A calculation yields
\[ \kappa_n(x)= \Psi(x)
\sum_{k\in {\mathcal {N}}} K_{n,k}(x), \]
where $K_{n,k}$ is defined by the following:
\begin{align} \label{Kjn-auxiliarynew1}
&\text{ For } k\in {\mathcal {N}}_1: \,\, K_{n,k} (x)= 2^{-kb} 2^{ (k-n(1-a) )d} \eta_\circ(2^{-n(1-a)} 2^k x-2^{ka} e_1).
\\
\label{Kjn-auxiliarynew2}
&\text{ For } k\in {\mathcal {N}}_2: \,\, \widehat{ K_{n,k} }(\xi) = 2^{-kb} \varphi_\circ(2^{n(1-a)-k }\xi)
e^{i (2^{2n(1-a) -k(2-a) })|\xi|^2/2}.
\end{align}
We let $n\in {\mathcal {N}}$ and decompose
\[\kappa_n = \kappa_{n}^{\mathrm{main}}
+u_{n} +\sum_{k\neq n} \kappa_{n,k}, \]
where
\begin{subequations}
\begin{align}
\kappa_{n}^{\mathrm{main}} (x)&= K_{n,n}(x),
\\
u_{n} (x)&= (\Psi(x)-1)K_{n,n}(x),
\\
\kappa_{n,k}(x)&=
\Psi(x) K_{n,k} (x).
\end{align}
\end{subequations}

We first consider the case $n\in {\mathcal {N}}_1$ and show a lower bound for the $L^p\to L^q$ norm of the operator $T^{\mathrm{resc}}_{R_n, \mathrm{main}}$ with convolution kernel
$\kappa_{n}^{\mathrm{main}}: =K_{n,n}$.
By scaling and translation we have
\begin{equation}\label{mainkappan1} \|T^{\mathrm{resc}}_{R_n, \mathrm{main}} \|_{L^p\to L^q}
= \|\varphi_\circ\|_{M^{p\to q} } 2^{ -n
(b-ad(\frac 1p-\frac 1q))} , \quad n\in {\mathcal {N}}_1.
\end{equation}
Moreover, for $n\in {\mathcal {N}}_2$,
\begin{align*}
\|T^{\mathrm{resc}}_{R_n, \mathrm{main}} \|_{L^p\to L^q}
&=\|2^{-nb} \varphi_\circ(2^{-na}\cdot) e^{-i2^{-na}|\cdot|^2/2} \|_{M^{p\to q}}
\\&= 2^{-nb} 2^{na d(\frac 1p-\frac 1q) } \|\varphi_\circ e^{i 2^{na} |\cdot|^2/2}\|_{M^{p\to q}} .
\end{align*}
Applying the method of stationary phase we get
\[ |{\mathcal {F}}^{-1}[ \varphi_\circ e^{i 2^{na} |\cdot|^2/2}] (x) |\approx 2^{-nad/2} \quad \text{ for
$||x|-2^{na}|\le 2^{na/4}$.}
\]
Hence
\begin{align*} \|\varphi_\circ e^{i 2^{na} |\cdot|^2/2}\|_{M^{p\to q}} &\gtrsim
\Big(\int_{||x|-2^{na}|\le 2^{-na/4}} \big|{\mathcal {F}}^{-1}[ \varphi_\circ e^{i 2^{na} |\cdot|^2/2}] \big|^q \, \mathrm{d} x\Big)^{1/q} \\&\gtrsim
2^{na d(\frac 1q-\frac 12)}
\end{align*}
and combining the above we get for $q\ge p$,
\begin{equation}\label{mainkappan2} \|T^{\mathrm{resc}}_{R_n, \mathrm{main}} \|_{L^p\to L^q}
\gtrsim 2^{ -n
(b-ad(\frac 1p-\frac 12))} , \quad n\in {\mathcal {N}}_2.
\end{equation}
In order to deduce \eqref{kappanlowerbd1}, \eqref{kappanlowerbd2} from \eqref{mainkappan1}, \eqref{mainkappan2} we show error bounds for the convolution operators with kernels $u_n$ and $\Psi K_{n,k}$.

The contributions for $K_{n,k}$ are negligible for $k,n\in {\mathcal {N}}$ with $k\neq n.$
Indeed from \eqref{Kjn-auxiliarynew1} it is immediate that for $k\in {\mathcal {N}}_1$, $n\in {\mathcal {N}}$, $k\neq n$,
\[
|K_{n,k}(x)|\lesssim_N 2^{-kb} \frac{2^{ (k-n(1-a))d}}{
(1+ 2^{k-n(1-a)} |x-2^{(n-k)(1-a) }e_1|)^N}.
\]
Now consider $k\in {\mathcal {N}}_2$ and if $n\neq k$ for $n\in {\mathcal {N}}$, then
$(n-k)(1-a)\notin [-10,10].$ We have
then
\[K_{n,k}(x)=(2\pi)^{-d} 2^{-kb}\int \varphi_\circ(2^{n(1-a)-k}\xi) e^{i\phi_{n,k}(x,\xi)} \, \mathrm{d} \xi,\]
where $\phi_{n,k}(x,\xi)= 2^{2n(1-a)-k(2-a)} |\xi|^2/2-\inn x\xi$.
Compute that for $x\in {\mathrm{supp}\,} \Psi$, $|\xi|\approx 2^{k-n(1-a)}$
\[ |\nabla_\xi \phi_{n,k} (x,\xi) | \approx \begin{cases} 2^{(n-k) (1-a) } &\text{ if }(n-k)(1-a)\ge 10 \\
1 &\text{ if } (n-k)(1-a) \le -10.
\end{cases}
\]
This implies after an $N$-fold integration by parts for $|x|\approx 1$, $k\in {\mathcal {N}}_2$, $n\in {\mathcal {N}}$, $|(n-k)(1-a)|\ge 10$
\[ |K_{n,k}(x) |\lesssim_N
\begin{cases} 2^{-kb} 2^{-(n-k)(1-a) d} 2^{-ka (N-d)}
&\text{ for } (n-k)(1-a)\ge 10,
\\
2^{-kb} 2^{-ka(N-d)} 2^{(n-k)(1-a) (N-d)} &\text{ for } (n-k)(1-a)\le -10.
\end{cases}
\]
Finally, by the support properties of $(1-\Psi)$ and $\varphi_\circ$, an integration by parts also yields
\[|u_n(x) |\lesssim_N 2^{-nb+na(N-d)} (1+|x|)^{-N}\]
for all $n\in {\mathcal {N}}$.

The above estimates and the resulting consequences for upper bounds for the corresponding $L^p\to L^q$ operator norms
(obtained via Young's inequality) show that those terms are small compared to the lower bounds in \eqref{mainkappan1}, \eqref{mainkappan2} and as a consequence we obtain \eqref{kappanlowerbd1}, \eqref{kappanlowerbd2}.
\end{proof}

\section{Proofs of some auxiliary facts}
\label{sec:appendix}

\subsection{Proof of Observation \ref{obs:basicobs}}
\label{observation-proof}
We first note that the estimate \eqref{eq:main sparse triple} immediately implies the analogous estimate with $L(S_0) \geq N_2$, by writing $S_0$ as a disjoint union of cubes $Q\in {\mathfrak {Q}}(S_0)$ with $L(Q)=N_2$, applying the estimate on each such $Q$, and noting that by the disjointness of such $Q$,
\begin{equation}
{\mathfrak {S}}_{S_0}:=\bigcup_{\substack{Q \in {\mathfrak {Q}}(S_0) \\ L(Q)=N_2}} \bigcup_{\substack{{\mathfrak {S}}_{Q} \subseteq {\mathfrak {Q}}(Q) \\ {\mathfrak {S}}_Q : \gamma-\text{sparse}}} {\mathfrak {S}}_{Q}
\end{equation}
is a $\gamma$-sparse collection of cubes in ${\mathfrak {Q}}(S_0)$.

Secondly, following the argument in \cite[\S 4.2]{BRS} (based on results from
\cite{lerner-nazarov}), one can replace $\Lambda^{**}_{S_0,p,q'}$ in \eqref{eq:main sparse triple} by an actual maximal sparse form $\Lambda^*_{S_0,p,q'}$ as in \eqref{eq:maxLambda}; we omit the details.
Lastly, in view of \eqref{eq:from Tkl to Tj}, the sums $\sum_{j=N_1}^{N_2} {\mathcal {T}}_j$ in \eqref{eq:main sparse triple} can be replaced by $\sum_{k \in \digamma} \sum_{\ell=1}^N P_kT_{k}^{(\ell-k)} P_k$.

To summarize the above reductions we see that
Theorem \ref{thm:main} implies
that the inequality
\begin{equation} \label{eq:finite-simple}
\Big|\biginn{\sum_{k \in \digamma} \sum_{\ell=1}^N P_kT_{k}^{(\ell-k)} P_k f_1}{f_2}\Big| \lesssim {\mathcal {C}}\, \Lambda^*_{p,q'}(f_1,f_2)
\end{equation}
holds uniformly in $N$ and $\digamma$, for all
$C^\infty_c$ functions $f_1$, $f_2$.

We now use a limiting argument from \cite{BRS} together with Lemma \ref{lem:Kk-k} to show that \eqref{eq:finite-simple} can be upgraded to
\begin{equation}\label{eq:finalgoal}
\big|\inn{m(D) f_1}{f_2}\big| \lesssim\,{\mathcal {C}}\, \Lambda^*_{p,q'}(f_1,f_2),
\end{equation}
which in conjunction with \eqref{eq:Besovmultnorm} leads to the statements of Theorems \ref{thm:qle2}, \ref{thm:qge2} and \ref{thm:multq>2}.

To this end we use \eqref{eq:reproducing}
to decompose, for $f\in {\mathcal {S}}$,
\begin{equation} \label{eq:weak-op-conv}m(D) f = \sum_{k\in {\mathbb {Z}}} P_k m(D) L_k P_kf\end{equation} with convergence in the sense of tempered distributions.
We now apply \cite[Lemma A.1]{BRS} for the subspace ${\mathcal {V}}\equiv {\mathcal {V}}_1={\mathcal {V}}_2$ consisting of all $f\in{\mathcal {S}}$ for which $\widehat f$ is compactly supported in ${\mathbb {R}}^d\setminus\{0\}$ (and use that these are dense in $L^\rho$ for $1<\rho<\infty$).
This lemma tells us that it suffices to prove the inequality
\begin{equation} \label{eq:sparsemultiplier} |\inn {m(D)f_1}{f_2}|\lesssim {\mathcal {C}}\, \Lambda^*_{p,q'} (f_1,f_2)\end{equation} for all $f_1,f_2\in {\mathcal {V}}$.
For fixed $f_1\in {\mathcal {V}}$ the
$k$-sum in \eqref{eq:weak-op-conv} reduces to a sum over indices in a finite set $\digamma(f_1).$
It therefore suffices to prove the inequality
\begin{equation} \label{eq:sp-fin-multiplier} \Big|\biginn {\sum_{k\in \digamma} P_k m(D) L_k P_k f_1}{f_2} \Big|\lesssim {\mathcal {C}} \,\Lambda^*_{p,q'} (f_1,f_2)\end{equation}
for all finite families $\digamma\subset {\mathbb {Z}}$ and for all $f_1, f_2\in {\mathcal {V}}$. Again, by \cite[Lemma A.1]{BRS} it follows that
\eqref{eq:sp-fin-multiplier} for all $f_1, f_2\in {\mathcal {V}}$ is equivalent
to \eqref{eq:sp-fin-multiplier} for all $f_1\in \widetilde{\mathcal {V}}_1$, $f_2\in \widetilde {\mathcal {V}}_2$ where, for given $\rho$ with $p<\rho<q'$, $\widetilde{\mathcal {V}}_1$ is any dense subspace of $L^\rho$, and $\widetilde {\mathcal {V}}_2$ is any dense subspace of $L^{\rho'}$. It thus suffices to prove \eqref{eq:sp-fin-multiplier} for all $f_1, f_2\in C^\infty_c$ and all finite families $\digamma\subset{\mathbb {Z}}$.

Let $f_1, f_2\in C^\infty_c$ so that the union of the supports of $f_1$ and $f_2$ is contained in a set of diameter $R$. Let $k_{\mathrm {max}}=\max \digamma$.
Observe that for all $k\in \digamma$,
\[ \inn {T_k^{(\ell-k)} f_1}{f_2}=0 \quad \text{ if $\,\,2^{\ell-k_{\mathrm max}-3}>R\,\,$ and $\,\,\ell>0$ }.
\]
Then, we have
$\sum_{k\in \digamma} P_k m(D) L_k P_k f_1=\lim_{N\to\infty} \sum_{k \in \digamma} \sum_{\ell=0}^N P_k T_k^{(\ell-k)} P_k f_1.$
The terms for $\ell=0$ are taken care of by Lemma \ref{lem:Kk-k}; note that
by \eqref{eq:LinftyvsBesov} ${\mathcal {C}}$ can be used both in Lemma \ref{lem:Kk-k} and Theorem \ref{thm:main}. We have thus shown that \eqref{eq:sp-fin-multiplier} follows from \eqref{eq:finite-simple} for $f_1,f_2\in C^\infty_c$ and the proof is complete. \qed

\subsection{Proof of Lemma \ref{lem:Kk-k}} \label{app:Kk-k}
We are proving that for any
$1 < p \leq q < \infty$ the inequality
\begin{equation}
|\inn{\sum_{k \in \digamma} T_k^{(-k)} P_k f_1}{f_2}| \leq C \| m \|_\infty \, \Lambda_{p,q'}^*(f_1,f_2)
\end{equation} holds
uniformly in $\digamma$.

The assertion can be derived for example from
\cite[Theorem 1.1]{BRS}.
The verification of the hypotheses on that theorem is similar to the computations in \cite[\S 6.3]{BRS} and it is included for completeness. Let $m_{\digamma}=\sum_{k \in \digamma} \widehat{K}_k^{(-k)}\eta(2^{-k} \cdot)$, so that $T_{m_\digamma}=\sum_{k \in \digamma} T_k^{(-k)}P_k$. The support condition \cite[(1.6)]{BRS} clearly holds. The boundedness conditions \cite[(1.7)]{BRS}
follow
from the standard H\"ormander multiplier theorem after verifying that
\begin{equation}\label{eq:mult cond l=0 2}
\sum_{|\alpha| \leq d+1} \sup_{t>0} \sup_{\xi \in {\mathbb {R}}^d} |\partial_\xi^\alpha (\varphi m_\digamma(t \cdot))(\xi)| \lesssim \| m \|_{\infty}
\end{equation}
uniformly in $\digamma$. We need to analyze the derivatives of
\[\varphi m_\digamma(t\xi)= \sum_{\substack {k\in \digamma\\ 2^{-k}t\sim 1} } \varphi(\xi) \int \varphi (\omega) m(2^k\omega) \widehat{\Phi_0} (2^{-k}t\xi-\omega) \, \mathrm{d}\omega
\]
and \eqref{eq:mult cond l=0 2} follows after straightforward computation.

For the hypothesis \cite[(1.8)]{BRS}, we have
\begin{align*}
\| \mathrm{Dil}_{2^{-k}} T_k^{(-k)}P_k \|_{L^p \to L^q} = \| \eta \big( [\varphi m(2^k \cdot)] \ast \widehat{\Phi_0} \big) \|_{M^{p \to q}} &\\
\lesssim \| \eta \|_1 \| \Phi_0 \|_{\infty} \| \mathcal{F}^{-1}[\varphi m(2^k \cdot) ] \|_\infty \lesssim \| m \|_\infty&,
\end{align*}
for all $1 < p \leq q < \infty$, using the compact support of $\Phi_0$ and $\varphi$. Finally, the hypothesis \cite[(1.9)]{BRS} also follows from noting that
\begin{equation*}
\| \mathrm{Dil}_{2^{-k}} T_k^{(-k)}P_k \|_{L^p \to L^q} \leq \|\eta [e^{i \inn{\cdot}{h}}-1] \|_{M^p} \| [\varphi m(2^k \cdot)] \ast \widehat{\Phi_0} \|_{M^{p \to q}},
\end{equation*}
the previous bound $\| [\varphi m(2^k \cdot)] \ast \widehat{\Phi_0} \|_{M^{p \to q}} \lesssim \| m \|_\infty$ and that
\begin{equation*}
\|\eta [e^{i \inn{\cdot}{h}}-1] \|_{M^p} \leq \| \widecheck{\eta}(\cdot + h) - \widecheck{\eta} \|_1 \lesssim |h|
\end{equation*}
for any $1 < p < \infty$.
\qed

\subsection{Some embeddings for multiplier classes}
\label{sec:multiplier-embeddings}
We begin with a simple observation for compactly supported multipliers.

\begin{lemma} \label{lem:M1infty}Let $1\le p\le q\le\infty$ and $m\in M^{p\to q} $ be supported in a compact set $E$.
Then $m\in M^{1\to\infty} $ and $\|m\|_{M^{1\to\infty}}\lesssim_E \|m\|_{M^{p\to q}}$.
\end{lemma}
\begin{proof}
Let $\chi\in C^\infty_c$ be such that $\chi$ is supported on a compact subset of diameter less than twice the diameter of $E$, such that $\chi(\xi)=1$ on a neighborhood of $E$.
Since $\chi \in M^{1\to p}\cap M^{q\to \infty}$ we get
\begin{align*}
\|{\mathcal {F}}^{-1} [m\widehat f]\|_\infty&\le \|\chi\|_{M^{q\to \infty} }
\|{\mathcal {F}}^{-1} [m\widehat f]\|_q
\\
\|{\mathcal {F}}^{-1} [m\widehat f]\|_q&\le \|m\|_{M^{p\to q}} \|
\|{\mathcal {F}}^{-1} [\chi \widehat f]\|_p
\\
\|{\mathcal {F}}^{-1} [\chi\widehat f]\|_p&\le \|\chi\|_{M^{1\to p}}
\|f\|_1
\end{align*}
and putting the three inequality together we deduce the assertion.
\end{proof}
\begin{lemma}\label{lem:r1r2}
For $r_1\le r_2\le q$ let $g\in M^{1\to \infty} $ be supported in a compact set $E$. Let $\Phi$ be a Schwartz function and $\Phi_\ell(x)=\Phi(2^{-\ell} x). $ Then
\[\|g*\widehat{\Phi_\ell}\|_{M^{r_1\to q}} \lesssim
\|g*\widehat{\Phi_\ell}\|_{M^{r_2\to q}} + C_N 2^{-\ell N} \|g\|_{M^{1\to \infty}}.
\]
\end{lemma}
\begin{proof}
Let $E_\circ$ be a compact set which contains a neighborhood of $E$ and let $\chi\in C^\infty_c$ such that $\chi(\xi)=1$ for $\xi$ in a neighborhood of $E_\circ$.
Clearly $\chi\in M^{r_1\to r_2}$ and therefore
\begin{equation}\label{eq:chi}\big\|\chi (g*\widehat{\Phi_\ell})\big\|_{M^{r_1\to q}} \lesssim
\|g*\widehat{\Phi_\ell}\|_{M^{r_2\to q}}. \end{equation}
Next we will examine the multiplier
\[ (1-\chi(\xi) ) \big(g*\widehat{\Phi_\ell}(\xi)\big)= (1-\chi(\xi)) \biginn{g }{\widehat{\Phi_\ell} (\xi-\cdot)}
\]
where $\inn{g}{\cdot}$ refers to the standard pairing of a tempered distribution $g$ with a Schwartz function.
We have $M^{1\to \infty} = {\mathcal {F}} L^\infty \hookrightarrow L^2_{-N}$ with any $N>d/2$, where $\|f\|_{L^2_{-N}}^2=\int (1+|\xi|^2)^{-N}|\widehat f(\xi)|^2\, \mathrm{d} \xi$. Since $g$ is supported in an open subset of $E_\circ$ we have for $\xi\notin {\mathrm{supp}\,} (\chi)$
\begin{align*} &\Big|\partial_\xi^\gamma \big[(1-\chi(\xi))
\biginn{g }{\widehat{\Phi_\ell} (\xi-\cdot)} \big] \Big| \\
&\lesssim_N\|g\|_{L^2_{-N} }
\sum_{\beta+\beta'=\gamma} |\partial_\xi^{\beta'} (1-\chi(\xi)) |
\big\| \partial_\xi^\beta \widehat{\Phi_\ell} (\xi-\cdot) \big\|_{L^2_N (E_\circ)}
\\
&\lesssim_{N,\gamma,N_1} \|g\|_{L^2_{-N} }
2^{\ell (d+|\gamma|+N)}\sup_{\eta\in K_\circ}
|2^\ell {\mathrm{dist}} (\xi, \eta) |^{-N_1}
\\
&\lesssim_{N,|\gamma|,N_1} \|g\|_{M^{1\to \infty}} 2^{\ell (d+|\gamma|+N)-N_1}
\sup_{\eta\in K_\circ}
|{\mathrm{dist}} (\xi, \eta) |^{-N_1}.
\end{align*}
We use that ${\mathrm{dist}}({\mathrm{supp}\,}(1-\chi), E_\circ)>0$. We apply the displayed inequality for any $|\gamma|\le 2d$ and then choose $N_1> 3d+N+N_2$ to get (for all $1\le r_1\le q\le\infty$)
\begin{equation}\label{eq:1-chi} \big\|(1-\chi) \big( g*\widehat{\Phi_\ell} \big)\big\|_{M^{r_1\to q} }\lesssim_N 2^{-\ell N_2} \|g\|_{M^{1\to \infty} }.
\end{equation}
The desired estimate follows from \eqref{eq:chi} and \eqref{eq:1-chi}.
\end{proof}

\begin{cor} \label{cor:Aembeddings} Let $A_{p,r,q}^{k,\ell}$ be as in \eqref{Aprq}. For $p\le r_1\le r_2\le q$ we have
\begin{equation} \label{eq:A-embeddings-app}
A_{p,r_1,q}^{k,\ell} \lesssim A_{p,r_2,q}^{k,\ell} + C_N2^{-\ell N} \sum_{\tilde \ell\ge 0} 2^{-\tilde \ell d(\frac 1p-\frac 1q)} A_{p, r_2,q}^{k,\tilde\ell}.
\end{equation}
\end{cor}
\begin{proof}
Use Lemma \ref{lem:r1r2} with $g=\varphi m(2^k\cdot)$, expand for the error term $g=g*\widehat{\Phi_0}+\sum_{\tilde \ell>0} g*\widehat{\Psi_{\tilde\ell}}$ and invoke Lemma \ref{lem:M1infty}.
\end{proof}

\subsection{Proof of Lemma \ref{lem:studia-verification}} \label{app:hmultipliers}
Define
\begin{align*}
u_\lambda(\xi) &= \phi(\xi) \eta(\lambda\xi)
\\
\rho_{\ell,k} (\xi)&= [\phi m(2^k\cdot) ]*\widehat{\Psi_\ell} (\xi)
\\
\rho_{\ell,k,t} (\xi)& = \rho_{\ell, k} (t2^{-k} \xi)
\end{align*}
and verify that
\[\phi(\xi) h(t\xi) = \sum_{k\in \digamma} a_k u_{t2^{-k}}(\xi) \sum_{\ell\in \Lambda(k)} \rho_{\ell,k,t} (\xi).\]

We shall use that $\eta$ is a Schwartz function which vanishes at the origin and thus get the estimate
\begin{subequations}
\begin{equation}\label{eq:ulambdaptw} | {\mathcal {F}}^{-1}[u_\lambda](x)|\le C_N \begin{cases} \lambda (1+|x|)^{-N} &\text{ if }\lambda\le 1
\\
\lambda^{-N} (1+|x|)^{-N} &\text{ if }\lambda\ge 1
\end{cases}
\end{equation} for all $N$. Consequently,
\begin{equation}\label{eq:ulambdaL1}
\|{\mathcal {F}}^{-1} [u_{\lambda}]\|_1\lesssim_N \min \{ \lambda, \lambda^{-N} \}.
\end{equation}
\end{subequations}

For any $t>0$ we have
\begin{multline}\label{eq:phihcdot}
\|\phi h(t\cdot) \|_{B^{d(\frac 1{q'}-\frac 1r)}_1(M^{r\to r} )} \\
\le \sum_{n\ge 0} 2^{nd(\frac{1}{q'}-\frac 1r)} \sum_{k\in \digamma} \sum_{\ell\in \Lambda(k) } \| [u_{t2^{-k}} \rho_{\ell,k,t} ] *\widehat{\Psi_n}\|_{M^{r\to r}} .
\end{multline}
We split the sets $\Lambda(k)= \Lambda^*(k,t,n) \cup \Lambda_*(k,t,n) $ where
\begin{align*}
\Lambda^*(k,t,n) &=\big \{\ell \in \Lambda(k): 2^{\ell}\ge 2^{n-5} \min\{1, 2^k t^{-1}\}
\big\}
\\
\Lambda_*(k,t,n) &=\Lambda(k)\setminus \Lambda^*(k,t,n).
\end{align*}

We first argue that in \eqref{eq:phihcdot} the terms with $\ell \in \Lambda_*(k,t,n)$
are negligible. For $\ell\in \Lambda_*(k,t,n)$ we have
$2^n\ge 2^{\ell+5} \max\{ 1, 2^{-k} t\} $.
We use crude estimates for ${\mathcal {F}}^{-1} [\rho_{\ell,k,t}]$ and take advantage of support properties.
Write $\rho_{\ell,k} (\xi) =\int \widehat{\Psi_\ell}(\xi-\omega) \phi(\omega)m(2^k\omega) \, \mathrm{d} \omega$ and by a $(d+1)$-fold integration by parts we get
\[ |{\mathcal {F}}^{-1} [\rho_{\ell, k} ](w) |\le 2^{\ell(d+1)} (1+|w|)^{-d-1} \|\phi m(2^k\cdot)\|_\infty, \] and we have ${\mathcal {F}}^{-1}[\rho_{\ell,k}](w)=0$ for $|w|>2^{\ell-1}$. Hence we obtain
\begin{align*}\label{eq:smallellpointwise}
& |{\mathcal {F}}^{-1} [(u_{t2^{-k}} \rho_{\ell,k,t} )*\widehat{ \Psi_n} ](x) |
\\ \notag
&\le |\Psi_n(x)| \int|{\mathcal {F}}^{-1} [u_{t2^{-k} } ](y)| (2^k t^{-1} )^d |{\mathcal {F}}^{-1}[\rho_{\ell,k} ] (2^k t^{-1} (x-y) ) | \, \mathrm{d} y
\\
\notag&\le |\Psi_n(x) |\int_{|x-y|\le 2^{-k} t 2^{\ell-1}} \frac{\min\{t2^{-k} , (t2^{-k} )^{-N}\}}
{(1+|y|)^N} \frac{2^{\ell(d+1)} (2^{k}t^{-1})^d}{(1+ 2^k t^{-1} |x-y|)^{d+1} }\, \mathrm{d} y.
\end{align*}
We invoke the condition
$2^{\ell}\le 2^{n-5} \min\{1, 2^k t^{-1}\}$ for $\ell\in \Lambda_*(k,t,n)$.
For $x\in {\mathrm{supp}\,} \Psi_n$ we have $|x|\ge 2^{n-3} $. Thus in the above integral we can use
\[|y|\ge |x|-|x-y|\ge 2^{n-3} - 2^{-k} t2^{\ell-1} \ge 2^{n-3}-
2^{n-6} \, 2^{-k } t \min\{1, 2^k t^{-1} \} \ge 2^{n-4}\] and hence
\begin{align*} &\|{\mathcal {F}}^{-1} [(u_{t2^{-k}} \rho_{\ell,k,t} )*\widehat{\Psi_n}] \|_1 \\ &\lesssim 2^{\ell(d+1)}
\int_{|y|\ge 2^{n-4} } \frac{\min\{t2^{-k}, (t2^{-k} )^{-N}\}}
{(1+|y|)^N}
\int \frac{(2^{k}t^{-1})^d}{(1+ 2^k t^{-1} |x-y|)^{d+1} } \, \mathrm{d} x
\, \mathrm{d} y
\\
&\lesssim 2^{\ell (d+1)} 2^{-n (N-d)} \min\{t2^{-k}, (t2^{-k} )^{-N}\}.
\end{align*}
Consequently, using
$\| [u_{t2^{-k}} \rho_{\ell,k,t} ] *\widehat{ \Psi_n}\|_{M^{r\to r}} \le
\|{\mathcal {F}}^{-1} [(u_{t2^{-k}} \rho_{\ell,k,t} )*\widehat{\Psi_n}] \|_1$ in \eqref{eq:phihcdot} we get (assuming $N>3d$)
\begin{align} \label{Lambdalowerstar}
&\sum_{n\ge 0} 2^{nd(\frac{1}{q'}-\frac 1r)} \sum_{k\in \digamma} \sum_{\ell\in \Lambda_*(k,t,n) } \| [u_{t2^{-k}} \rho_{\ell,k,t} ] *\widehat{\Psi_n}\|_{M^{r\to r}}
\\ \notag
&\lesssim
\sum_{n\ge 0} 2^{nd(\frac{1}{q'}-\frac 1r)} 2^{-n(N-d)} \sum_{k}
\min\{t2^{-k}, (t2^{-k} )^{-N}\} \sum_{\ell\le{n-5} }2^{\ell(d+1)}
\\ \notag
&\lesssim\sum_{n\ge 0} 2^{n(3d-N) } \sum_{k}
\min\{t2^{-k}, (t2^{-k} )^{-N}\} \lesssim 1.
\end{align}

We now turn to the main terms with $\ell \in \Lambda^*(k,t,n)$ in \eqref{eq:phihcdot}, i.e. the terms with $2^n< 2^{\ell+5} \max\{ 1, 2^{-k} t\} $.
Notice that
\begin{align*} &\| [u_{t2^{-k}} \rho_{\ell,k,t} ] *\widehat{\Psi_n}\|_{M^{r\to r}}
\lesssim \|u_{t2^{-k}} \|_1\| \rho_{\ell,k,t} \|_{M^{r\to r}}
= \|u_{t2^{-k}} \|_1 \|\rho_{\ell,k} \|_{M^{r\to r}}
\\&
\lesssim
\min \{ t2^{-k}, (t 2^{-k})^{-N} \} A_{q',r,r}^{k,\ell} 2^{-\ell d(\frac 1{q'}-\frac 1r)} .
\end{align*}
Hence in \eqref{eq:phihcdot} we can estimate the terms with $\ell\in \Lambda^*(k,t,n)$ as
\begin{align*}
&\sum_{n\ge 0} 2^{nd(\frac{1}{q'}-\frac 1r)} \sum_{k\in \digamma} \sum_{\ell\in \Lambda^*(k,t) } \| [u_{t2^{-k}} \rho_{\ell,k,t} ] *\widehat{\Psi_n}\|_{M^{r\to r}}
\\
&\lesssim \sum_{\substack{k\in \digamma\\ 2^{-k} t\le 1}} 2^{-k} t \sum_{\ell>0} A_{q',r,r}^{k,\ell} 2^{-\ell d(\frac 1{q'}-\frac 1r)}
\sum_{ \substack{1\le 2^n\le 2^{\ell+5} }} 2^{nd(\frac{1}{q'}-\frac 1r)}
\\&\quad +
\sum_{\substack{k\in \digamma\\ 2^{-k}t >1}} (2^{-k} t)^{-N} \sum_{\ell\ge 0} A_{q',r,r}^{k,\ell} 2^{-\ell d(\frac 1{q'}-\frac 1r)}
\sum_{ \substack{1\le 2^n\\ \le 2^{\ell+2} 2^{-k} t }} 2^{nd(\frac{1}{q'}-\frac 1r)}
\\
&\lesssim \sup_k\sum_{\ell\ge 0}
A_{q',r,r}^{k,\ell}.
\end{align*}
This finishes the proof of \eqref{mult-verification} and thus the proof of the lemma. \qed

\subsection{An elementary lemma} \label{sec:elementary-lemma}

The following elementary lemma is repeatedly used in the induction step for constructing sparse families of cubes.

\begin{lem}\label{lemma:lift lp to lq}
Let ${\mathcal {Q}}$ be a family of cubes with bounded overlap and let $\{f_Q\}_{Q \in {\mathcal {Q}}}$ be a family of functions such that ${\mathrm{supp}\,} f_Q \subseteq Q$. Then, for all $1 \leq p \leq q < \infty$,
\begin{equation*}
\Big\| \sum_{Q \in {\mathcal {Q}}} f_Q \Big\|_p \lesssim \big(\sum_{Q \in {\mathcal {Q}}} |Q| \big)^{\frac{1}{p}-\frac{1}{q}} \Big( \sum_{Q \in {\mathcal {Q}}} |Q|^{1-\frac qp} \|f_Q\|_{p}^q \Big)^{1/q}.
\end{equation*}
\end{lem}

\begin{proof}
By assumption, there is a constant $C$ such that every $x$ is contained in at most $C$ of the cubes in ${\mathcal {Q}}$. We may split ${\mathcal {Q}}$ into $O(C)$ disjoint families ${\mathcal {Q}}_\nu$ and it suffices to prove the inequality for each ${\mathcal {Q}}_\nu$.
From H\"older's inequality,
\begin{align*}
\Big\| \sum_{Q \in {\mathcal {Q}}_\nu} f_{Q} \Big\|_p &=\Big (\sum_{Q \in {\mathcal {Q}}_\nu}
\| f_{Q} \|_p^p \Big )^{1/p}
= \Big(\sum_{Q\in {\mathcal {Q}}_\nu}
|Q|^{1-p/q} |Q|^{p/q-1} \|f_Q\|_{p}^p \Big)^{1/p}
\\&\le
\Big(\sum_{Q\in {\mathcal {Q}}_\nu} |Q|\Big )^{1/p-1/q} \Big (\sum_{Q\in {\mathcal {Q}}_\nu} |Q|^{1-q/p} \|f_Q\|_{p}^q \Big)^{1/q} . \qedhere
\end{align*}
\end{proof}

\def\MR#1{}

\end{document}